\newif\ifndjflsub

\ifndjflsub
\documentclass{ndjflart}
\else
\documentclass{article}
\usepackage[margin=3.5cm]{geometry}
\newenvironment{acks}{\subsection*{Acknowledgements}}{}
\fi

 \usepackage[T1]{fontenc}
 \usepackage{bbold}
\usepackage{tgtermes}
\usepackage[scaled=.92]{helvet}
\ifndjflsub \setoptfont{enc={T1},fam={pop}} \fi 

\usepackage{amsthm,amsmath,amssymb}
\usepackage[shortlabels]{enumitem}
\usepackage{enumitem}
\usepackage{makebox}
\usepackage{mathrsfs}
\usepackage[numbers]{natbib}  
\usepackage[colorlinks,citecolor=blue,urlcolor=blue,breaklinks]{hyperref}  
\usepackage[capitalise]{cleveref}
\usepackage{etoolbox,mathtools}

\usepackage[nottoc,notlot,notlof]{tocbibind}
\usepackage{graphicx,array}
\usepackage[matrix,arrow,curve,rotate]{xy}

\ifndjflsub \startlocaldefs \fi
\ifndjflsub \artstatus{am} \fi 
\ifndjflsub  \fi

\Crefname{propenumi}{Proposition}{Propositions}
\AtBeginEnvironment{proposition}{%
    \crefalias{enumi}{proposition}%
    \setlist[enumerate,1]{
        label={\textit{(\alph*)}},
        ref={\alph*}
    }%
  }

\AtBeginEnvironment{proof}{%
    \setlist[enumerate,1]{
        label={\textit{(\alph*)}},
        ref={\alph*}
    }%
  }

\AtBeginEnvironment{theorem}{%
    \crefalias{enumi}{theorem}%
    \setlist[enumerate,1]{
      label={\textit{(\alph*)}},
        ref={\alph*}
    }%
}
  
\Crefname{defenumi}{Definition}{definitions}
\AtBeginEnvironment{definition}{%
    \crefalias{enumi}{definition}%
    \setlist[enumerate,1]{
        label={\textit{(\alph*)}},
        ref={\alph*}
    }%
  }

\Crefname{corenumi}{Corollary}{corollaries}
\AtBeginEnvironment{corollary}{%
    \crefalias{enumi}{corollary}%
    \setlist[enumerate,1]{
        label={\textit{(\alph*)}},
        ref={\alph*}
    }%
  }

\newtheorem{theorem}{Theorem}[section]
\newtheorem{proposition}[theorem]{Proposition}

\newtheorem{corollary}[theorem]{Corollary} 
\theoremstyle{definition}
\newtheorem{definition}[theorem]{Definition}

\theoremstyle{remark}
\newcommand{\wellfound}{well-founded}
\newcommand{\biginfty}{\infty\hspace{-0.95em} \infty}
\newcommand{\subbiginfty}{\infty\hspace{-0.825em} \infty}

\newcommand{\subbigginfty}{\subbiginfty}
\newcommand{\ddset}{\mathop{\rotatebox[origin=c]{270}{\(\twoheadrightarrow\)}}}
\newcommand{\waybel}[1]{\ddset{#1}}
\newcommand{\subwaybel}[1]{\text{\raisebox{-0.3ex}{$\ddset$}}{#1}}
\newcommand{\graphof}[1]{\mathsf{graph}({#1})}

\newtheorem{example}[theorem]{Example}
\newcommand{\wiscres}[2]{{#1}}
\newcommand{\gamdltof}[2]{\Gamma_{({#1},{#2})}}
\newcommand{\gamdlt}{\gamdltof{D}{<}}
\newcommand{\delwlof}[3]{\Delta_{({#1},{#2})}^{#3}}
\newcommand{\delwl}{\delwlof{D}{<}{L}}
\newcommand{\rankf}{\rho}
\newcommand{\ordtype}[1]{\text{order-type}{#1}}
\newcommand{\full}{Full}
\newcommand{\jlt}{J}
\newcommand{\sbwf}{\wellfound{} set-based}

\newcommand{\Qwide}{Quasiwide}
\newcommand{\qwide}{quasiwide}

\newcommand{\jltstar}{J^{*}}
\newcommand{\extord}{\Ord_{\infty}}
\newcommand{\copair}[2]{\mathsf{copair}({#1},{#2})}

\newcommand{\jltplus}{J^{+}}
\newcommand{\bbI}{\mathbb{M}}
\newcommand{\dof}[2]{D_{{#2},{#1}}}
\newcommand{\fund}[1]{{#1}_{\mathsf{small}}}

\newcommand{\alestar}{\aleph^{*}}
\newcommand{\ob}{\mathsf{ob}\ }
\newcommand{\cfof}[1]{\mathsf{cf}({#1})}
\newcommand{\ppart}[1]{\part{#1}}
\newcommand{\psection}[1]{\section{#1}}
\newcommand{\psubsection}[1]{\subsection{#1}}
\newcommand{\standcoll}{standard ordered collection}

\newcommand{\iof}[2]{i_{{#1},{#2}}}
\newcommand{\algof}[1]{R({#1})}

\newcommand{\makelitfam}[1]{\overline{#1}}
\newcommand{\makefam}[1]{\overline{#1}}
\newcommand{\makerub}[1]{\dot{#1}}
\newcommand{\compsof}[1]{\chi({#1})}

\newcommand{\arsof}[1]{\mathsf{Arit}({#1})}
\newcommand{\setb}{set-based}
\newcommand{\fullsetb}{iteratively set-based}
\newcommand{\Sset}{Set}

\newcommand{\setgen}{\Sset{} Generation}
\newcommand{\classchoice}{Collective Choice}

\newcommand{\cwise}[1]{\widehat{#1}}
\newcommand{\cowise}[1]{\widetilde{#1}}

\newcommand{\jspl}{$^{*}$}
\newcommand{\jwew}{$^{\dag}$}
\newcommand{\spl}{$\,^{*}$}
\newcommand{\wew}{$\,^{\dag}$}
\newcommand{\ptset}{\text{Powerset}}

\newcommand{\rnkof}[2]{\overline{#1}}

\newcommand{\upto}[1]{\overline{#1}}

\newcommand{\predof}[1]{\downarrow\! {#1}}

\newcommand{\limso}{\mathsf{LimPt}}
\newcommand{\limsof}[1]{\limso({#1})}

\newcommand{\prefof}[1]{\mathsf{Pref}({#1})}

\newcommand{\ovresof}[1]{\mathsf{O}_{\catr}({#1})}

\newcommand{\ovreso}{\mathsf{O}_{\catr}}

\newcommand{\classfamof}[1]{\mathsf{ClassFam}({#1})}
\newcommand{\classfamcat}[1]{\mathbf{ClassFam}({#1})}

\newcommand{\injcfof}[1]{\mathsf{InjClassFam}({#1})}
\newcommand{\injfof}[1]{\mathsf{InjFam}({#1})}

\newcommand{\bfclass}{\mathbf{Class}}

\newcommand{\classfamoo}[2]{\partialprod{#2}{#1}}
\newcommand{\funcok}{\acute{K}}
\newcommand{\subclassof}[1]{\mathsf{Sub}({#1})}
\newcommand{\class}{\mathsf{Class}}

\newcommand{\mayg}{}
\newcommand{\pmiplong}[1]{\makebox*{\ptset $+$ Wide \Supgeneration{}}[l]{${#1}$}}
\newcommand{\pmipblong}[1]{\makebox*{\ptset $+$ Broad \Supgeneration{}}[l]{${#1}$}}

\newcommand{\twodots}{\mathinner {\ldotp \ldotp}}
\newcommand{\id}{\mathsf{id}}
\newcommand{\gamfamcatr}{\Delta_{\catr}}

\newcommand{\incs}{\mathsf{Maybe}}
\newcommand{\incof}[1]{\mathsf{Maybe}\,{#1}}
\newcommand{\const}[1]{\mathsf{Const}_{#1}}
\newcommand{\constalpha}{\const{\alpha}}
\newcommand{\constk}{\const{\catk}}
\newcommand{\constl}{\const{\catl}}
\newcommand{\regord}{\mathsf{Reg}}

\newcommand{\limclass}{\mathsf{Lim}}
\newcommand{\ssup}{\mathsf{ssup}}
\newcommand{\supo}{\bigvee}
\newcommand{\supclosed}{supclosed}
\newcommand{\supcomplete}{supcomplete}
\newcommand{\Supgeneration}{Supgeneration}
\newcommand{\supgenerat}{supgenerat}
\newcommand{\noidea}{}

\newcommand{\symb}{\mathsf{Pred}}

\newcommand{\psetinh}{\mathcal{P}_{\mathsf{inh}}}

\newcommand{\ccolc}{\! : \!}

\newcommand{\derivsof}[1]{\mathsf{Deriv}_{#1}}

\newcommand{\assumi}[1]{\emph{(Assuming {#1}.)}}

\newcommand{\injcomp}{\preccurlyeq}
\newcommand{\surjcomp}{\preccurlyeq^{*}\!}

\newcommand{\inone}{}

\newcommand{\explind}[1]{}
\newcommand{\tzero}{\mathrm{zero}}

\newcommand{\tnat}{\mathrm{nat}}
\newcommand{\tpow}[1]{\mathrm{pow}({#1})}

\newcommand{\tsigma}[2]{\mathrm{sigma}({#1},{#2})}

\newcommand{\allrub}[1]{\mathsf{WideRub}({#1})}
\newcommand{\allbroadrub}[1]{\mathsf{BroadRub}({#1})}
\newcommand{\widerule}[1]{\mathsf{WideRule}({#1})}
\newcommand{\broadrule}[1]{\mathsf{BroadRule}({#1})}

\newcommand{\sig}{\famof{\allsets}}

\newcommand{\wfimpure}{V_{\mathsf{impure}}}
\newcommand{\wfpure}{V_{\mathsf{pure}}}

\newcommand{\Ord}{\mathsf{Ord}}

\newcommand{\Start}{\rStart}
\newcommand{\yBuild}[3]{\rBuild{#1}{\tuple{{#2},{#3}}}}

\newcommand{\elset}[1]{\mathcal{E}({#1})}
\newcommand{\preBuild}[2]{\rBuild{#1}{#2}}
\newcommand{\Succm}[1]{{#1}}

\newcommand {\cate}{\mathcal{E}}

\newcommand{\catb}{\mathcal{B}}
\newcommand{\catd}{\mathcal{D}}

\newcommand{\bbegin}[3]{Narrow}

\newcommand{\ccont}[4]{}
\newcommand{\Zero}{\mathsf{Nothing}}
\newcommand{\Succ}[1]{\mathsf{Just}{#1}}
\newcommand{\Succj}{\mathsf{Just}}

\newcommand{\succof}[1]{\mathsf{S}{#1}}
\newcommand{\Succb}[1]{\mathsf{Just}({#1})}

\newcommand{\cats}{\mathcal{S}}
\newcommand{\set}{\mathbf{Set}}

\newcommand{\famof}[1]{\mathsf{Fam}({#1})}

\newcommand{\catk}{\mathcal{K}}
\newcommand{\catl}{\mathcal{L}}

\newcommand{\catr}{\mathcal{R}}

\newcommand{\mus}{\wide{S}}
\newcommand{\inck}{\mathsf{Maybe}^{\circ}_{K}}
\newcommand{\incf}{\mathsf{Maybe}^{\circ}_{F}}
\newcommand{\incfprime}{\mathsf{Maybe}^{\circ}_{F'}}
\newcommand{\incg}{\mathsf{Maybe}_{G}}
\newcommand{\ovs}{\overline{S}}
\newcommand{\redbroad}[1]{\mathsf{SimpleBroad}({#1})}
\newcommand{\broad}[1]{\mathsf{Broad}({#1})}
\newcommand{\redwide}[1]{\mathsf{SimpleWide}({#1})}
\newcommand{\wide}[1]{\mathsf{Wide}({#1})}

\newcommand{\op}{\mathsf{op}}
\newcommand{\rStart}{\mathsf{Begin}}
\newcommand{\rBuild}[2]{\mathsf{Make}({#1},{#2})}
\newcommand{\summay}{\tesum^{\mathsf{Maybe}}}

\newcommand{\psetset}{\pset}

\newcommand{\transset}{\mathsf{TrSet}}

\newcommand{\xff}[1]{\xf{#1}}
\newcommand{\xf}[1]{\mu^{{#1}}h}

\newcommand{\memreach}[1]{\mathsf{\mathcal{E}}^{*}({#1})}

\newcommand{\memplus}[1]{\mathsf{\mathcal{E}}^{+}({#1})}

\newcommand{\pset}{\mathcal{P}}
\newcommand{\eqdef}{\stackrel{\mbox{\rm {\tiny def}}}{=}}
\newcommand{\iffdef}{\stackrel{\mbox{\rm {\tiny def}}}{\iff}}
\newcommand{\cata}{\ensuremath{\mathcal{A}}}
\newcommand{\catc}{\ensuremath{\mathcal{C}}}
\newenvironment{spaceout}[1]{\begin{displaymath}\setlength{\extrarowheight}{3pt}\begin{array}{#1}}{\end{array}\setlength{\extrarowheight}{0pt}\end{displaymath} \noindent}
\newcommand{\univv}{\mathfrak{U}}

\newcommand{\domof}[1]{\mathsf{Dom}({#1})}
\newcommand{\rangeof}[1]{\mathsf{Range}({#1})}

\newcommand{\nats}{\mathbb{N}}

\newcommand{\setbr}[1]{\{{#1}\}}
\newcommand{\tuple}[1]{\langle {#1} \rangle}

\newcommand{\itinl}{\mathsf{inl}\;}

\newcommand{\itinr}{\mathsf{inr}\;}

\newcommand{\totall}{\mathfrak{T}}
\newcommand{\allsets}{\mathsf{Set}}

\newcommand{\isset}[1]{\mathsf{IsSet}({#1})}

\newcommand{\partialprod}[1]{\prod_{{#1}}^{\hspace{-0.73em}\text{\raisebox{0.15ex}{$\rightharpoonup$}}}}

\newcommand{\qqed}{}

\newcommand{\smin}{\!\in\!}

\newcommand{\bnfgo}{::=\hspace{0.7em}}
\newcommand{\bnf}{\hspace{0.5em} | \hspace{0.5em}}

\newcommand{\tesum}{\textstyle{\sum}}

\newcommand{\existsu}{\exists!}

\ifndjflsub \endlocaldefs \fi

\begin{document}


\ifndjflsub \begin{frontmatter} 
 \title{\emph{Broad Infinity and Generation Principles}} 
 \runtitle{Broad Infinity and Generation Principles}
  \author{\fnms{Paul Blain}
    \snm{Levy}
    \corref{}
    \ead[label=e1]{P.B.Levy@cs.bham.ac.uk}
    \ead[label=u1,url]{http://www.cs.bham.ac.uk/\textasciitilde pbl}
  }
  \address{School of Computer Science\\
    University of Birmingham\\
    Edgbaston\\
    Birmingham\\
    B15 2TT \\
    UNITED KINGDOM \\
    \printead{e1} \\ \printead{u1}}

  \runauthor{P.~B.~Levy}
  
  \else
  \title{Broad Infinity and Generation Principles}
  \author{Paul Blain Levy, University of Birmingham}
  \date{}
  \fi


\bibliographystyle{alpha}
\maketitle

\begin{abstract}
  We introduce Broad Infinity, a new set-theoretic axiom scheme based
  on the slogan ``Every time we construct a new element, we gain a new
  arity.''  It says that three-dimensional trees whose growth is
  controlled by a specified class function form a set.  Such trees are
  called ``broad numbers''.

Assuming AC (the axiom of choice) or at least the weak version known
as WISC (Weakly Initial Set of Covers), we show that Broad
Infinity is equivalent to
  Mahlo's principle, which says that the class of all regular limit ordinals is stationary.
  Assuming AC or WISC, Broad Infinity also yields a convenient principle for
generating a subset of a class using  a ``rubric'' (family of rules); this
directly gives the existence of Grothendieck universes, without
requiring a detour via ordinals.

In the absence of choice, Broad Infinity implies that the derivations
of elements from a rubric form a set; this yields the existence of Tarski-style universes.

Additionally, we reveal a pattern of resemblance between ``Wide'' principles, that are provable in ZFC, and ``Broad''
principles, that go beyond ZFC.

Note: this paper uses a base theory that is weaker than ZF but includes
classical first-order logic and Replacement.
\end{abstract}

\ifndjflsub
\begin{keyword}[class=AMS]
  \kwd[Primary ]{03E10} \kwd{03E25} \kwd[; Secondary ]{03E70} \kwd{03E55}
\end{keyword}

\begin{keyword}
 \kwd{Set theory}  \kwd{Broad Infinity} \kwd{Generation
    Principles}  \kwd{Axiom of Choice}  \kwd{Grothendieck universe}
  \kwd{Tarski-style universe} \kwd{Inaccessible cardinal} \kwd{WISC}
  \kwd{Ordinal} \kwd{Inductive chain} \kwd{Scaffold} \kwd{Rubric} 
  \kwd{Mahlo's principle} \kwd{Lindenbaum number} 
\end{keyword}
\end{frontmatter}
\fi

\tableofcontents

\ppart{Introduction} 

\psection{Broad Infinity vs Mahlo's Principle}  \label{sect:broadvsmahlo}

 \psubsection{Broad Infinity in a Nutshell} \label{sect:nutshell}

 This paper is about a new axiom scheme of set theory, which is easy
 to state.

 First, some preliminaries.  
 For the sake of this introduction, assume either ZF or a variant  
  that allows urelements.  
 We write $\totall$ for the universal class and $\allsets$ for the
 class of all sets; they are the same in ZF. The axiom of choice (AC) is
 not assumed.     

 We must say how to encode ordered pairs and
the like.
\begin{definition}  \label{def:encode}
  Let $C$ be a class.
  \begin{enumerate}
  \item \label{item:encodeop} An  \emph{ordered pair encoding} on $C$ is a binary operation
    $\tuple{-,-} \ccolc C^2 \rightarrow C$ such that, for all $x,y,x',y' \smin C$, if
    $\tuple{x,y} = \tuple{x',y'}$, then $x=x'$ and $y=y'$.
  \item  \label{item:encodeunary} A \emph{unary Dedekind encoding} on $C$
    consists of an element $\Zero \smin C$ and a unary operation
    $\mathsf{Just} \ccolc C \rightarrow C$, such that
    \begin{itemize}
    \item for all $x \smin C$, we have $\Succb{x} \not= \Zero$
     \item  for all $x,x' \smin C$, if $\Succb{x} = \Succb{x'}$, then $x=x'$.
     \end{itemize}
   \item \label{item:encodebinary} A \emph{binary Dedekind encoding} on $C$
    consists of an element $\rStart \smin C$ and a binary operation
    $\mathsf{Make} \ccolc C^2 \rightarrow C$, such that
    \begin{itemize}
    \item for all $x,y \smin C$, we have $\rBuild{x}{y} \not= \rStart$
     \item  for all $x,y,x',y' \smin C$, if $\rBuild{x}{y} =
       \rBuild{x'}{y'}$, then $x=x'$ and $y=y'$.
     \end{itemize}
  \end{enumerate}
\end{definition}
The following encodings are fixed throughout the paper.
 \begin{definition} \hfill \label{def:specificenc}
  \begin{enumerate}
  \item \label{item:specifop} We give an ordered pair encoding on $\totall$ as follows:
    \begin{eqnarray*}
    \tuple{x,y} & \eqdef & \setbr{\setbr{x},\setbr{x,y}}
    \end{eqnarray*}
  \item \label{item:specifunary} We give a unary Dedekind encoding on $\totall$ as follows:
    \begin{eqnarray*}
    \Zero & \eqdef & \emptyset \\ \Succb{x} & \eqdef & \setbr{x}
    \end{eqnarray*}
  \item \label{item:specifbinary} We give a binary Dedekind encoding on $\totall$ as follows:
    \begin{eqnarray*}
    \rStart & \eqdef & \emptyset \\
    \rBuild{x}{y} & \eqdef & \setbr{\setbr{x},\setbr{x,y}}
    \end{eqnarray*}
  \end{enumerate}
\end{definition}

For a class $C$ and set $K$, we write $C^{K}$ for the class of all
functions from $K$ to  $C$. 

The axiom of 
\emph{Infinity} is included
in ZF.  As formulated by Zermelo~\cite{Zermelo:natnumber}, it says that
there is a set $X$ with the following properties:
\begin{itemize}
\item $\Zero \in X$.
\item For any $x \smin X$, we have $\Succb{x} \in X$. 
\end{itemize}
The new axiom scheme of \emph{Simple Broad Infinity} is similar.  It says that, for any
function $F \ccolc \totall \rightarrow \allsets$, there is a set $X$
with the following properties:
\begin{itemize}
\item $\rStart \in X$.
\item For any $x \smin X$ and $y \smin X^{Fx}$,  we have $\rBuild{x}{y}\in X$.
\end{itemize}
Here is a slogan: ``Every time we construct a new element, we gain a new arity.''

ZF extended with this scheme is called \emph{Broad ZF}.  
 The following sections will motivate  
 this extension in light of a previously studied principle.

\psubsection{Regular Limits and Stationary Classes} \label{sect:regstat}

 We begin with some useful notions concerning ordinals.  We write
 $\Ord$ for the class of all
 ordinals and $\limclass$ for the class of all \emph{limit ordinals}---ordinals that are
 neither 0 nor a successor.  An \emph{initial} ordinal is one that is
 not the range of a function from a smaller ordinal; examples are the
 finite ordinals, $\omega$ and $\omega_1$.    (In ZFC, an initial
 ordinal is also called a ``cardinal''.)


 A limit ordinal $\kappa$ 
 is \emph{regular} when, for all $\alpha < \kappa$,
the supremum function
  $\Ord^{\alpha} \rightarrow \Ord$ restricts to a
  function $\kappa^{\alpha} \rightarrow \kappa$.  (See \Cref{sect:regularity} for an alternative definition.)  It follows that
  $\kappa$ is initial, so $\omega$ is the only regular limit ordinal
  that is countable.   We write $\regord$ for the class of all regular
  limit ordinals.

For a function  $F \ccolc \Ord \rightarrow \Ord$, we say that an
limit ordinal $\lambda$ 
is  \emph{$F$-closed} when $F$ restricts to a function 
$\lambda \rightarrow \lambda$.  Here are some examples:
\begin{itemize}
\item Let $\mathsf{S}$ be the successor function.  Every limit ordinal
  is 
  $\mathsf{S}$-closed.
\item For an ordinal $\alpha$, let $\constalpha$ be the constant
  function $\gamma \mapsto \alpha$.  A limit ordinal is
   $\constalpha$-closed iff it is $> \alpha$.
 \item For functions $F,G \ccolc \Ord \rightarrow \Ord$, let $F \vee
   G$ be the pointwise maximum $\gamma \mapsto F(\gamma) \vee
   G(\gamma)$.  A limit ordinal is $(F\vee G)$-closed iff it is 
  both  $F$-closed and $G$-closed.  
\end{itemize}
 A class  of limit ordinals $D$ is \emph{stationary} when, for 
every function $F \ccolc \Ord \rightarrow \Ord$, there is an  
 an 
 $F$-closed member of $D$. (See
 \Cref{sect:unbstat,sect:powstat} for alternative definitions.)  It
 follows that $D$ is unbounded and that, for
 every function  $F \ccolc \Ord \rightarrow \Ord$, there are 
 stationarily many $F$-closed members of $D$.  

\psubsection{Two Principles from the Literature}  \label{sect:twoprinc}

Next we look at two principles that use the above notions.
\begin{itemize}
\item \emph{Mahlo's principle}, also known as ``Ord is Mahlo'', says that $\regord$
  is stationary~\cite{Hamkins:simplemax,Jorgensen:generate,Levy:stronginf,Mayberry:book,Wang:largesets}.  To illustrate its
  power, note that ZFC $+$ Mahlo's principle
  proves that there are stationarily many inaccessible
  cardinals.\footnote{In the absence of AC, there is no
  accepted notion of
  inaccessible.  See~\cite{BlassDimitriouLoewe:inaccnoac} for
  a comparative analysis.} 
  That is because, in ZFC, an inaccessible cardinal is precisely  an 
    uncountable   $F$-closed limit ordinal, where $F$ sends $\alpha$ to $2^{\alpha}$ if
  $\alpha$ is a cardinal and to 0 otherwise.
\item \emph{Blass's axiom}~\cite{Blass:initalg} says merely that $\regord$ is
  unbounded.  It follows from AC, but
  is it provable in ZF alone?  To answer this question, Gitik~\cite{gitik:uncsing}
  showed that, if ZFC $+$ ``Arbitrarily large 
  strongly compact cardinals exist'' is consistent, then so is ZF $+$
  ``Every limit ordinal is the supremum of a strictly
  increasing $\omega$-sequence''.  This
  means that ZF
  cannot even prove the existence of an uncountable regular limit ordinal, let alone prove Blass's axiom.
\end{itemize}

\psubsection{Limitations of Mahlo's Principle} \label{sect:limitmah}

Appealing though Mahlo's principle may be, I consider it 
deficient as an axiom scheme, in two respects.

 Firstly, it does not meet the ZF standard of simplicity.  Each ZF
 axiom, other than Extensionality and Foundation, expresses the idea that
some easily grasped things form a set: the natural numbers
(Infinity), the subsets of a  
 set (Powerset), the elements of a set that satisfy a 
property (Separation),
the images of a set's elements under a function (Replacement), and so forth.
  This is what makes these axioms so compelling. But Mahlo's principle  
  does not do this.  
 
  The second problem is that Mahlo's principle, or indeed any addition to ZF that implies the existence of an
  uncountable regular limit ordinal, seems to be \emph{entangled with choice} in
  light of Gitik's result.
  Admittedly this view is contentious, as some people would try to justify 
  Mahlo's principle via the following choiceless argument: ``For any 
  $F \ccolc \Ord \rightarrow \Ord$, the property of being an 
  $F$-closed regular limit 
 can be reflected down from Absolute Infinity to an ordinal.''   But such
 thinking is avoided in this paper.
 

 \psubsection{Motivating Broad Infinity} \label{sect:motivbi}

In light of the preceding discussion, my primary goal was 
to obtain an axiom scheme that 
\begin{enumerate}
\item \label{item:eqmahlo} is equivalent to Mahlo's
  principle, assuming AC
  \item \label{item:grasp} asserts that some easily grasped
  things form a set
\item \label{item:noreg} does not imply (given only ZF) that an uncountable
regular limit ordinal exists.
\end{enumerate}
To this end, I propose Simple Broad Infinity. Does it meet the requirements?
\begin{enumerate}
\item \label{item:eqnmahloans} Assuming AC,  we shall prove that Simple Broad Infinity is equivalent to Mahlo's
  Principle.  So this requirement is met.
\item \label{item:graspans} Simple Broad Infinity
  asserts, for each function $F \ccolc \totall \rightarrow \allsets$,
  that the class of all {simple $F$-broad numbers} (explained in Section~\ref{sect:broadinftech} below) is a
  set.  Arguably this is ``easily grasped'', but the question is
  subjective and must be left to the reader's judgement.
\item \label{item:noregans} In Broad ZF, I see no way to obtain the
  existence of an uncountable regular limit ordinal.  However,  
  an analogue of Gitik's result is currently lacking.
\end{enumerate}





\psection{Goals and Structure of the Paper} \label{sect:goals}

\psubsection{Plausible vs Useful} \label{sect:plausibleuseful}

Simple Broad Infinity has been designed to be as \emph{plausible} as
possible.  In other words, I aimed to minimize the mental effort needed to believe
it.  This is surely a desirable feature for an axiom scheme.
Furthermore, disentanglement from choice helps to achieve it because, even for a person who finds AC intuitively convincing (as
I do), it is easier to accept one intuition at a time.

My second goal was different: to find an equivalent scheme that is as \emph{useful} as possible.  In other words, I wanted to
minimize the effort needed to apply it.    
In particular,  it should \emph{obviously} imply the existence
of Grothendieck universes, without requiring a detour via notions of 
ordinal or cardinal.  

To this end, I propose a scheme called \emph{Broad Set
  Generation}.  For people who accept AC, this meets the stated
goal.  For those who do not, I offer instead a principle
called \emph{Broad Derivation Set}.  The latter yields the existence of
``Tarski-style'' universes that are sometimes used in the literature~\cite{MartinLoef:inttypetheory}.

\psubsection{Urelements and Non-Well-Founded Membership} \label{sect:urnwf}

In ZF, everything is a set and the membership relation is
\wellfound{}.  But our results also hold in variants of ZF that 
allow urelements and/or non-\wellfound{} membership~\cite{wikipedia:urelement,yao:thesis,wikipedia:nonwellfounded,Aczel:book}.  Making
this clear is the third goal.  

\psubsection{Weak Choice Principle} \label{sect:weakchoice}

Although---as stated above---some of our results depend on AC, the
full strength of this axiom is not needed.  More precisely, a weak
form of choice known as WISC (Weakly Initial Set of Covers) suffices
for our results.  Explaining this fact is the fourth goal.

Caveat: we shall see different versions of WISC, and care must be taken to use an appropriate one. In ZF, they are all equivalent.


\psubsection{Wide vs Broad} \label{sect:widevsbroad}

We give the name ``Broad'' to the principles studied in this paper that go beyond ZFC.  It
turns out that each of them has a   
ZFC-provable counterpart that we call ``Wide''.  For example,
Mahlo's principle is Broad, and its Wide counterpart is Blass's
axiom.

The fifth goal is to convey this pattern of resemblance, which is
depicted in \Cref{fig:diagsub}, a summary of the results in the
paper (using a base theory weaker than ZF).  The rows within each block are equivalent,  
 and each arrow represents inclusion of theories---i.e., reverse
 implication.  The Wide principles appear on the left and the
 corresponding Broad principles on the right.

 \begin{figure}
   \noindent \textbf{Without assuming the Axiom of Choice}
{\small \begin{displaymath} 
    \xymatrix@R=1.5pc@C=5em{
      {
        \begin{array}{l}
          \pmiplong{\ptset + \text{Infinity}}  \\
              \text{Wide Infinity\spl{}} \\
          \text{Wide Derivation Set\wew{}} \\
          \text{Injective Wide \setgen{}\wew{}}
         \end{array}
      } \ar[r] \ar[d]
     &{\begin{array}{l}
         \pmipblong{\noidea} \\
                  \text{Broad Infinity\spl{}} \\
          \text{Broad Derivation Set} \\
         \text{Injective Broad \setgen{}} 
        \end{array}
      }  \ar[d] \\ 
      {
        \begin{array}{l}
          \pmiplong{\ptset +\text{Wide \Supgeneration{}\spl{}}\quad} \\
          \ptset + \text{Blass's Axiom} 
          \\
          \text{Wide \setgen{}\wew}
        \end{array}
      }  \ar[r] & 
        {
          \begin{array}{l}
                \pmipblong{\ptset +\text{Broad \Supgeneration{}\spl{}}\quad} \\
            \ptset + \text{Mahlo's Principle} 
            \\
          \text{Broad \setgen{}}
        \end{array} 
      }
    } \end{displaymath} }
\noindent \textbf{Assuming the Axiom of Choice or at least WISC}
{\small \begin{displaymath} 
    \xymatrix@R=1pc@C=5em{
      {
        \begin{array}{l}
          \pmiplong{\ptset + \text{Infinity}}  \\
              \text{Wide Infinity\spl} \\
          \text{Wide Derivation Set\wew} \\
            \pmiplong{\ptset +\text{Wide \Supgeneration{}\spl} \,} \\
          \ptset + \text{Blass's Axiom} 
          \\
          \text{Wide \setgen{}\wew} 
         \end{array}
      } \ar[r] 
     &{\begin{array}{l}
         \pmipblong{\noidea} \\
                  \text{Broad Infinity\spl} \\
          \text{Broad Derivation Set} \\
              \pmipblong{\ptset +\text{Broad \Supgeneration{}\spl}} \\
            \ptset + \text{Mahlo's Principle} 
            \\
         \text{Broad \setgen{}} 
        \end{array}
      } 
    } \end{displaymath} }
\jspl{} The Simple and \full{} versions are equivalent.

\jwew{} The Wide and \Qwide{} versions are equivalent.

\caption{Diagram of theories, each extending the base theory} \label{fig:diagsub}
\end{figure}

\psubsection{Summary of Goals} \label{sect:summgoals}

To summarize the previous sections, our goals are as follows.
\begin{enumerate}
\item \label{item:plaus} To give a simple and plausible axiom scheme, disentangled
  from choice, that is equivalent over ZFC to
  Mahlo's principle.  \ \textbf{Solution} \ Simple Broad Infinity.
\item \label{item:useful} To give an equivalent principle that is 
  convenient for applications.  \ \textbf{Solution} \ Broad \setgen{} for
  those who accept AC, and Broad
  Derivation Set for those who do not.
 \item \label{item:urnwf} To show that our results hold even when urelements
   and non-\wellfound{} membership are allowed.
  \item \label{item:wiscgoal} To show that, for the results that rely on AC, a weak choice
    principle suffices.
   \item \label{item:widebroad} To convey the resemblance between Wide principles (which are
     provable in ZFC) and Broad principles (which are not, provided ZF is consistent).  
 \end{enumerate}
Throughout the paper, we use a base theory that includes classical
first-order logic, and do not consider the issue of logical complexity.    Other versions of set theory are left to future work.

 \psubsection{Related Work} \label{sect:relwork}

Many formulations of Mahlo's principle have been 
studied~\cite{Jorgensen:generate,Mayberry:book,Levy:stronginf,Montague:twocontribs,Dowd:somenewaxioms},
and variations have been given for type
theory~\cite{Rathjen:superjump,Setzer:mlttonemahlo} and Explicit
Mathematics~\cite{KahleSetzer:extendedpredmahlo}.  Other principles
have been considered that are equiconsistent with Mahlo's principle~\cite{Hamkins:simplemax,Mathias:happyfam}.

Another related topic---which inspired the Broad Derivation Set
principle---is  the treatment of
``induction recursion'' in type
theory~\cite{DybjerSetzer:indindrec,GhaniHancock:contmonir}.  It is used in the proof assistant Agda, allows the
formation of Tarski-style universes (as in \Cref{sect:tarski}),
and was modelled in~\cite{DybjerSetzer:indindrec} using a Mahlo
cardinal.

\psubsection{Structure of Paper} \label{sect:structpaper}

Before treating the wide and broad principles, the paper presents
various foundational concepts in \cref{part:founds}, beginning in \Cref{sect:basicsets} with an introduction to
sets and classes.  \Cref{sect:wfscaffolds} treats
\wellfound{}ness, and gives a way to generate subclasses and
partial functions; this is 
is used throughout the paper, and especially to formulate the Derivation
Set principles.  Next, \cref{sect:basicord} is devoted to ordinals, and explains how
to use an inductive chain to obtain a least prefixpoint.  Lastly,
\cref{sect:cattheory} is devoted to category theory, notably the concept of an initial algebra.

\Cref{part:widebroadsets} is devoted to those wide and broad principles that are
concerned with sets and rubrics (not ordinals), beginning in \Cref{sect:infprin} with the wide and
broad infinity principles.  This is followed in \cref{sect:rubrics} by
the useful principles  of 
\setgen{} and Derivation Set, with restricted versions of these
principles considered in \cref{sect:specialrub}.  In  \cref{sect:provesetgen}, we see how
to deduce Set Generation principles from Derivation Set principles, by
either imposing an injectivity condition or assuming AC or WISC.

\Cref{part:widebroadord} presents the wide and broad principles for ordinals, beginning
in \cref{sect:supgen} with ``supgeneration'' principles
that connect the world of sets to that of ordinals.  \Cref{sect:lind} explains the concept of
Lindenbaum numbers, which is known from the literature on choiceless
mathematics.   This allows us in 
\Cref{sect:mahlo}  to 
develop Mahlo's principle and establish all its relationships.
Lastly, \Cref{sect:powstat} presents the traditional use of Mahlo's
principle (in a class setting) to prove the
existence of various kinds of ordinal.

\Cref{part:wrap} wraps up the paper by summarizing the contributions 
and suggesting further work.

Some readers may just want to see the ZFC proof that Simple Broad
Infinity is equivalent to Mahlo's principle.  This is divided into
several steps:
\begin{itemize}
\item Simple Broad Infinity is equivalent to Full Broad
  Infinity---\cref{prop:broadinf}(\ref{item:broadsimplefull}).
\item Full Broad Infinity implies Broad Derivation Set---\cref{prop:implderiv}(\ref{item:broadimplderiv}).
\item Broad Derivation Set implies Broad \setgen{}---\cref{prop:chainac}(\ref{item:acrange}).  
  Only this step uses AC.
\item Broad \setgen{} implies Full Broad
  Infinity---\cref{prop:setgeninf}(\ref{item:broadsginf}).
\item Broad \setgen{} is equivalent to Broad \Supgeneration{}---\cref{prop:setsup}(\ref{item:broadsetsup}).
\item Broad \Supgeneration{} is equivalent to Mahlo's principle---\cref{prop:equivbound}(\ref{item:equivboundbroad}).
\item Various definitions of stationarity,  each giving a different
  formulation of Mahlo's principle,  are equivalent---\cref{prop:equivstat}.  
\end{itemize}

\ppart{Foundations} \label{part:founds}

\psection{Basic Theory of Sets}
\label{sect:basicsets}
 \psubsection{Our Base Theory} \label{sect:basetheory}

 For this paper, I have chosen a base theory that differs from ZF in
 several ways:
 \begin{itemize}
 \item It allows urelements
   and non-\wellfound{} membership.
 \item It excludes Powerset and 
 Infinity, so that we can examine how these
 axioms relate to other principles.
 \item It allows undefined unary 
   predicate symbols, also known as class variables.
 \end{itemize}
For a given set $\symb$ of predicate symbols, the syntax  
is as follows:
\begin{spaceout}{lll}
  \phi,\psi & \bnfgo & P(x) \bnf  \isset{x} \bnf x \in y \bnf x = y \bnf
  \mathsf{True} \bnf \mathsf{False}  \\
 & & \quad \bnf \neg \phi \bnf \phi
  \wedge \psi  \bnf 
         \phi \vee \psi \bnf \phi \Rightarrow \psi  \bnf
         \forall
              x.\,\phi \bnf \exists x.\,\phi
            \end{spaceout}%
with $P$ ranging over $\symb$.  The formula $\isset{x}$ asserts that $x$ is a
set.   

We define the \emph{base theory} over $\symb$ to be the classical first-order theory
with equality, axiomatized as follows. 
 \begin{itemize}
  \item \label{item:ext} Axiom of \emph{Extensionality}: Any two sets with the same elements are equal.
   \item Axiom of \emph{Inhabitation}: Anything that has an element is a set.   
     \item \label{item:replace} Axiom scheme of \emph{Replacement}: For any set $A$ and binary predicate $F$ 
    such that each $x \smin A$ has a unique $F$-image, there is a
    set $\setbr{F(x) \mid x \smin A}$ of all $F$-images of  
    elements of $A$.
 \item Axiom of \emph{Twoity}: There are sets $0,1,2$ such that $0 =
   \setbr{}$ and $1 = \setbr{0}$ and $2 = \setbr{0,1}$.  
 \item  Axiom of  \emph{Union Set}: For any set of sets $\cata$, there
s   is a set $\bigcup \cata$ of all elements of elements of $\cata$.
 \end{itemize}

{\bf  Henceforth we assume the base theory.}    Pairing and Separation follow via 
  \begin{eqnarray*}
    \setbr{x_0,x_1} & \eqdef & \setbr{x_i \mid i \in \setbr{0,1}} \\
    \setbr{x \smin A \mid P(x)} & \eqdef & \bigcup_{x \in A} \left\{
                                         \begin{array}{ll}
                                           \setbr{x} & \text{ if
                                                       $P(x)$} \\
                                           \emptyset & \text{ otherwise}
                                         \end{array} \right.
  \end{eqnarray*}

\paragraph*{Related work} \quad  For set theory without Powerset,
see~\cite{GitmanHamkins:nopset}.  For set theory without
Infinity, see~\cite{KayeWong:arithset}.

\psubsection{Classes, Functions and Partial Functions} \label{sect:classes}

 Since classes are so important in our story, much of this paper is
 devoted to studying them.  As usual in set theory, a class is
 represented as a predicate formula with
 parameters.  We write $\totall$ for the class of all things, $\allsets$ for the class of all sets, and
 $\mathsf{Ur}$ for that of all urelements (things that are not sets).
 A class $C$ is \emph{inhabited} when it has an element---i.e., is not
 empty.  For $x \smin C$, the phrase 
 ``$x$ is contained in $C$'' mean $x \smin C$. 

 Given classes $A$ and $B$, we write 
 \begin{eqnarray*}
 A \times B & \eqdef & \setbr{\tuple{x,y} \mid x \smin A, y \smin B} \\
A+B & \eqdef & \setbr{\itinl x \mid x \smin A} \cup
\setbr{\itinr y \mid  y \smin B} 
\end{eqnarray*}
where $\tuple{x,y} \eqdef \setbr{\setbr{x},\setbr{x,y}}$ and $\itinl x \eqdef
\tuple{0,x}$ and $\itinr y \eqdef \tuple{1,y}$.  The notation  $A
\subseteq B$ means that $A$ is \emph{included} in (i.e., a subclass of) $B$.
The notation $F
\ccolc A \rightarrow B$ means that $F$ is a function sending each $x
\smin A$ to an element of $B$.

For a function $F$ on a class $A$, the restriction of $F$ to a
subclass $C$ of $A$ is written 
$F \restriction_C$, or simply as $F$ when $C$ is clear from the context.

For a set $K$, a \emph{$K$-tuple} is a function on $K$.  It is
written as $[x_k]_{k \in K}$ and envisaged as a column with $K$
entries.  It is \emph{within} a class $C$ when, for all $k \smin K$,
we have $x_k \in C$.

A \emph{family} consists of a set $I$ and a function $x$ on $I$.  More generally, a \emph{class-family} consists of a class
$I$ and function $x$ on $I$.  It may be written
as $(I,x)$ or as $(x_i)_{i \in I}$.  It is \emph{injective} when the function $x$
is injective, and its \emph{range} is the range of $x$.  It is \emph{within} a class $C$
when, for all $i \smin I$, we have $x_i \in C$.

Given class-families $(x_i)_{i \in I}$ and
$(y_j)_{j \in J}$,  we say that the former is \emph{included} in the latter when $I \subseteq J$ and $x =
y \restriction_I$.  A  \emph{map} $(x_i)_{i \in I} \rightarrow (y_j)_{j \in J}$ is a
  function $f \ccolc I \rightarrow J$ such that, for all $i \smin I$,
  we have $x_i = y_{f(i)}$.  It is an \emph{isomorphism} when $f$ is
  bijective; note that isomorphic class-families have
  the same range.

  The following will be useful.
  \begin{proposition} \label{prop:basicsurj}
 Let $C$ be a class.  Any sets $A,B$ and surjection  $f \ccolc A
\rightarrow B$ yield an injective function $C^f \ccolc C^B \rightarrow C^A$ sending
$[x_b]_{b \in B}$ to $[x_{f(a)}]_{a \in A}$.  Its range
is the class of all $A$-tuples $y$ such that, for
all $a,a' \smin A$ with the same $f$-image, we have $y_{a} =
y_{a'}$.  
\end{proposition}
\begin{proof}
 Straightforward.
\end{proof}

For a class $C$, we write $\psetset C$ for the class of all subsets of
$C$, and $\psetinh C$ for the class of all inhabited subsets of $C$.
We write $\famof{C}$ for the class of all families within $C$, and
$\injfof{C}$ for the class of all injective families within $C$.  For
any set $I$, we write $C^I$ for the class of all $I$-tuples within $C$.

Given a classes $A$ and $B$, a \emph{partial function} $G \ccolc A \rightharpoonup B$ consists of a subclass
$\domof{G}$ of $A$ and a function $\overline{G} \ccolc \domof{G} \rightarrow B$.  Put
differently,  it is a class-family $(M,F)$ such that for all
$x \smin M$ we have $x \in A$ and $F(x) \in B$. 

We also speak about ``collections'', although such talk is
informal.

We write $\class$ for the collection of all classes.  Given a class $C$, we write $\subclassof{C}$ for the collection of all
subclasses of $C$.  We write $\classfamof{C}$ for the collection of all
class-families within $C$, and $\injcfof{C}$ for the collection of all
injective class-families within $C$.






Given a class $A$, we may speak of a function $B \ccolc A
\rightarrow \class$, also called an \emph{$A$-tuple of classes} and
written $[B_a]_{a \in A}$.   It is represented as a binary predicate formula $\phi(x,y)$ with
 parameters, so that, for $x \smin A$, we have  $B_x = \setbr{y \mid
   \phi(x,y)}$.  The pair $(A,B)$, also written $(B_x)_{x \in A}$, is
 called a \emph{class-family of classes}, or a \emph{family of
   classes} if $A$ is  a set.

Given a class-family of classes $(B_x)_{x \in A}$,  we form the class
 \begin{eqnarray*}
   \sum_{x \in
    A}B_x & \eqdef & \setbr{\tuple{x,y} \mid x \smin A, y \smin B_x}  
 \end{eqnarray*}
The notation $F \smin \prod_{x \in A} B_x$ means that $F$ is a function on $A$ that sends each $x
  \smin A$ to an element of $B_x$.  
  Likewise, a \emph{partial function} $G \smin \classfamoo{B_x}{x \in A}$ consists of a subclass $\domof{G}$ of $A$ and function
  $\overline{G} \in \prod_{x \in \domof{G}} B_x$.  Put differently, it is a
  class-family $(M,F)$  
   such that  for all
   $x \smin M$ we have $x \in A$ and $F(x) \in B_x$.  It is
   \emph{small} when it is a family---i.e., the domain is a set. 
   

 Two other kinds of function occur in the paper.
\begin{itemize}
\item Given a class $A$, we speak of a function $F \ccolc A \rightarrow \catb$,
  where $\catb$ is a collection, or---more generally---of a function 
  $F \in \prod_{x \in A} \catb_x$, where $\catb$ is an
  $A$-tuple of collections.  In  each instance, it is obvious how $F$ can be
  represented.  For
  example, we can represent $F \ccolc A \rightarrow \class^2$ as a
  pair of functions $A \rightarrow \class$.
\item Given collections $\cata$ and $\catb$, we speak of a function $\cata \rightarrow \catb$.  For example, $\psetset$ is an endofunction on
  $\class$.  
\end{itemize}

\psubsection{Class Reasoning: Arbitrariness and Predicativity} \label{sect:arbpred}


When speaking about classes, we must observe two
disciplines.

Firstly, in order to assert that all classes have a given property, it
 is insufficient to prove this merely for classes that are 
 definable from first-order parameters.  Instead we must prove that an
 \emph{arbitrary} class has the property.  
 That is because our base theory's syntax 
includes class variables.

Secondly, in order to assert the existence of a class with a given property, we
must prove this \emph{predicatively}---i.e., without quantification
 over class variables.  That is because our base theory's syntax does not provide
 such quantification.

Both requirements are illustrated in Section~\ref{sect:nats}.

 \psubsection{Powerset, Choice and Collection} \label{sect:powchoicecoll}
 
We now present some additional principles, beginning with Powerset.
\begin{proposition} \label{prop:powequiv}
    The following are equivalent.
    \begin{itemize}
    \item \emph{Powerset}: For any set $A$, the class $\pset A$ is a set.
    \item \emph{Exponentiation}:  For any sets $A$ and $B$, the class
      $B^A$ is a set.
    \item For any family of sets
      $(B_i)_{i \in I}$, the class $\prod_{i \in I}B_i$ is a set.
    \end{itemize}
  \end{proposition}
  \begin{proof}
    Via the following constructions.
    \begin{eqnarray*}
      \pset A & \eqdef & \setbr{\setbr{x \smin A \mid f(x) =
                         1} \mid f \in \setbr{0,1}^A } \\
      B^A & \eqdef & \prod_{x \in A} B \\
      \prod_{i \in I}B_i & \eqdef & \setbr{f \in \pset \sum_{i \in
                                    I}B_i \mid \forall i \smin
                                    I.\, \existsu b \smin B_i.\,
                                    \tuple{i,b} \in f}
    \end{eqnarray*}
  \end{proof}

  We continue with the following principles.
  \begin{itemize}
 \item The axiom of \emph{Choice} (AC): For any family of inhabited
   sets $(A_i)_{i \in I}$, the class $\prod_{i \in I}A_i$ is
  inhabited.    
  \item The axiom scheme of \emph{\classchoice{}}: For any family of
    inhabited classes $(A_i)_{i \in I}$, the class $\prod_{i \in I}A_i$ is
inhabited.
 \item The axiom scheme of \emph{Collection}: For any family of
    inhabited classes $(A_i)_{i \in I}$, the class $\prod_{i \in I}
   \psetinh A_i$ is inhabited.
 \end{itemize}

 \begin{proposition} \label{prop:collchoice} \hfill
   \begin{enumerate}
   \item \label{item:classchoice} \classchoice{} is equivalent to Collection $+$ AC.
   \item \label{item:zfcoll} {\bf In ZF} Collection holds.
   \end{enumerate}
 \end{proposition}
 \begin{proof} \hfill
   \begin{enumerate}
   \item \label{item:classchoiceproof} For ($\Rightarrow$), it is obvious that AC holds, and Collection is proved
     as follows.  Given a set $I$ and $I$-tuple of inhabited
     classes $A$, we obtain $x \smin \prod_{i \in I}A_i$ by \classchoice{}, and then
     $i \mapsto \setbr{x_i}$ inhabits $\prod_{i \in I} \psetinh A_i$.
     For ($\Leftarrow$), given a set $I$ and $I$-tuple of
     inhabited classes $A$, we obtain
     $B \smin \prod_{i \in I} \psetinh A_i$ by Collection, and then---by
     AC---the subclass 
       $\prod_{i \in I} B_i$ of
       $\prod_{i \in I}A_i$ is inhabited.
     \item \label{item:zfcollproof} We write $(V_{\alpha})_{\alpha \in \Ord}$ for the cumulative
  hierarchy in the usual way.  Given a set $I$ and $I$-tuple of inhabited
     classes $A$, we proceed as follows.  For each $i \smin I$, define
     $t(i)$ to be the least ordinal $\alpha$ such that $A_i \cap
     V_{\alpha}$ is inhabited. Then $i \mapsto A_i \cap V_{t(i)}$
     inhabits $\prod_{i \in I} \psetinh A_i$.  \qedhere
     \end{enumerate}
 \end{proof}

\psubsection{Ordered Collections}  \label{sect:ordcoll}

Order plays
a large role in our story, so we present
some useful notions.

Given an ordered collection $\cata$, a subcollection $\catb$ is \emph{lower}
when, for any $x \smin \catb$ and $y \leqslant x$, we have $y \in
\catb$.

Given  ordered collections $\cata$ and $\catb$,
a function $h \ccolc \cata \rightarrow \catb$ is \emph{monotone} when
$x \leqslant y$ implies $h(x) \leqslant h(y)$.




In this paper, a \emph{\standcoll{}} is a collection
$\cate$ equipped with a class-family of classes $(B_x)_{x \in A}$, and a bijection
$\cate \cong \classfamoo{B_x}{x \in A}$.  This structure induces an
order on $\cate$, so we write $\bigvee$ for supremum and $\bot$ for
the least element.   It also induces a notion of smallness: we write $\fund{\cate}$ for the class of all small
elements.  For $x \smin \cate$, we write
$\waybel{x}$ for the class of all $y \smin \fund{\cate}$ such that $y
\leqslant x$.  



For example, let $C$ be a class.
\begin{itemize}
\item $\subclassof{C}$ is a \standcoll{}, since subclasses of
  $C$ correspond to partial functions $C \rightharpoonup 1$ via the
  bijection $X \mapsto (X, x \mapsto *)$.    A small element is a
  subset of $C$.
\item $\classfamof{C}$ is a \standcoll{}, since class-families within
  $C$ are the same thing as partial functions $\totall \rightharpoonup
  C$.  A small element is a family within $C$.
\end{itemize}

The following is adapted from~\cite{Aczel:book,Takahashi:inductionset}.

\begin{definition} \label{def:setcont}
  Let $\catd$ and $\cate$ be \standcoll{}s.
  \begin{enumerate}
  \item \label{item:extfund} For any monotone function 
    $f \ccolc \fund{\catd} \rightarrow \cate$, the monotone
    extension $\hat{f} \ccolc \catd \rightarrow \cate$ is defined as 
    $x \mapsto \bigvee_{y \in \subwaybel{x}} f(y)$.\footnote{In the
      language of category theory, $\hat{f}$ is the left Kan extension of
      $f$ along the inclusion $\fund{\catd} \subseteq \catd$.  That is, the least monotone function $g \ccolc \catd
      \rightarrow \cate$ such that $f \leqslant g \restriction_{\fund{\catd}}$.}
  \item \label{item:setcont} A function
    $h \ccolc \catd \rightarrow \cate$ that arises from
    its restriction to $\fund{\catd}$ in this way is said to be
    \emph{set-continuous}.  Explicitly, this means that $h$ sends each
    $x \smin \catd$ to $\bigvee_{y \in \subwaybel{x}} h(y)$.
  \end{enumerate}
\end{definition}
  
  For example, $\pset$ is a set-continuous endofunction on $\class$.  All 
  functions between \standcoll{}s considered in this paper are
  set-continuous (and hence monotone).

Note: since a set-continuous function $\catd \rightarrow \cate$ can be
represented by its restriction to $\fund{\catd}$, any quantifier
ranging over set-continuous functions can be regarded as ranging over classes.

 \psubsection{Fixpoints} \label{sect:fixpoint}

 We present some notions of fixpoints and order, as they are repeatedly used
 in the paper.
 

Let $\cate$ be a collection, and $h$ an endofunction on $\cate$.  Then  an
element $x \smin \cate$ is
\emph{$h$-fixed} or an \emph{$h$-fixpoint} iff $h(x) =
x$.   The prefix $h$ can be omitted when clear from the context.

Now let $\cate$ be an ordered collection, and  $F$ a monotone
endofunction on $\cate$.  An element $x \smin \cate$
is
\emph{$h$-prefixed} or an \emph{$h$-prefixpoint} when $h(x) \leqslant x$, and
   \emph{$h$-postfixed} or an \emph{$h$-postfixpoint} when $x \leqslant h(x)$.\footnote{This
   terminology follows~\cite{SmythPlotkin:catrecdom}.  Some authors use the opposite
   terminology, following~\cite{MannaShamir:convfun}.}  
 So $x$ is fixed iff it is both \mayg{}prefixed and
\mayg{}postfixed.  Note that an infimum of
 \mayg{}prefixpoints is \mayg{}prefixed, and a supremum of
 \mayg{}postifxpoints is \mayg{}postfixed.

We say that $h$ is \emph{inflationary} when every $x \smin
\cate$ is a postfixpoint, and \emph{deflationary} when every $x \smin \cate$ is a
prefixpoint.



The least \mayg{}prefixpoint of $h$, if it exists, is written
$\mu h$.  It is necessarily \mayg{}fixed; this fact is called
\emph{inductive inversion}.  Dually the greatest \mayg{} postfixpoint
of $h$, if it exists, is written $\nu h$ and is necessarily
\mayg{}fixed.   




A \mayg{}prefixpoint $x$ is \emph{minimal} when the only 
\mayg{}prefixpoint $y$ such that $y \leqslant x$ is $x$ itself.  A least prefixpoint is
minimal, and conversely if $\cate$ has
 binary meets, which is always the case for a \standcoll{}.

 \psubsection{Natural Numbers} \label{sect:nats}

Bearing in mind that we do not assume Infinity, we must carefully
define the class $\nats$ of
all natural numbers.  Specifically, we shall construct the \emph{Zermelo
  natural numbers}:
\begin{eqnarray*}
  \nats & = & \setbr{\Zero, \Succb{\Zero}, \Succb{\Succb{\Zero}},
              \ldots} 
\end{eqnarray*}
where $\Zero \eqdef \emptyset$ and $\Succb{x} \eqdef \setbr{x}$.  
Firstly, we define the monotone endofunction
$\incs$ on
$\class$ that sends $X$ to 
\begin{math}
  \setbr{\Zero} \cup \setbr{\Succb{x} \mid x \smin X}
\end{math}.\footnote{This terminology comes from functional programming~\cite{Haskell:maybe}.}
So a class $X$ is
$\incs$-prefixed iff it  contains $\Zero$ and, for any $x \smin X$,
contains $\Succb{x}$.  We want $\nats$ to be the least $\incs$-prefixed class---``least'' means that $\nats$ is included in an arbitrary
$\incs$-prefixed class.

We 
cannot simply define $\nats$ to be the intersection of all
$\incs$-prefixed classes, as that would be impredicative.  Instead we
proceed as follows.
     \begin{proposition} \label{prop:nats}
     The class $\nats \eqdef \mu \incs$ exists.
   \end{proposition}
   \begin{proof}
First note that a class $X$ is $\incs$-postfixed iff every $x \smin
     X$ is either $\Zero$ or $\Succb{x}$ for some $x \smin
     X$.   When this is so, say that a subclass $U$ of $X$ is \emph{inductive}
     when it contains $\Zero$ if $X$ does, and, for all $x \smin
     U$, contains $\Succb{x}$ if $X$ does. 

  For an inhabited class $I$, we note that $\incs$ preserves
  $I$-indexed intersections, 
 and so any $I$-indexed intersection of $\incs$-postfixpoints is
 $\incs$-postfixed.  Therefore, any thing  $x$  that is contained in a $\incs$-postfixed set is
  contained in a least such set---viz., the intersection of all
  $\incs$-postfixed sets that contain $x$.  We call this set $\predof{x}$. 


     Define $\nats$ to be the
     class of all $x$ such that $\predof{x}$ exists and its only
     inductive subset is 
  $\predof{x}$ itself.  This class has the required properties.
\end{proof}

   We often write $0 \eqdef \Zero$ and, for $n \smin \nats$, write
   $n+1 \eqdef \Succb{n}$.  
   The standard properties of $\nats$ hold, including
   the following.
   
   \begin{proposition}[Recursion over $\nats$] \label{prop:recurnat}
     For any sequence of classes $(B_n)_{n \in \nats}$ and
     any 
      $p \smin B_0$ and $L \smin
     \prod_{n \in \nats} (B_n \rightarrow B_{n+1})$, there is a
       unique sequence $b \smin \prod_{n \in \nats} B_n$ such that $b_0 = p$ and,
       for all $n \smin \nats$, we have $b_{n+1} = L_n(b_n)$.
     \end{proposition}

The axiom of \emph{Infinity} says that a $\incs$-prefixed set
exists; this is equivalent to $\nats$ being a set.

\psection{Well-Foundedness and Scaffolds} \label{sect:wfscaffolds}

 \psubsection{Set-Based Relations} \label{sect:setbased}

 We shall now consider relations on a class, leading up to the key notion
 of \wellfound{}ness.
     \begin{definition} \label{def:desc}
       Let $(C,<)$ be a class equipped with a relation.  
       \begin{enumerate}
       \item \label{item:childdesc} Let $x \smin C$.  An element $y \smin C$ is
         \begin{itemize}
         \item a \emph{child} of $x$ when $y < x$.
         \item a \emph{descendant} of $x$, written
           $y <^* x$, when there is a sequence
           \begin{displaymath}
y = z_0 < \cdots < z_n =
           x
         \end{displaymath}
         \item a \emph{strict descendant} of $x$, written
           $y <^+ x$, when there is such a sequence with $n > 0$.
         \end{itemize}
         We write $J_{<}(x)$ for the class of all children of $x$, and 
  $J_{<}^*(x)$ for the class of all descendants, and
  $J_{<}^+(x)$ for the class of all strict descendants. 
  The subscript $<$ may be omitted when clear from the context.
    \item \label{item:hered} A subclass $X$ of $C$ is \emph{hereditary} when every
           child of an element of $X$ is in $X$.
         \end{enumerate}
       \end{definition}
       Thus, for $x \smin C$, we see that $J_{<}^*(x)$ is the least hereditary subclass of $C$ that contains
  $x$. 

\begin{example} \label{example:membershipdesc}
  Consider the membership relation on $\totall$. A class $X$ is \emph{membership-hereditary} or \emph{transitive} when every element of
  an element of $X$ is in $X$.    For a thing
       $x$, we write $\elset{x}$ for its \emph{element set},
       which is  $x$ or $\emptyset$ according as $x$ is a set or
an urelement.  We write 
       $\memreach{x}$ for the class of all 
membership-descendants of $x$, and $\memplus{x}$ for the class of all
 strict ones.  Thus $\memreach{x}$ is the least
  transitive class containing $x$.  
\end{example}

     \begin{definition} \label{def:setbased} 
       A relation $<$ on a class $C$ is
       \begin{itemize}
       \item \emph{\setb{}} when $\jlt(x)$ is a set for all $x
         \smin C$.
       \item \emph{\fullsetb{}} when each $x\smin C$ is contained
         in a hereditary subset of $C$; this is equivalent to 
         $\jltstar(x)$ being a set.  
         %
       \end{itemize}
     \end{definition}
    \noindent Thus $<$ is \fullsetb{} iff $<^*$ is \setb{}.

     
     \begin{proposition} \label{prop:inffull}
      Infinity is equivalent to the statement: ``Every \setb{}
         relation on a class is \fullsetb{}.''
     \end{proposition}
     \begin{proof} 
              For ($\Rightarrow$), let $<$ be a \setb{} relation on a
              class $C$, and $x
              \smin C$.  By induction on $n \smin \nats$, the class $\jlt^{n}(x)$ of
              descendants of $x$ at depth $n$ is a set, since
              \begin{eqnarray*}
                \jlt^{0}(x) & = & \setbr{x} \\
                \jlt^{n+1}(x) & = & \bigcup_{y \in \jlt^{n}(x)} \jlt(y) 
              \end{eqnarray*}
 So the class $\jltstar(x) =
 \bigcup_{n \in \nats} \jlt^n(x)$ is a set.

           For ($\Leftarrow$), the relation on $\totall$ given by
           $\setbr{\tuple{x,y} \mid y = \Succb{x}}$ is \setb{}, and the
           descendant class of $\Zero$ is $\nats$.  So if \setb{} 
           implies \fullsetb{}, then $\nats$ is a set.
         \end{proof}

          \begin{example} \label{example:membershipsetbased}
     The membership relation on $\totall$ is \setb{}, and is \fullsetb{} iff  the \emph{Transitive
      Containment} axiom holds: Every thing is contained in a transitive
    set.  Thus, by \cref{prop:inffull}, this axiom follows from Infinity.
   \end{example}

         
\psubsection{Well-Founded Set-Based Relations} \label{sect:wellfound}

The notion of \wellfound{}ness relies on the following concepts.
\begin{definition} \label{def:indmin}
  Let $(C,<)$ be a class equipped with a relation.  Let $X$ be a subclass.
  \begin{enumerate}
  \item \label{item:ind} $X$ is \emph{inductive} when every
    element of $C$ whose children are all in $X$ is in $X$.
 \item \label{item:min} An element of $X$ is \emph{minimal} when it has no
       child in $X$.
  \end{enumerate}
  \end{definition}

  \noindent Now we are ready to formulate \wellfound{}ness.
    \begin{proposition} \label{prop:wellfounded}
      Let $(C,<)$ be a class equipped with a set-based relation.  The
      following are equivalent:
      \begin{enumerate}
      \item \label{item:wfinduct} The only inductive subclass of $C$
        is $C$ itself.
      \item\label{item:wfinhab}  Every inhabited subclass of $C$ has a minimal element.
      \item \label{item:wfinhabset}  The relation $<$ is \fullsetb{}, and every inhabited
        subset of $C$ has a minimal element.
      \end{enumerate}
    \end{proposition}
    \begin{proof}
     Conditions (\ref{item:wfinduct}) and 
     (\ref{item:wfinhab}) are equivalent because a subclass of $C$ is inductive
     iff its complement has no minimal element.  

   To prove (\ref{item:wfinduct}) implies (\ref{item:wfinhabset}), we
   need only show that $<$ is \fullsetb{}: for all $x \smin
   C$, the class $\jltstar(x)$ is a set, by induction on $x$.

   To show (\ref{item:wfinhabset}) implies (\ref{item:wfinhab}), any
   subclass $Y$ of $C$ inhabited by $x$ gives a subset $Y \cap \jltstar(x)$
    of $C$ inhabited by $x$.  The latter has a minimal element, which
    is also a minimal element of $Y$.  
    \end{proof}
When the above conditions hold, we say that $<$ is a 
\emph{\sbwf{} relation}.

\begin{example} \label{example:membershipwo}
Consider again the membership relation on $\totall$.  A class $X$ is  \emph{membership-inductive} when every thing whose elements are all in
  $X$ is in $X$.  An element $x \smin X$ is \emph{membership-minimal}
  when it has no element in common with $X$.   We can say that membership is \wellfound{} in each of
    the following ways.
    \begin{itemize}
    \item The axiom scheme of \emph{Membership Induction}: The only 
      membership-inductive class is $\totall$.
    \item The axiom scheme of \emph{Class Regularity}: Every inhabited class
      has a membership-minimal element.
    \item Transitive Containment $+$ the axiom of \emph{Regularity}: Every
      inhabited set has a membership-minimal element.
    \end{itemize}
 \end{example}

  Here are some basic properties of \wellfound{} relations.
  \begin{proposition} \label{prop:reflwf}
    Let $(A,<)$ and $(B,<')$ be classes equipped with a set-based
    relation, and $f \ccolc A \rightarrow B$ a function such that $x
    < y$ implies $f(x) <' f(y)$.  If $<'$ is \wellfound{}, then $<$ is too.
  \end{proposition}
  \begin{proof}
    Let $X$ be an inductive subclass of $A$.  We prove by induction on
    $y \smin B$ that $f^{-1}(y) \subseteq X$.  
  \end{proof}
  
    \begin{proposition} \label{prop:wf}
Let $(C,<)$ be a class equipped with a set-based relation.
      \begin{enumerate}
      \item \label{item:wfplus} The relation $<^+$ is \wellfound{} iff
        $<$ is.
     \item \label{item:irref}  If $<$ is \wellfound{}, then there is no infinite sequence $\ \cdots
       < x_1 < x_0$.  
      \end{enumerate}
    \end{proposition}
    \begin{proof} \hfill
      \begin{enumerate}      
       \item The direction ($\Rightarrow$) is by
         \cref{prop:reflwf}.  For ($\Leftarrow$), let $X$ be a $<^+$-inductive subclass of $C$.  For all $x
         \in C$, we prove $\jltstar(x) \subseteq C$ by induction on $x$.
       \item Fix such a sequence.  We prove that every $x \smin C$
         fails to appear in it, by induction
         on $x$. \qedhere
      \end{enumerate}
    \end{proof}

  Functions and partial functions can be defined  by \wellfound{} recursion:
 \begin{proposition} \label{prop:wfrecur}
   Let $(A,<)$ be a class equipped with a \sbwf{}
   relation, and $B$ an $A$-tuple of classes.
   \begin{enumerate}
   \item \label{item:wfrecurtot} For any $L \smin \prod_{x \in A} ((\prod_{y \in
       \jlt(x)}B_y) \rightarrow B_x)$, there is a unique function
     $F \smin \prod_{x \in A} B_x$ sending $x \smin A$ to
     $L_x(F \restriction_{\jlt(x))})$.  In other words, the
     endofunction $\Phi_L$ 
     on $\prod_{x \in A}B_x$ sending $F$ to
     $x \mapsto L_x(F \restriction_{\jlt(x)})$ has a unique fixpoint.
   \item \label{item:wfrecurpart}  For any $L \smin \prod_{x \in A}( (\prod_{y \in \jlt(x)}B_y)
    \rightharpoonup B_x)$, let $\Psi_L$ be the monotone endofunction
    on $\classfamoo{B_x}{x \in A}$  sending $(M,F)$ to $(N,G)$, where $N$
    is the class of all $x \smin A$ such that $\jlt(x) \subseteq M$ and $F \restriction_{\jlt(x)}
    \in \domof{L_x}$, and $G$ sends such an $x$ to $\overline{L_x}(F
    \restriction_{\jlt(x)})$.  Then $\Psi_L$ has a least prefixpoint that is also a
    greatest postfixpoint and therefore a unique fixpoint.
   \end{enumerate}
  \end{proposition}
  \begin{proof} We first prove part~(\ref{item:wfrecurpart}).   For a hereditary subclass $M$ of $A$, an \emph{attempt} on $M$
    is  a function $F \smin \prod_{x \in M}  B_x$ such that,
    for all $x \smin M$, we have $F \restriction_{\jlt(x)} \in \domof{L_x}$ and
    $F(x) = \overline{L_x}(F_x \restriction_{\jlt(x)})$. Thus a $\Psi_L$-postfixpoint $(M,F)$ consists of a hereditary subclass $M$ of
    $A$, and an attempt $F$ on $M$.  Induction shows that
    \begin{itemize}
    \item any  attempt on $M$ and attempt on $M'$ agree on $M \cap M'$
    \item any postfixpoint is included in any prefixpoint.
    \end{itemize}
   Let $P$ be the class of
    all $x$ such that there is a (necessarily unique) attempt on
    $\jltstar(x)$.  Let $H$ send each $x \smin P$ to its image under the
    attempt on $\jltstar(x)$.  Then $(P,H)$ is a fixpoint, since any
    attempt $g$ on $\jltplus(x)$ such that $g\restriction_{\jlt(x)}\in \domof{L_x}$
    extends to an attempt  $g
  \cup \setbr{\tuple{x,\overline{L_x}(g \restriction_{\jlt(x)})}}$ on
  $\jltstar(x)$.  So part~(\ref{item:wfrecurpart}) is proved.

  If $L_x$ is total for all $x \smin A$, then any
  $\Psi_{L}$-prefixpoint is total, and  $\Phi_{L}$ is the restriction
  of $\Psi_{L}$ to total functions, so part~(\ref{item:wfrecurtot}) follows.
  \end{proof}

  
\psubsection{Generating a Subclass} \label{sect:gensubclass}

Suppose we have a class $C$.  We sometimes want to show that a given endofunction on
$\subclassof{C}$ has a least
prefixpoint. The following makes this possible.
\begin{definition} \hfill \label{def:scaff}
  \begin{enumerate}
  \item \label{item:scaff} A \emph{scaffold} on $C$ consists of
    \begin{itemize}
    \item a subclass $D$
    \item a relation $<$ from $C$ to $D$.
    \end{itemize}
    We call $x \smin D$ a \emph{parent} and $y < x$ a \emph{child} of
    $x$.  The scaffold is \emph{set-based} when, for all $x\smin D$,
    the class $J_{<}(x) \eqdef \setbr{y \smin C \mid y < x}$ is a set.
  \item \label{item:scafffun} A scaffold $(D,<)$ on $C$ gives rise to a set-continuous
    endofunction $\gamdlt$ on  $\subclassof{C}$, sending $X$ to the class of all $x \smin D$ whose
    children are all in $X$.
  \end{enumerate}
\end{definition}
\noindent Thus a subclass $X$ of $C$ is
  \begin{itemize}
 \item $\gamdlt$-prefixed iff every parent whose
   children are all in $X$ is in $X$
  \item $\gamdlt$-postfixed iff it is a hereditary subclass of $D$.
  \end{itemize}

\begin{proposition} \label{prop:gensub} 
    Let $(D,<)$ be a scaffold on $C$.
    \begin{enumerate}
   \item \label{item:scaffgen} \emph{(Generated subclass.)} \label{item:gensub} If the
      scaffold is set-based, then $\gamdlt$ has a least prefixpoint,
      which is also the greatest postfixpoint on which $<$  is
      \wellfound{}.
       \item \label{item:scaffcogen} \emph{(Cogenerated subclass.)} \label{item:cogensub} The
      endofunction $\gamdlt$ has a greatest postfixpoint.
    \end{enumerate}
 \end{proposition}
 \begin{proof} \hfill
   \begin{enumerate}
   \item \label{item:scaffgenproof}  First note that $<$ is a set-based relation on $C$, since for
     $x \smin C \setminus D$, the class $\jlt(x)$ is $\emptyset$.

     Take the class of all $x \smin C$ such that $\jltstar(x)$ is a
     subset of $D$ whose only inductive subset is itself.  This is
     clearly the least $\gamdlt$-prefixed subclass.  By
     inductive inversion, it is $\gamdlt$-postfixed.  The rest is
     straightforward, using the fact that any \setb{} \wellfound{}
     relation is \fullsetb{}.
      \item \label{item:scaffcogenproof}  Take the class of all $x \smin C$ such that $\jltstar(x) 
        \subseteq D$.
        \qedhere
   \end{enumerate}
 \end{proof}

 \paragraph*{Note} All scaffolds in this paper are on $\totall$,
 except in the proof of \cref{prop:powinfty}, where we use a scaffold on
 a set.  
 
 \begin{example} \label{example:nat}
The endofunction $\incs$ arises from the following scaffold on $\totall$: a parent is either $\Zero$, which has
 no children, or $\Succb{x}$, whose sole child is $x$.  So
 \cref{prop:nats} is an instance of 
 \cref{prop:gensub}(\ref{item:gensub}).
\end{example}

\begin{example} \label{example:wfimpure}
 The endofunction $\gamdltof{\totall}{\in}$ sends a class $X$ to $\mathsf{Ur} \cup 
  \pset X$.  We define
  \begin{eqnarray*}
\wfimpure & \eqdef & \mu
  \gamdltof{\totall}{\in}
  \end{eqnarray*}
  which is the least $\psetset$-prefixed class that includes 
  $\mathsf{Ur}$.   \cref{prop:gensub}(\ref{item:gensub}) tells us that $\wfimpure$ is 
  the least membership-inductive class, and the
  greatest transitive class on which membership is \wellfound{}.  A
  member of this class is called a \emph{vonniad}---the name alludes to
``von Neumann  iteration''. 
\end{example}

\begin{example} \label{example:wfpure}
A thing is \emph{pure} when its membership-descendants are all sets.  The class of all pure things is $\nu \pset$,
which is an instance of
\cref{prop:gensub}(\ref{item:cogensub})  since $\pset =
\gamdltof{\allsets}{\in}$.   Likewise, the class of all
pure vonniads is given by 
  \begin{eqnarray*}
    \wfpure & \eqdef & \mu \pset
  \end{eqnarray*}
which is an instance of \cref{prop:gensub}(\ref{item:gensub}).
\end{example}


\psubsection{Generating a Partial Function} \label{sect:genpfn}

Now suppose we have a class-family of classes $(B_x)_{x \in a}$.  We sometime want
to show that a given endofunction on $\classfamoo{B_x}{x \in
  A}$ has a least prefixpoint.  The following makes this possible.
\begin{definition}  Let $(D,<)$ be a set-based scaffold on $A$. \label{def:scafffunc}
   \begin{enumerate}
   \item \label{item:functscaff} A \emph{functionalization} of $(D,<)$ on $B$ is an $L \in
     \prod_{x \in D}((\prod_{y \in \jlt(x)}B_y) \rightharpoonup
     B_x)$.
   \item \label{item:scafffam} Let $L$ be such a functionalization.  The set-continuous endofunction $\delwl$
  on $\classfamoo{B_x}{x \in A}$ sends $(M,F)$ to $(N,G)$, where
  \begin{itemize}
  \item $N$ is the class of $x \smin D$ such that
    $\jlt(x) \subseteq M$ and $F \restriction_{J(x)} \in \domof{L_x}$
   \item $G$
    sends each such $x$ to $\overline{L_x}(F \restriction_{\jlt(x)})$.
  \end{itemize}
   \end{enumerate}
    \end{definition}

    \begin{proposition}[Generated partial function]  \label{prop:genpf}
 Let $(D,<)$ be a set-based scaffold on $A$ with  functionalization
  $L$ on $B$.  Then $\delwl$ has a least prefixpoint,
  which is also the greatest postfixpoint $(M,F)$ such that $<$ is
  \wellfound{} on $M$. 
\end{proposition}
\begin{proof}
  Let $E$ be the subclass of $A$ generated by $(D,<)$. Since $E$ is $\gamdlt$-prefixed, $\delwl$ restricts to an
  endofunction on $\classfamoo{B_x}{x \in E}$.  By
  \cref{prop:wfrecur}(\ref{item:wfrecurpart}), the latter has a least
  prefixpoint $(M,F)$ that is also a greatest
  postfixpoint, since $<$ is \wellfound{} on $E$.  Because $(M,F)$ is a minimal prefixpoint in
  $\classfamoo{B_x}{x \in E}$, which is a lower subcollection of $\classfamoo{B_x}{x \in A}$, it is a minimal and therefore least
  prefixpoint in 
  $\classfamoo{B_x}{x \in A}$.

  For any postfixpoint $(N,G)$, the
  class $N$ is $\gamdlt$-postfixed.  So if $<$ is \wellfound{} on $N$,
  then $N \subseteq E$, giving 
  $(N,G) \in \classfamoo{B_x}{x \in E}$ and so $(N,G) \leqslant (M,F)$.  
\end{proof}

\psubsection{Introspection} \label{sect:introspect}

\quad \emph{This section is not used in the sequel.}

If membership is \wellfound{}, then $\nats$ is the
\emph{unique}  $\incs$-fixpoint, i.e., the unique class $X$ such
that $x \in X$ iff either $x = \Zero$ or $x = \Succb{y}$ for some $y
\smin X$.   Various other classes defined in the
paper, such as $\wide{S}$ and $\broad{G}$ and $\derivsof{\catr}$ and
$\Ord$, have a similar property.  That is because they are ``introspectively
generated'', in a sense that I now explain.
\begin{definition}  Let $C$ be a class. \label{def:introspect}
  \begin{enumerate}
  \item \label{item:introrel} A relation $<$ on $C$ is \emph{introspective} when $<$ is
    included in $\in^+$.  In other words: when, for all $x \smin C$, we
    have $\jlt(x) \subseteq \memplus{x}$.
  \item \label{item:introscaff} Likewise, a scaffold $(D,<)$ on $C$ is \emph{introspective} when $<$ is
    included in $\in^+$. 
  \end{enumerate}
\end{definition}

\begin{proposition} \hfill \label{prop:equivintrospect}
  \begin{enumerate}
  \item \label{item:tcintro} Transitive Containment is equivalent to the statement: ``Every
    introspective relation on a class is \fullsetb{}.''
  \item \label{item:memintro} Membership Induction is equivalent to the statement: ``Every introspective relation on a class is \wellfound{}.''
  \end{enumerate}
\end{proposition}
\begin{proof} \hfill
  \begin{enumerate}
\item \label{item:tcintroproof}    Transitive Containment is equivalent to membership being 
    \fullsetb{}, which is equivalent to $\in^+$ being \fullsetb{},
    which is equivalent to every introspective relation being
    \fullsetb{}.
  \item \label{item:memintroproof} Similar.  \qedhere
  \end{enumerate}
\end{proof}
Now we come to the key result of the section:
\begin{proposition}  \label{prop:biindscaff}
  Each of the following is equivalent to Membership Induction.
  \begin{enumerate}
  \item \label{item:biindscaff} For any class $C$, and any set-based introspective 
    scaffold $(D,<)$ on $C$, the endofunction $\gamdlt$ on $\subclassof{C}$ has a unique fixpoint.
 \item \label{item:biindscafffam} For any class $A$ and $A$-tuple of classes $B$, and any set-based introspective scaffold $(D,<)$ on $A$ with
   functionalization $L$ on $B$, the
   endofunction $\delwl$ on $\classfamoo{B_x}{x \in A}$ has a unique fixpoint.
  \end{enumerate}
\end{proposition}
\begin{proof} \hfill
  \begin{enumerate}
  \item \label{item:biindscaffproof} Membership Induction implies this by \cref{prop:equivintrospect}(\ref{item:memintro}) and
    \cref{prop:gensub}(\ref{item:gensub}), since a least
    prefixpoint that is also a greatest postfixpoint is a unique fixpoint.  For the
    converse, since $\totall$ is a fixpoint of
    $\gamdltof{\totall}{\in}$, it must be the least prefixpoint and
  so  $\in$ is \wellfound{} over it.
     \item \label{item:biindscafffamproof} Membership Induction implies this by \cref{prop:equivintrospect}(\ref{item:memintro}) and
       \cref{prop:genpf}.  For the converse, the scaffold
       $(\totall,\in)$ on $\totall$ has a functionalization $L$ on
       $(1)_{x \in \totall}$
       that at $x$ takes $[*]_{y \in \elset{x}}$ to $*$.  Since the
       partial function $(\totall, x \mapsto *)$ is a fixpoint of
       $\delwlof{\totall}{\in)}{L}$, it must be the least prefixpoint and
       so $<$ is \wellfound{} over $\totall$.  \qedhere
  \end{enumerate}
\end{proof}

\begin{example} \label{example:intronat}
Let us apply  \cref{prop:biindscaff}(\ref{item:biindscaff})
to Example~\ref{example:nat}.  We see that, if Membership
Induction holds, then $\nats$ is the unique $\mathsf{Maybe}$-fixpoint.  
\end{example}

 \psection{Using Ordinals}\label{sect:basicord}
 \psubsection{Set-Based Well-Orderings} \label{sect:setbasedwo}

 Our next task is to give the basic principles of ordinals, which
 requires us to first develop the notion of well-ordering.  As in \Cref{sect:wellfound}, we treat not only relations on a set, but also
   on a class. 

 Let $C$ be a class with a relation $<$.  Recall the notation
 $J(x) \eqdef \setbr{y \smin C \mid y < x}$.  We write $\sqsubseteq$
 for the \emph{extensional 
   preorder}, given by $x \sqsubseteq y \iffdef J(x) \subseteq
 J(y)$.  This is an order iff $J$ is injective, and we then say that $<$ is \emph{extensional}.

A \emph{strict order} on a class $C$ is an irreflexive transitive
  relation $<$.   We write $\leqslant$ for the corresponding order,
  given by $x \leqslant y \iffdef x < y \vee x=y$.  Thus a
  subclass is lower iff it is hereditary.

  A \emph{linear order} on a class $C$ is a relation $<$ such that no
  two elements of $C$ are mutually related and, for any $x,y \smin C$, either
  $x=y$ or $x<y$ or $y<x$.  It follows that $<$ is both extensional and a strict order, with $\sqsubseteq$ and $\leqslant$ coinciding.

  

  
      
Now let us formulate the notion of
well-ordering.\footnote{Cf.~\cite[Page 93]{grayson:thesis}.}
  \begin{proposition} \label{prop:equivwo}
  Let $(C,<)$ be a class equipped with a set-based relation.  The
  following are equivalent:
  \begin{enumerate}
  \item  \label{item:wftrex}  $<$ is \wellfound{}, extensional and transitive.
  \item \label{item:wflin} $<$ is \wellfound{} and, for any $x,y \smin
    C$, either $x=y$ or $x<y$ or $y <x$.
  \item \label{item:strclass}  $<$ is a strict order, and any inhabited subclass has a least
    element.
  \item \label{item:strset} $<$ is a strict order, and any inhabited
    subset has a least element.
  \end{enumerate}
 Moreover, when these conditions hold, the relations $\sqsubseteq$ and $\leqslant$ coincide.
\end{proposition}
\begin{proof}
Firstly, (\ref{item:wflin}) implies (\ref{item:wftrex}) since a
\wellfound{} relation cannot mutually relate two elements, by \cref{prop:wf}(\ref{item:irref}).

  For  (\ref{item:wftrex}) $\Rightarrow$ (\ref{item:wflin})  say that $x,y \smin C$
  are \emph{comparable} when either $x=y$ or $x < y$ or $y < x$.  We claim that, for all $a,b \smin C$, if $a$ is
  comparable with all $y \smin \jlt(b)$, and all $x \smin \jlt( a)$ with $b$, then
  $a$ is with $b$.  It follows by induction that any $a,b \smin C$ are
  comparable.

  To prove the claim, it suffices to show that $a \not< b$ and $b \not< a$
  implies $a = b$.  Any $x \smin \jlt(a)$ is comparable with $b$, and
  is therefore $<b$ as $b \leqslant x$ would
  imply $b < a$.  Thus $\jlt(a) \subseteq \jlt(b)$, and 
  likewise $\jlt(b) \subseteq \jlt(a)$.   Extensionality gives $a=b$.

The rest follows from \cref{prop:wellfounded}, noting that
any \setb{} 
transitive relation is \fullsetb{}.
\end{proof}
A relation satisfying the above conditions is called a
\emph{set-based well-ordering}.   Henceforth, we often abbreviate $(C,<)$ as
$C$.  Now we consider how to compare set-based well-ordered classes.
\begin{proposition} \label{prop:embeddef}
  Let $A$ and $B$ be set-based well-ordered classes.  For a function
  $f \ccolc A \rightarrow B$, the following are equivalent.
  \begin{itemize}
  \item $f$ is an isomorphism from $A$ to a hereditary subclass of $B$.
  \item The square
    \begin{displaymath}
      \xymatrix{
        A \ar[d]_{J} \ar[r]^{f} & B \ar[d]^{J} \\
       \psetset A \ar[r]_{\psetset f} & \psetset B
        }
    \end{displaymath}
    commutes.\footnote{In the language of category theory, this says
      that $f$ is a
      $\pset$-coalgebra map.}  Explicitly: for all $x \smin A$, we have $\jlt(f(x)) =
    \setbr{f(y) \mid y \smin \jlt(x)}$.  
  \end{itemize}
\end{proposition}
\begin{proof}
  ($\Rightarrow$) is straightforward.  For ($\Leftarrow$), we show
  that $f(x) = f(x')$ implies $x=x'$ by induction on $x \smin A$ as
  follows: $f(x) = f(x')$ implies that for all $y \smin \jlt(x)$ there
  is $y' \smin \jlt(x')$ such that $f(y) = f(y')$ and hence $y = y'$
  so $\jlt(x) \subseteq \jlt(x')$, and the reverse inclusion likewise,
  so extensionality of $A$ gives $x=x'$.   
\end{proof}
A function satisfying the above conditions is called an
\emph{embedding}.  
\begin{proposition} \label{prop:embed} \hfill
  \begin{enumerate}
      \item  \label{item:compembed}  For set-based well-ordered classes $A,B,C$, the composite of
    embeddings $A \rightarrow B $ and $ B \rightarrow C$ is an
    embedding $A \rightarrow C$.
   \item \label{item:atmostone} For set-based well-ordered classes $A$ and $B$, there is at
     most one embedding $A \rightarrow B$.
   \end{enumerate}
 \end{proposition}
 \begin{proof} \hfill
  \begin{enumerate}
  \item  \label{item:compembedproof} Follows from the fact that, for any hereditary subclass $X$ of
    $B$, an embedding $B \rightarrow C$ restricts to an embedding 
    $X \rightarrow C$.
   \item\label{item:atmostoneproof}  For embeddings $f,g \ccolc A \rightarrow B$, we show $f(x) =
    g(x)$ by induction on $x \smin A$.   \qedhere
  \end{enumerate}
 \end{proof}
 Next we consider isomorphisms:
 \begin{proposition} \label{prop:isowo}
   Let $A$ and $B$ be set-based well-ordered classes.
   \begin{enumerate}
 \item \label{item:equivisowo} For a function  $f \ccolc A
    \rightarrow B$, the following are equivalent:
    \begin{itemize}
    \item \label{item:iso} $f$ is an isomorphism.
    \item \label{item:surjembed} $f$ is a surjective embedding.
  \item \label{item:mutuembed} $f$ is an embedding, and there is an
    embedding $B \rightarrow
    A$.
  \end{itemize}
     \item \label{item:uniqueisowo} There is at most one isomorphism $A \cong B$.
  \end{enumerate}
\end{proposition}
\begin{proof} \hfill
  \begin{enumerate}
  \item Clearly, the first two conditions are equivalent and imply the
    third.  Lastly, if $f$ and
  $g\ccolc D \rightarrow C$ are embeddings, then \cref{prop:embed}(\ref{item:compembed})
  gives 
  endo-embeddings $g \circ f$ on $C$ and $f \circ g$ on $D$.  By
  \cref{prop:embed}(\ref{item:atmostone}), both  
  are identity maps.
 \item Special case of
    \cref{prop:embed}(\ref{item:atmostone}).   \qedhere
  \end{enumerate}
\end{proof}


\begin{proposition} \label{prop:heredwo}
  Let $C$ be a set-based well-ordered class.
  \begin{enumerate}
  \item \label{item:onlyhered} The only hereditary subclasses of $C$ are $\jlt(x)$, for $x \smin
    C$, and $C$ itself.
  \item \label{item:noniso} These are pairwise non-isomorphic.  In other words:
    \begin{itemize}
    \item For any $x,y \smin C$, if $\jlt(x) \cong \jlt(y)$, then $x=y$.
    \item For any $x\smin C$, we do not have $\jlt(x) \cong C$.
    \end{itemize}
  \end{enumerate}
\end{proposition}
\begin{proof} \hfill
  \begin{enumerate}
  \item Let $X$ be a hereditary subclass.  If it has a strict upper
    bound, it has a strict supremum $x$ and is $\jlt(x)$.  Otherwise
    it is $C$.
   \item For $x \smin C$ and a hereditary subclass $Y$ of $C$, any
     isomorphism   $\theta \ccolc \jlt(x) \cong Y$ is an embedding
     $\jlt(x) \rightarrow C$. By Proposition \ref{prop:embed}(\ref{item:atmostone}), this must be
     the inclusion, so $\jlt(x) = Y$.  Since $x \not\in Y$, we cannot have $Y = C$,
     nor 
     $X = \jlt(y)$ for $y > x$.  Lastly, for $y < x$, we canot have $Y = \jlt(y)$
     since $y \in Y$.   \qedhere
  \end{enumerate}
\end{proof}

\begin{proposition} \label{prop:relatewellorder} \hfill
  \begin{enumerate} 
  \item \label{item:embeddisj} For set-based well-ordered classes $C$ and $D$, either $C$
    embeds into $D$ or vice versa.
      \item \label{item:deflat} Let  $B$ be a subclass of a set-based well-ordered class
        $(C,<)$.  Then $<$ is a set-based well-ordering on
        $B$, and there is a deflationary embedding $B \rightarrow C$.
  \end{enumerate}
\end{proposition}
\begin{proof} \hfill
  \begin{enumerate} 
  \item \label{item:embeddisjproof} Define $R(C,D)$ to be the relation from $C$ to $D$ that relates
    $x \smin C$ to $y \smin D$ when $\jlt(x) \cong \jlt(y)$.  It is an
    isomorphism from a hereditary subclass $X$ of $C$ to a hereditary 
    subclass $Y$ of $D$.  If $X = \jlt(x)$ and $Y = \jlt(y)$, then
    $(x,y) \in R(C,D)$, contradiction.  Therefore either $X=C$ or $Y=D$,
    which gives the two cases.
  \item \label{item:deflatproof} On the subclass $B$, the relation $<$ is set-based
    and \wellfound{} by \cref{prop:reflwf}, and also
    linear, hence a well-ordering.  Obtain
    $R(B,C)$ as before.  Induction shows that, for each $x \smin B$, 
    there is $y\leqslant x$ such that $(x,y) \in R(B,C)$.  Thus $R$ is
    total and deflationary.  \qedhere
  \end{enumerate}
\end{proof}

\begin{proposition} \label{prop:complwell}
  Let $(C,<)$ be a set-based well-ordered class. The following are
  equivalent.
  \begin{enumerate}
  \item \label{item:propclass} $C$ is a proper class.
  \item \label{item:allssup} Every subset of $C$ has a strict supremum.
  \item \label{item:strictsup} Every hereditary subset of $C$ has a strict supremum.
  \item \label{item:alliso} Every set-based well-ordered class embeds
    into $C$.
  \item \label{item:allisoset} Every well-ordered set is isomorphic to $\jlt(x)$ for some 
    $x \smin C$.
  \end{enumerate}
\end{proposition}
\begin{proof}
  Clearly (\ref{item:allssup}) implies (\ref{item:strictsup}).  For
  the converse, given a subset $A$ of $C$, its hereditary closure 
  $\setbr{x \smin C \mid \exists y \smin A.\,x \leqslant y}$ 
  has the same strict upper bounds.
  
  (\ref{item:strictsup}) implies (\ref{item:propclass}),
  because $C$ itself does not have a strict upper bound, so is not a set.

  (\ref{item:propclass}) implies (\ref{item:alliso}) because, for any
  set-based well-ordered class $B$ that does not embed into $C$, \cref{prop:relatewellorder}(\ref{item:embeddisj}) gives an
  embedding $f \ccolc C \rightarrow B$.  The
  range of $f$ is a hereditary 
  subclass of $B$ that is not a set, and therefore is  $B$ by
  \cref{prop:heredwo}(\ref{item:onlyhered}).   So $f$ is an isomorphism $C \cong B$,
  and its inverse is an embedding $B \rightarrow C$, contradiction.

  Clearly (\ref{item:alliso}) implies (\ref{item:allisoset}).

Lastly, (\ref{item:allisoset}) implies (\ref{item:strictsup}) because,
for any hereditary subset $A$ of $C$, there is an isomorphism  
  $f \ccolc {A} \cong \jlt(x)$ for some $x \smin A$, which is the
  identity by \cref{prop:embed}(\ref{item:atmostone}).  So $x$ is a strict supremum of $A$.
\end{proof}
We say that $(C,<)$ is \emph{complete} when it has the above
properties.  
\begin{proposition} \label{prop:compliso}
  Any two complete set-based well-ordered classes are
    uniquely isomorphic.
\end{proposition}
\begin{proof}
  By Propositions~\ref{prop:complwell}(\ref{item:alliso}) and~\ref{prop:isowo}.
\end{proof}

\psubsection{Ordinals} \label{sect:ord}

Our next task is to define $\Ord$. To do this, we write $\transset$
for the class of all transitive sets.  Thus the function $\gamdltof{\transset}{\in}$
sends a class $X$ to the class of all its transitive subsets.  By
\cref{prop:gensub}(\ref{item:gensub}), we can define
\begin{eqnarray*}
  \Ord & \eqdef & \mu \gamdltof{\transset}{\in}
\end{eqnarray*}
Alternatively, we can say that an  
\emph{ordinal} is a transitive set of transitive pure vonniads.
Here is yet another characterization of $\Ord$:
\begin{proposition}  \label{prop:ordcomplete}
    $\Ord$ is the
    unique class of sets $X$ such that $(X,\in)$ is a complete set-based
    well-ordered class.
\end{proposition}
\begin{proof} 
    To prove uniqueness, let $X,Y$ be two such classes.
    \cref{prop:compliso} gives an isomorphism
    $f \ccolc (X,\in) \cong (Y,\in)$.  For all $x \smin X$, we
    have $f(x) = x$, by induction on $x$: since $f(x)$ and $x$ are
    sets that (by the inductive hypothesis) have
    the same elements,
    they are equal.  So $X=Y$.

    Now we show $\Ord$ has the required property.  By inductive inversion, every ordinal is a transitive set of
    ordinals, so membership is extensional and transitive on $\Ord$.
    Furthermore, membership is (set-based and) \wellfound{} on $\Ord$,
    so it is a well-ordering.  Lastly, $(\Ord,\in)$ is complete since
    any transitive set of ordinals is an ordinal and its own supremum.
\end{proof}

It is convenient to refer to a special value
$\biginfty$ that is 
deemed greater than every ordinal. (Some authors call it ``Absolute Infinity''.)   To ensure that it is not an
ordinal, we define $\biginfty
\eqdef \setbr{\setbr{0}}$.  We obtain the ordered class $\extord
\eqdef \Ord \cup \setbr{\biginfty}$, whose elements are called
\emph{extended ordinals}.  For any extended ordinals $\alpha \leqslant \beta$, we define intervals
\begin{eqnarray*}
 {}  (\alpha \twodots \beta) & \eqdef & \setbr{\gamma \smin \Ord \mid
                                     \alpha < \gamma < \beta}
  \\
 {}   [\alpha \twodots \beta) & \eqdef & \setbr{\gamma \smin \Ord \mid
                                      \alpha \leqslant \gamma < \beta} 
 \end{eqnarray*}
We define $0 \eqdef \emptyset$ and, for any ordinal $\alpha$,
define $\mathsf{S} \alpha \eqdef
\alpha \cup \setbr{\alpha}$.  For each class $I$, we have the supremum
function  $\bigvee_I \ccolc \extord^{I} \to \extord$ and the strict
supremum function $\mathsf{ssup}_I \ccolc \Ord^I \rightarrow
\extord$.  They are connected via the equation 
 \begin{math}
  \ssup_{i
  \in I} \alpha_i   =  \bigvee_{i  \in I} \mathsf{S}\alpha_i
\end{math}.






By \cref{prop:heredwo}(\ref{item:onlyhered}), we have a bijection from $\extord$ to the
collection of all lower classes of ordinals, sending $\alpha$ to
$[0\twodots \alpha)$.   The inverse sends a lower class to its strict supremum.

By completeness, any set-based
well-ordered class $C$ is uniquely isomorphic to a lower class of ordinals, whose
strict supremum is called the \emph{order-type} of $C$.  

More generally, let $(C,<)$ be a class equipped with a \sbwf{}
relation.  We recursively define the rank function $\rankf \ccolc C \rightarrow
\Ord$, sending $x$ to $\mathsf{ssup}_{y \in J(x)}
\rankf(y)$, which by induction is also the strict supremum of the
lower set 
\begin{math}
\setbr{\rankf(y) \mid y \smin J^+(x)}
\end{math}.   The range of $\rankf$ is a lower class, and its strict
supremum called the \emph{height} of
$(C,<)$.

As an example, for $m,n \smin \nats$, define $m \prec n
\iffdef m+1 = n$.  Since $\prec$ is a set-based
\wellfound{} relation on $\nats$, we obtain the injection $\iota \ccolc \nats \rightarrow \Ord$
sending 
 $\Zero \mapsto 0$ and $\Succb{n} \mapsto \mathsf{S} (\iota n)$.   The
 height of $(\nats, \prec)$, denoted $\omega$, is also the order-type
 of $(\nats, <)$.  It is  a limit ordinal 
 if Infinity holds, and $\biginfty$ otherwise.

 \psubsection{Inductive Chains for Set-Continuous Functions} \label{sect:inchainsetcont}

 We recall a widely used notion in set theory:
 \begin{definition} \label{def:ascchain}
   An \emph{ascending chain} within an ordered collection $\cata$ is a
   monotone function $\Ord \rightarrow \cata$.  That is, a
   sequence $(x_{\alpha})_{\alpha \in \Ord}$ such that
   $\alpha \leqslant \beta$ implies $x_{\alpha} \leqslant x_{\beta}$.\
 \end{definition}
 
 \begin{proposition} \label{prop:pressup}
  For \standcoll{}s $\catd$ and $\cate$,  a set-continuous function $h \ccolc \catd \rightarrow \cate$
   preserves the supremum of every ascending chain $(x_{\alpha})_{\alpha \in \Ord}$.
 \end{proposition}
 \begin{proof}
   Define $p \eqdef \bigvee_{\alpha \in \Ord} x_{\alpha}$.

 First we show that $\waybel{p} = \bigcup_{\alpha \in
    \Ord} \waybel{x_{\alpha}}$.  We just prove $\leqslant$, as
    $\geqslant$ is obvious.  Suppose $\cate = \classfamoo{B_x}{x
      \in A}$ and $p =
    (N,F)$.  We thus have $N = \bigcup_{\alpha \in \Ord}
    M_{\alpha}$, where for each $\alpha \smin \Ord$ we have $x_{\alpha} = (M_{\alpha},F \restriction_{M_{\alpha}})$.   Given an element $y \smin
    \waybel{p}$, we have $y = (K, N \restriction_{K})$,
    for some subset $K$ of $N$. For each $k \smin
    N$, let $\overline{k}$ be
the least ordinal $\beta$ such that   $k \in M_{\beta}$.  Put $\alpha \eqdef
      \bigvee_{k \in K} \overline{k}$, and we see that $K \subseteq 
      M_{\alpha}$.  Hence $y \in \waybel{x_{\alpha}}$ as required.

        It follows that
         \begin{eqnarray*}
       h(p) & = & \bigvee_{y \in \subwaybel{p}} h(y) \\
       & = & \bigvee_{\alpha \in \Ord} \bigvee_{y\in\subwaybel{x_{\alpha}}} h(y) \\
       & = & \bigvee_{\alpha \in \Ord} x_{\alpha}
     \end{eqnarray*}
 \end{proof}

The notion of ascending chain is used as follows.
 \begin{definition} \label{def:indchain}
   Let $h$ be a set-continuous
   endofunction on a \standcoll{} $\cate$.   An \emph{inductive chain} for $h$ is an ascending chain $(x_{\alpha})_{\alpha \in \Ord}$ within $\cate$ such that
   \begin{spaceout}{rcll}
     x_{0} & = & \bot \\
    x_{\mathsf{S}\alpha} & = & h(x_{\alpha}) & \text{ for any
       ordinal
       $\alpha$} \\
     x_{\alpha} & = & \bigvee_{\beta < \alpha} x_{\beta} & \text{
       for any limit ordinal $\alpha$.}
   \end{spaceout}%
   Equivalently: such that
   $X_{\alpha} = \bigvee_{\beta < \alpha} h(\xff{\beta})$, for every
   ordinal $\alpha$.  
 \end{definition}
Clearly, there is at most one inductive chain.    It it exists, it is written $(\xf{\alpha})_{\alpha \in \Ord}$ and its supremum
$\xf{\subbiginfty}$.  Existence is unclear in general, since we
cannot recursively define a sequence of classes, but we shall
see various cases where it does exist:
\begin{itemize}
\item if $h$ has a small prefixpoint (Proposition~\ref{prop:stabord})
\item if $h$ preserves smallness (Proposition~\ref{prop:pressmall})
\item if Collection holds (Proposition~\ref{prop:collmu})
\item if $h$ arises from a scaffold (Proposition~\ref{prop:scaffclasschain}), or scaffold with
  functionalization (Proposition~\ref{prop:scaffpfnchain}).
\end{itemize}
The key property of an inductive chain is that it yields a least prefixpoint:
\begin{proposition} \label{prop:stabbig}
 Let $h$ be a set-continuous
   endofunction on a \standcoll{} $\cate$.  If $h$ has an inductive chain, then $\xf{\subbiginfty}$ is a least $h$-prefixpoint.
\end{proposition}
\begin{proof} 
   Applying Proposition~\ref{prop:pressup} to the inductive chain tells us that $\xf{\subbiginfty}$ is $h$-fixed.    To show leastness, let $z$ be a prefixpoint.  Induction on $\alpha
      \smin \Ord$ gives $\xff{\alpha} \leqslant z$.  So
      $\xff{\subbiginfty} \leqslant z$.
    \end{proof}

 For an extended ordinal $\alpha$, an inductive chain
 \emph{stabilizes} at $\alpha$ when $\xf{\alpha}$ is $h$-prefixed, or
 equivalently when $\xf{\alpha} = \xf{\subbiginfty}$.  Here is an
 application of this notion:
 \begin{proposition} \label{prop:stabord}
   For a set-continuous
   endofunction $h$ on a \standcoll{} $\cate$, the following
   are equivalent.
   \begin{itemize}
   \item $h$ has a small prefixpoint.
   \item $h$ has a small least prefixpoint.
   \item $h$ has an inductive chain within $\fund{\cate}$ that
     stabilizes at an ordinal.
   \end{itemize}
 \end{proposition}
 \begin{proof}
   If $h$ has an inductive chain within $\fund{\cate}$ that stabilizes
   at an ordinal $\alpha$, then $\xf{\alpha}$ is a small least
   prefixpoint.  Conversely, suppose $h$ has a small prefixpoint $x$.
   Then we can define the inductive chain by \wellfound{} recursion as
   a function $\Ord \rightarrow \waybel{x}$.  For stabilization, we
   prove a general fact: for any   any ascending chain
  $(x_{\alpha})_{\alpha \in \Ord}$ in $\cate$,  if $\bigvee_{\alpha
    \in \Ord} p \in \fund{\cate}$,  then there is
  $\beta \smin \Ord$ such that $(xf{\alpha})_{\alpha \geqslant
    \beta}$ is constant.   Suppose $\cate = \classfamoo{B_x}{x \in
    A}$, and $\bigvee_{\alpha \in \Ord} x_{\alpha} = (K,F)$. By hypothesis, $K$ is a set.  For each ordinal $\alpha$, we have
  $x_{\alpha} = (M_{\alpha}, F\restriction_{M_{\alpha}})$ for a subset
  $M_{\alpha}$ of $N$. For each $k \smin K$, let $\overline{k}$ be the
  least ordinal $\alpha$ such that $k \in M_{\alpha}$.  Then $\beta
  \eqdef \bigvee_{k \in K} \overline{k}$ has the required property.
 \end{proof}

Here is the most important case (for our purposes) where inductive chains exist:
\begin{proposition}  \label{prop:pressmall} Let $\cate$ be a \standcoll{}, and $h$ a set-continuous endofunction on $\cate$ that preserves smallness.   
  \begin{enumerate}
  \item \label{item:pressmallchain} $h$ has an inductive chain within
    $\fund{\cate}$, and a least prefixpoint.
  \item\label{item:pressmallstab} The following are equivalent:
    \begin{itemize}
    \item $h$ has a small prefixpoint.
    \item $\mu h$ is small.
    \item The inductive chain stabilizes at an ordinal.
    \end{itemize}
  \end{enumerate}
\end{proposition}
\begin{proof} \hfill
  \begin{enumerate}
  \item  \label{item:pressmallchainproof}  We define the inductive chain by \wellfound{} recursion as a
    function $\Ord \rightarrow \fund{\cate}$.  
  \item  \label{item:pressmallstabproof} By Proposition~\ref{prop:stabord}.
    \qedhere
  \end{enumerate}
\end{proof}

As an example of \cref{prop:pressmall}(\ref{item:pressmallchain}), if Powerset is assumed, then $\pset$ has an inductive chain consisting of sets.  This is the well-known 
cumulative hierarchy.

 We see next that  Collection guarantees the existence of inductive
 chains.   This is adapted from~\cite[Theorem 1]{Takahashi:inductionset}
 and~\cite[Theorem 6.4]{Aczel:book} and~\cite[Theorem 5.1]{AczelRathjen:notescst}.
 \begin{proposition} \assumi{Collection}
   \label{prop:collmu}
   Let $h$ be a set-continuous endofunction on a \standcoll{}
   $\cate$.  Then $h$
   has an inductive chain and a least prefixpoint.
 \end{proposition}
 \begin{proof}
   
  Via the bijection  $\cate \cong \classfamoo{B_x}{x \in A}$ we identify
   elements of 
   $\cate$ with subclasses $W$ of $\sum_{x\in A} B_x$ that are
   \emph{single-valued}, meaning that, for all $x \smin A$, there is at
   most one $y \smin B_x$ such that $\tuple{x,y} \in W$.

   A subclass $M$
  of $\Ord \times \tesum_{x \in A} B_x$ represents a sequence
  $(M_{\alpha})_{\alpha \in \Ord}$ of subclasses of $\sum_{x \in
       A}B_x$ via $M_{\alpha} \eqdef \setbr{x \smin C \mid
       \tuple{\alpha,x} \in M}$.  We say that $M$ is \emph{good} when
     $(M_{\alpha})_{\alpha \in \Ord}$ is ascending and, for any ordinal
     $\alpha$ and $\tuple{x,y} \smin M_{\alpha}$, there is $\beta <
     \alpha$ and a single-valued subset $W$ of $M_{\beta}$ such that $\tuple{x,y}
     \in  h(W)$. 

  For a good class $M$, we show by induction on  $\alpha \smin
  \Ord $ that the class  
  $M_{\alpha}$ is single-valued, as follows.  For $x \smin A$ and
  $y,y' \smin B_x$, suppose $\tuple{x,y}$ and $\tuple{x,y'}$ are in $M_{\alpha}$.  We obtain $\beta <
     \alpha$ and a single-valued subset $W$ of $M_{\beta}$ such that $\tuple{x,y}
     \in  h(W)$, and likewise $\beta' <
     \alpha$ and a single-valued subset $W'$ of $M_{\beta'}$ such that $\tuple{x,y'}
     \in  h(W')$.  Without loss of generality we have $\beta
     \leqslant \beta'$, so $W'$ is a subset of $M_{\beta'}$, and
     hence $W\cup W'$ is too.  By the inductive hypothesis at
     $\beta'$, the class 
     $M_{\beta}$ is single-valued, so the set $W \cup W'$ is too.  Next,
     $\tuple{x,y} \in h(W)$ gives $\tuple{x,y} \in h(W \cup W')$, and
     likewise $\tuple{x,y'} \in h(W \cup W')$.  So $y=y'$.

 Clearly, any union of good classes is good; in
     particular, the union  $R$ of all good sets.    We show that, for any ordinal $\alpha$, we
have
\begin{math}
  R_{\alpha}  =  \bigvee_{\beta < \alpha}  h(R_{\beta})
\end{math}.  
so that $(R_{\alpha})_{\alpha \in \Ord}$ is an inductive chain.  We just prove $\geqslant$, as $\leqslant$ is obvious.  Given $\beta <
\alpha$ and $\tuple{x,y} \smin h(R_{\beta})$, there is a subset $K$ of $R_{\beta}$
such that $\tuple{x,y} \in h(K)$.  For each $k \smin K$, we have $k
\in R_{\beta}$ since $K \subseteq R_{\beta}$, so the class $T_{k}$ of 
all good sets $X$ such that
$k \in X_{\beta}$ is inhabited.   By
Collection, there is $S \smin \prod_{k \in K}\psetinh T_k$.  For each
$k \smin K$, the set $L_k \eqdef \bigcup S_k$ is a member of $T_k$.  The set
$N \eqdef \bigcup_{k \in K} L_k$ is good; and $K \subseteq N_{\beta}$
since, for all $k \smin K$, we have 
$k \in (L_k)_{\beta} \subseteq N_{\beta}$.  Putting $N' \eqdef \setbr{\tuple{\alpha,\tuple{x,y}}} \cup
N$, we have $N'_{\beta} = N_{\beta}$, so  $\tuple{x,y} \in h(K) \subseteq h(N_{\beta}) = h(N'_{\beta})$, so $N'$ is
good.  Thus $\tuple{x,y} \in N'_{\alpha} \subseteq R_{\alpha}$ as
required.

The supremum $\bigvee_{\alpha \in \Ord}$ is a least prefixpoint
by \cref{prop:stabbig}.
 \end{proof}

\Cref{prop:gensub}(\ref{item:cogensub}) and \cref{prop:collmu} suggest the following question: for a class $C$, does every 
set-continuous endofunction on $\subclassof{C}$ have a greatest
postfixpoint?  The question is
addressed in~\cite[Theorem 6.5]{Aczel:book}  and~\cite{Aczel:relreflect}.

\psubsection{Inductive Chains for Scaffold Functions} \label{sect:indchainscaff}

We see next that every function arising from a set-based scaffold has
an inductive chain.
\begin{proposition}  \label{prop:scaffclasschain}
  Let $(D,<)$ be a  set-based
  scaffold on a class $C$, generating $M \smin \subclassof{C}$. 
  \begin{enumerate}
  \item \label{item:scaffchain} For each extended ordinal $\alpha$, we
    define  \begin{math}
      M_{\alpha}  \eqdef  \setbr{x \smin M \mid \rho(x) < \alpha}
     \end{math}.  Then $\Gamma_{(D,<)}$ has inductive chain
     $(M_{\alpha})_{\alpha \in \Ord}$ with supremum $M_{\subbiginfty}$.  
   \item \label{item:scaffchainstable} The height of $(M,<)$ is the least extended ordinal at
     which the inductive chain stabilizes.
   \item \label{item:scaffchainsmall} If $\Gamma_{(D,<)}$ preserves
     smallness, then $M$ is a set iff the height of
     $(M,<)$ is an ordinal.
  \end{enumerate}
\end{proposition}
\begin{proof} \hfill
  \begin{enumerate}
  \item \label{item:scaffchainproof} The requirements for zero, limit ordinals and $\biginfty$ are
    obvious.  The requirement for $M_{\mathsf{S}\alpha}$ holds because
    it says that an element of $M$ with rank $\leqslant \alpha$ is the
    same thing as an element of $D$ whose children are all in $M$ and
    have rank $< \alpha$.
  \item \label{item:scaffchainstableproof}  Since the height is the strict supremum of the ranks in $M$.
 \item \label{item:scaffchainsmallproof} By part~(\ref{item:scaffchainstable}) and Proposition~\ref{prop:pressmall}.
   \qedhere
  \end{enumerate}
\end{proof}
As an example of
Proposition~\ref{prop:scaffclasschain}(\ref{item:scaffchain})--(\ref{item:scaffchainstable}), we
see that $\psetset$ has an inductive chain of classes, which does not stabilize at any ordinal, since every ordinal is
its own rank.   Its supremum (union) is $\wfpure$.  Another example: for
any broad arity $F$, we see that $\incf$ has an inductive chain, whose
supremum is $\redbroad{F}$.

Here is the analogue of Proposition~\ref{prop:scaffclasschain} for partial
functions:
\begin{proposition} \label{prop:scaffpfnchain}
  Let $(D,<)$ be a set-based scaffold on a class $A$ with
    functionalization $L$ on an $A$-tuple of classes $B$, generating
    $(M,F) \smin \classfamoo{B_x}{x \in A}$.
    \begin{enumerate}
    \item \label{item:pfnchain} For each extended ordinal $\alpha$, we
      define $M_{\alpha}  \eqdef  \setbr{x \smin M \mid \rho(x) <
        \alpha}$ and $p_{\alpha} \eqdef  (M_{\alpha},F
      \restriction_{M_{\alpha}}) $.  Then $(p_{\alpha})_{\alpha \in Ord}$ is an inductive
    chain for 
    $\Delta_{(D,<)}^{L}$ with supremum $p_{\subbiginfty}$.
   \item \label{item:pfnstable} The height of $(M,<)$ is the least extended ordinal at which
     the inductive chain stabilizes.  
   \item \label{item:pfnsmall} If $\Delta_{(D,<)}^{L}$ preserves smallness, then $M$ is a
     set iff the height of
     $(M,<)$ is an ordinal.
    \end{enumerate}
\end{proposition}
\begin{proof}
  Similar to the proof of Proposition~\ref{prop:scaffclasschain}
\end{proof}

\psection{Category Theory} \label{sect:cattheory}
\psubsection{Categories and Functors} \label{sect:catfunc}

To help organize parts of our story, we use some category theory:
\begin{itemize}
\item The fact that initial algebras of isomorphic functors are
  isomorphic.
\item The notion of an ``algebraically least'' prefixpoint.
\item The fact that class summation preserves connected limits.
\end{itemize}

For the sake of this paper, a
 \emph{category} $\catc$ consists of a collection $\mathsf{ob}\; \catc$
 and an indexed family of collections $(\catc(x,y))_{x, y \in \mathsf{ob}\;\catc}$, together
 with composition and identity morphisms satisfying
 the usual three laws.  For example:
 \begin{itemize}
 \item $\bfclass$ is the category of all classes and functions.
 \item For a class $E$, we write  $\classfamcat{E}$ for the category of all class-families within $E$
   and maps between them.
 \end{itemize}
 We also consider the following functors.
 \begin{itemize}
 \item $\pset$ is an endofunctor on $\bfclass$, sending  $f \ccolc A \rightarrow B$ to the function 
  \begin{eqnarray*}
 \psetset A & \rightarrow & \psetset B \\
U &  \mapsto & \setbr{f(x) \mid x \smin
    U}
  \end{eqnarray*}
\item $\incs$ is an endofunctor on $\bfclass$.
 \item Any class $I$ gives a functor $\sum \ccolc
   \bfclass^{I} \rightarrow \bfclass$.
  \item Any set $K$ gives a functor $\prod \ccolc \bfclass^{K}
    \rightarrow \bfclass$.
 \end{itemize}
As stated in \Cref{sect:classes}, all talk about
 collections---and therefore all talk about categories---is informal.

 \psubsection{Initial Algebras} \label{sect:initalg}
  

Since initial algebras frequently appear in our story, we present the
key notions.
 \begin{definition} \label{def:algebra}
   Let $F$ be an endofunctor on a category $\catc$.
   \begin{enumerate}
     \item \label{item:falg} An \emph{$F$-algebra} $(x,\theta)$ consists of a $\catc$-object $x$ (the
     \emph{carrier}) and morphism $\theta \ccolc Fx \rightarrow x$
     (the \emph{structure}).
     \item \label{item:algmap} Given $F$-algebras $(x,\theta)$ and
     $(y,\phi)$, a \emph{map} $(x,\theta) \rightarrow (y,\phi)$ is a
     $\catc$-morphism $f \ccolc x \rightarrow y$ such that the square
     \begin{math}
       \xymatrix{
         Fx \ar[d]_{\theta} \ar[r]^-{Ff} & Fy \ar[d]^{\phi} \\
         x\ar[r]_-{f} & y
       }
     \end{math}
     commutes.
   \end{enumerate}
 \end{definition}
An $F$-algebra $(x,\theta)$ is \emph{initial} when, for every
$F$-algebra $(y,\phi)$, there is a unique map $(x,\theta) \rightarrow
(y,\phi)$.  It follows that $\theta$ is an
   isomorphism; this fact is called \emph{Lambek's lemma}, and
   inductive inversion is a special case.  Here are some examples of
   initial algebras.
   \begin{itemize}
\item The endofunctor $\psetset$ on $\bfclass$ has initial algebra $(\wfpure,\id_{\wfpure})$.
\item The endofunctor $\incs$ on 
  $\bfclass$ has initial
  algebra $(\nats,\id_{\nats})$. 
\end{itemize}
These statements hold because $\wfpure$ and $\nats$ are each
equipped with a \wellfound{}
relation by \cref{prop:gensub}(\ref{item:gensub}).
So we can use \cref{prop:wfrecur}(\ref{item:wfrecurtot})
to construct a unique algebra map to any algebra.

We next see that initial algebras of isomorphic endofunctors are
isomorphic.  Recall that, for categories $\catc, \catd$ and functors
$F,G \ccolc \catc \rightarrow \catd$, a \emph{natural isomorphism}
$\alpha \ccolc F
\cong G$ associates to each object $x \smin \catc$ a
$\catd$-isomorphism $\alpha_x \ccolc Fx \cong Gx$ in such a way that,
for each $\catc$-morphism $f \ccolc x \rightarrow y$, the square
\begin{math}
  \xymatrix{
    Fx \ar[r]_-{\alpha_x} \ar[d]_{Ff} & Gx  \ar[d]^{Gf} \\
    Fy \ar[r]^-{\alpha_y} & Gy
    }
  \end{math} commutes.
    \begin{proposition} \label{prop:initalgiso}
     On a category $\catc$, let $F$ and $G$ be endofunctors with
     initial algebras $(x,\theta)$ and $(y,\phi)$ respectively.  Then
     any natural isomorphism $\alpha
\ccolc F \cong G$ induces an isomorphism $x \cong y$.  
   \end{proposition}
   \begin{proof}
     Let $f$ be the unique $F$-algebra map  $(x,\theta) \rightarrow (y,
      \phi \circ \alpha_y)$ and $g$ the unique $G$-algebra map $(y,\phi) \rightarrow
      (x,\theta \circ \alpha^{-1}_x)$.  The diagram
          \begin{math}
       \xymatrix{
         Fx \ar[r]^{Ff} \ar[ddd]_{\theta} & Fy \ar[d]_{\alpha_y} 
         \ar[r]^{Fg} & Fx \ar[d]_{\alpha_x} \ar@/^{1pc}/[dd]^{\id_{Fx}} \\
         & Gy \ar[dd]_{\phi} \ar[r]^{Gg} & Gx \ar[d] _{\alpha^{-1}_x} \\
         & & Fx \ar[d]^{\theta} \\
         x \ar[r]_{f} & y \ar[r]_{g} & x
         }
       \end{math} commutes,  so $g \circ f$ is an $F$-algebra endomap
       on $(x,\theta)$.  
 Since $\id_x$ is too, they are equal.  Likewise $f \circ g = \id_y$.
   \end{proof}

   \psubsection{Algebraically Least Prefixpoints} \label{sect:algleast}

 Observe that our examples $\wfpure$ and $\nats$ serve as both least
 prefixpoints and initial algebras.  We shall now formulate this
 situation in general.
   \begin{definition} \label{def:ordcat}
     An \emph{ordered category} consists of
     \begin{itemize}
     \item a category $\catc$
     \item an order $\leqslant$ on the collection $\ob \catc$
     \item for each pair of objects  $x \leqslant y$, an \emph{inclusion}
     morphism $\iof{x}{y} \ccolc x \rightarrow y$
   \end{itemize}
 For each object $x$ we must have $\iof{x}{x} = \id_x$, and 
  for any objects $x \leqslant y \leqslant z$  the triangle 
   \begin{math}
     \xymatrix{
       x \ar[r]^{\iof{x}{y}} \ar[dr]_{\iof{x}{z}} & y \ar[d]^{\iof{y}{z}} \\
       & z
       }
      \end{math} must commute.
    \end{definition}
\noindent    Our main examples of ordered categories are $\bfclass$ and
    $\classfamcat{E}$ for
    any class $E$.
\begin{definition} \label{def:monotonefunc}
  Let $\catc$ and $\catd$ be ordered categories.  A functor $F \ccolc \catc
\rightarrow \catd$ is \emph{monotone} when, for any $\catc$-objects $x
\leqslant y$, we have $Fx \leqslant Fy$ and $F\iof{x}{y} =
\iof{Fx}{Fy}$.
\end{definition}
Now let $F$ be a monotone endofunctor on an ordered category $\catc$.
Any $F$-prefixpoint $x$ gives an $F$-algebra $\algof{x} \eqdef (x,
\iof{Fx}{x})$.  Furthermore, for $F$-prefixpoints $x \leqslant y$, we
have an algebra map $\iof{x}{y} \ccolc \algof{x} \rightarrow
\algof{y}$.  
\begin{definition} \label{def:algleast}
  A least $F$-prefixpoint $a$ is \emph{algebraically least} when the $F$-algebra $\algof{a}$ is initial.
\end{definition}
For example:
\begin{itemize}
\item $\psetset$ is a monotone endofunctor on $\bfclass$, and $\wfpure$ is its algebraically least prefixpoint.
\item $\incs$ is a monotone endofunctor on $\bfclass$, and
  $\nats$ is its algebraically least prefixpoint.
\end{itemize}
Here is a final observation (not used in the sequel).   As stated
above, 
our main examples are where $\catc$ is either $\bfclass$ or 
$\classfamcat{E}$ for a class $E$.  In these cases, the requirement in
Definition~\ref{def:algleast}  for $a$ to
be least is redundant.  For suppose that $a$ is an
$F$-prefixpoint such that $\algof{a}$
is initial.  To show leastness of $a$, it suffices to show minimality, since
the ordered collection $(\ob \catc, \leqslant)$ has binary meets.  So
suppose $x$ is an $F$-prefixpoint such that $x \leqslant a$.  Since
$\algof{a}$ is initial, the
algebra morphism $\iof{x}{a} \ccolc \algof{x} \rightarrow \algof{a}$
has a section, so it is surjective, giving $x=a$.

   \psubsection{Connected Limits} \label{sect:connlimits}

 Before we can discuss limits, we have to formulate the notion of a diagram.
   
 A \emph{quiver} (also called a directed multigraph) consists of a set
of nodes and a set of edges, with each edge having a source node and a
target node. For nodes $m,n$ we write $m \leftrightarrow n$ when
 either an edge $m \rightarrow n$ or an edge $m \rightarrow
n$ exists.  A quiver $\bbI$ is
\emph{connected} when there is a node $a$ such that every node
$m$ satisfies $a \leftrightarrow^* m$.

Let $\bbI$ be a quiver.  For a category $\catc$, an \emph{$\bbI$-indexed diagram} $D$ in
  $\catc$ consists of
 \begin{itemize}
 \item an object $D_m$ for each node $m$
 \item a morphism
   $D_f \ccolc D_m \rightarrow D_n$ for each edge
   $f \ccolc m \rightarrow n$.
 \end{itemize}
In the case that $\catc = \bfclass$, we write $\lim_{m \in \bbI} D_m$ for the
\emph{limit} of $D$, i.e.,  the class of  all $x \smin \prod_{m \in \bbI} D_m$
 satisfying $D_f(x_m) = x_n$ for every edge $f \ccolc m \rightarrow
  n$.   Now we come to the promised result:
 

  \begin{proposition} \label{prop:sumconn}
    For any class $I$, the functor $\sum \ccolc
   \bfclass^{I} \rightarrow \bfclass$ preserves connected limits.
   More precisely, for a connected quiver $\bbI$ and $\bbI$-indexed
   diagram $D$ in $\class^I$, the map
   \begin{displaymath}     \alpha \ccolc \sum_{i \in I} \lim_{m \in \bbI}
   \dof{i}{m} \rightarrow \lim_{m \in \bbI} \sum_{i \in I} \dof{i}{m}
   \end{displaymath}
   sending $\tuple{i,[x_m]_{m \in \bbI}}$ to $[\tuple{i,x_m}]_{m \in
     \bbI}$ is bijective.\footnote{This result is an instance
     of~\cite[Theorem 4.2]{nlab:connectedlimit} via the 
equivalence
\begin{displaymath}
\bfclass^{I} \simeq \classfamcat{I} = \bfclass / I
\end{displaymath}
that sends $B$ to the
class-family $(i)_{\tuple{i,x} \in \sum_{i \in I} B_i}$.}
\end{proposition}
\begin{proof}
  Let $y \smin \lim_{m \in \bbI} \sum_{i \in I} \dof{i}{m}$.  For any
  edge $f \ccolc m
  \rightarrow n$, we see that $y_m$ and $y_n$ have the same left
  component, since 
  \begin{eqnarray*}
    (\sum_{i \in I} \dof{i}{f})(y_m) & = & y_n
  \end{eqnarray*}
  Hence, for any nodes $m$ and $n$ such that $m \leftrightarrow^* n$,
  we see that 
  $y_m$ and $y_n$ have the same left component.  By hypothesis,
there is a node $a$ such, for every node $m$, we have 
 $a \leftrightarrow^* m$.  Write $i$ for the left component of
$y_a$.  Then for each node $m$, we can express $y_m$ as
$\tuple{i,x_m}$.  We obtain $x \in \lim_{m \in \bbI} \dof{i}{m}$, since
for any edge $f \ccolc m
\rightarrow n$, we have
\begin{eqnarray*}
    (\sum_{i \in I} \dof{i}{f})(y_m) & = & y_n \\
  \text{i.e.,\ }\quad \tuple{i, \dof{i}{f}(x_m)} & = & \tuple{i,x_n} \\
 \text{so} \quad \dof{i}{f}( x_m) & = & x_n
  \end{eqnarray*}
 So $\tuple{i,[x_m]_{n \in \bbI}}$, which is the unique
 $\alpha$-preimage of $y$, is in $\sum_{i \in I} \lim_{m \in \bbI}
   \dof{i}{m}$.  
 \end{proof}
 
 \ppart{Wide and Broad principles for sets} \label{part:widebroadsets}

\psection{Infinity Principles}  \label{sect:infprin}
  \psubsection{Wide Infinity} \label{sect:wideinf}

  Let us review what we have seen previously.  The class $\nats$ of
  all natural numbers is the algebraically least
  $\incs$-prefixpoint.  The axiom of Infinity says that a
  $\incs$-prefixed set exists, which is equivalent to $\nats$ being a
  set.

 Now we continue.  A set $K$ (which in this context may be called an \emph{arity}) gives rise to a monotone endofunctor $\inck$ on
 $\bfclass$, sending $X$ to $\incof{X^K}$.  Thus a class $X$ is
 $\inck$-prefixed iff it contains $\Zero$ and, for any $K$-tuple $y$
 within $X$, contains $\Succb{y}$.
 
 The algebraically least such class is called the \emph{class of all simple $K$-wide
  numbers}, denoted by $\redwide{K}$.  To show it exists, 
note that $\incs$ arises
from the following set-based scaffold on $\totall$: a parent is either $\Zero$, which has no children, or
$\Succb{y}$, for any $K$-tuple $y$, in
which case the set of children is $\setbr{y_k \mid k \smin K}$.   So
\cref{prop:gensub}(\ref{item:gensub}) 
gives a least $\inck$-prefixed class, and
\cref{prop:wfrecur}(\ref{item:wfrecurtot}) gives the
initial algebra property.

\begin{example} \label{example:kwide}
Define $K$ to  be the arity $\setbr{0,1,2}$.  The following are 
  simple $K$-wide numbers:
  \begin{itemize}
  \item
    \begin{math}
      \Zero
    \end{math}
 \item
   \begin{math}
      \Succj
      \begin{bmatrix}
        \Zero \\
        \Zero \\
        \Zero 
      \end{bmatrix}
    \end{math}
  \item
    \begin{math}
      \Succj
      \begin{bmatrix}
        \Zero \\
        \Zero \\
   \Succj
      \begin{bmatrix}
        \Zero \\
        \Zero \\
        \Zero 
      \end{bmatrix}  
      \end{bmatrix}
    \end{math}
  \end{itemize}
\end{example}
\noindent We can visualise a wide number as  a two-dimensional \wellfound{} tree.  For example,
the last number in Example~\ref{example:kwide} is
      visualized in Figure~\ref{fig:kvisual}, using
        the vertical dimension for
        $\begin{bmatrix}{\vdots}\end{bmatrix}$ 
        and the horizontal dimension for internal structure, with the
        root appearing at the left and the $\mathsf{Nothing}$-marked
        leaves at the right.
 \begin{figure}  
   \begin{displaymath}
     \xymatrix{
       & \Zero &  \\
  \Succj  \ar@{-}[ur]^{0} \ar@{-}[r]^{1} \ar@{-}[dr]_{2}   & \Zero & \Zero \\
       & \Succj  \ar@{-}[ur]^{0} \ar@{-}[r]^{1} \ar@{-}[dr]_{2}  &  \Zero \\
       & & \Zero  
       }
        \end{displaymath}
        \caption{Visualization of a simple wide number}\label{fig:kvisual}
      \end{figure}
      The axiom of \emph{Simple Wide Infinity} says, for every
      arity 
$K$, that a $\inck$-prefixed set exists; this is equivalent to $\redwide{K}$ being a set.  

There is an alternative version, formulated as follows.

A \emph{signature} is a family of sets  $S= (K_i)_{i \in I}$, where 
we call $i \smin I$
a \emph{symbol} and the set $K_i$ its \emph{arity}.  It gives rise to a monotone endofunctor $H_S$ on $\bfclass$, sending
$X$ to $\sum_{i \in I}X^{K_i}$.   Thus a class $X$ is
 $H_S$-prefixed iff, for any $i \smin I$ and $K_i$-tuple $y$ within
 $X$, it contains  $\tuple{i,y}$.  As before, there is an 
 algebraically least such class, which we call the  \emph{class of all $S$-wide
  numbers}, denoted by $\wide{S}$.  

The axiom of \emph{\full{} Wide Infinity} says, for every signature
$S$, that an $H_S$-prefixed set exists; this is
equivalent to $\wide{S}$ being a set.  Previously this axiom has appeared
  in \cite[page 15]{vandenBerg:hab} under the name ``Smallness of W-types'',
 alluding to the notion of W-type in type theory~\cite{MartinLoef:inttypetheory,MoerdijkPalmgren:wellfoundtreecat,AbbottAltenkirchGhani:containers}.

 

Before giving the main result of the section, we give a 
result that allows us to work with injections rather than
inclusions:
\begin{proposition} \label{prop:injalg}  \hfill
  \begin{enumerate} 
  \item \label{item:injinf} Infinity is equivalent to \emph{Dedekind Infinity}: There is a set $X$ and
    injection $\incof{X} \rightarrow X$.\footnote{For a class $X$,  an
      injection $\incof{X} \rightarrow X$  corresponds to a unary
      Dedekind encoding on $X$.  When such an injection  
      exists,  $X$ is said to be \emph{Dedekind-infinite.}.}
  \item Let $K$ be a set.  Then $\redwide{K}$ is a set iff there is
    a set $X$ and injection $\inck X \rightarrow X$.
  \item \label{item:injalgsig} Let $S$ be a signature.  Then
    $\wide{S}$ is a set iff there
    is a set $X$ and injection $H_S X \rightarrow X$.
  \end{enumerate}
\end{proposition}
\begin{proof}
  We just prove part~(\ref{item:injalgsig}), as the rest is
  similar.  Any $H_S$-prefixed class $X$, such as $\mu H_S$, is equpped with an injection
  $\iof{H_SX}{X} \ccolc H_SX \rightarrow X$.   Conversely, given a class $X$ and
  injection $\theta \ccolc H_S X \rightarrow X$, the unique algebra morphism
  $f \ccolc \algof{\mu H_S} \rightarrow (X,\theta)$ is injective.  
  (For this says that any $x \smin \mu H_S$ is the only $f$-preimage of its $f$-image,
  which is proved by induction on $x$.)  
  So if $X$ is a set,
  then $\mu H_S$ is a set.  
\end{proof}

\begin{theorem} \label{prop:powinfty}
  Simple Wide Infinity,  Full Wide Infinity and Exponentiation $+$ Infinity
  are all equivalent.
\end{theorem}
\begin{proof}
  To show that Full Wide Infinity implies Simple Wide Infinity: given a set
  $K$, define $S$ to be the signature
    $(L_i)_{i \in \setbr{0,1}}$ where $L_0 \eqdef \emptyset$ and $L_1
    \eqdef K$.  We have a natural isomorphism $\inck \cong H_S$, so
    \cref{prop:initalgiso} tells us that 
    $\redwide{K}$ is a set iff $\wide{S}$ is a set.

To show that Simple Wide Infinity implies Infinity, we have a natural
isomorphism $\incs \cong \mathsf{Maybe}_1$, so \cref{prop:initalgiso} tells us that $\nats$ is a set iff
$\redwide{1}$ is a set.
  
  To show that  Full Wide Infinity implies Exponentiation, let $A$ and $B$ be sets.  Define the signature $S$ with
nullary symbols $(\mathsf{Leaf}_b)_{b \in B}$ and an $A$-ary symbol
$\mathsf{Node}$. 
 Consider the injection $B^A
\rightarrow \wide{S}$ sending a function $f$ to the $S$-wide number
\begin{displaymath}
  \tuple{\mathsf{Node},[\tuple{\mathsf{Leaf}_{f(a)},[\,]}]_{a \in A}}
\end{displaymath}
Since $\wide{S}$ is a set, $B^A$ is a set.

To show that Exponentiation$+$Infinity implies Full Wide Infinity, let $S$
be a signature.
By Exponentiation, $H_S$ restricts to an
endofunctor on $\mathbf{Set}$.  Following~\cite{Barr:terminalcoalg}, we form the $\omega^{\op}$-chain 
  \begin{displaymath}
    \xymatrix{
 1 & H_S1 \ar[l]_-{\tuple{}} & H_S^21  \ar[l]_-{H_S\tuple{}} & {} \ar[l]_-{H_S^2\tuple{}} \, \cdots
      }
    \end{displaymath}
  which may be called the ``coinductive chain''.   (Intuitively $H_S^n 1$ is the set
    of $S$-trees with stumps at level $n$.)  Let $M$ be the limit and
    $\theta \,\colon\,H_S M \to M$ the canonical map.  (Intuitively
    $M$ is the set of non-\wellfound{} $S$-trees.)   We note, by
    \cref{prop:sumconn}, that $H_S$ preserves connected
    limits.  So $\theta$ is bijective, and therefore
  $\wide{S}$  is a set by \cref{prop:injalg}(\ref{item:injalgsig}).

To show that Simple Wide Infinity implies Full Wide Infinity, let $S =
(K_i)_{i \in I}$ be a signature.  Define  the set $\ovs
\eqdef I+\sum_{i \in I}K_i$, so $\redwide{\ovs}$ is a set.  We obtain an injection $H_S
\redwide{\ovs} \rightarrow \redwide{\ovs}$ sending $\tuple{i, [a_k]_{k
        \in K_i}}$ to $\Succb{[b_p]_{p \in \ovs}}$, where
      \begin{spaceout}{rcll}
          b_{\itinl i} & \eqdef & \Succb{[\Zero]_{p \in \ovs}} \\
          b_{\itinl j} & \eqdef & \Zero  & \text{(for $j \smin I, j \not= i$)} \\
          b_{\itinr \tuple{i,k}} & \eqdef & a_k  & \text{(for $k \smin
            K_i$)} \\
          b_{\itinr \tuple{j,k}} & \eqdef & \Zero  & \text{(for $j
              \smin I, j \not= i, k \smin K_j$)} \\
      \end{spaceout}
So $\wide{S}$ is a set by \cref{prop:injalg}(\ref{item:injalgsig}).
\end{proof}

\paragraph*{Related work} \quad  The idea that every signature has an initial algebra (meaning: algebra
carried by a set) has often appeared
in the literature.  Apparently, the earliest ZFC proof was given
by S\l{}omi\'{n}ski~\cite{Slominski:theorymodinfoprel}, and the earliest 
ZF proof by 
Kerkhoff~\cite{Kerkhoff:freealg}.  Furthermore, as explained
in~\cite{PareSchumacher:absfamadjfunct,Blass:initalg,MoerdijkPalmgren:wellfoundtreecat},
the result holds in any topos with a natural number
object.  

\psubsection{Broad Infinity} \label{sect:broadinftech}

Now we come to the main principle of the paper, which was briefly described
 in \Cref{sect:nutshell}.  We shall use the following notation.   For any class $A$ and $A$-tuple
 of classes $B$, we write
 \begin{eqnarray*}
   \summay_{x \in A} B(x) & \eqdef & \setbr{\rStart} \cup
                                     \setbr{\rBuild{x}{y} \mid x \smin
                                     A, y \smin B(x)}
 \end{eqnarray*}
 Thus we have a bijection $\summay_{x \in A}B(x) \cong \incs
 \sum_{x\in A}B(x)$, sending $\rStart$ to $\Zero$ and $\rBuild{x}{y}$
 to $\Succ{\tuple{x,y}}$.  

A \emph{broad arity} is a function  $F
\ccolc \totall \rightarrow \set$.  It gives
rise to a monotone endofunction $\incf$ on $\class$, sending $X$ to
$\summay_{x \in X} X^{Fx}$.   Thus a class $X$ is $\incf$-prefixed when it contains $\rStart$ and, for any  $x \smin X$
and $Fx$-tuple $y$ within $X$, contains $\rBuild{x}{y}$.

The least such class is
called the \emph{class of all simple $F$-broad numbers} and written
$\redbroad{F}$.   To show it exists, note that $\incf$ arises from the
following set-based scaffold on $\totall$: a parent is either $\rStart$,
which has no children, or $\rBuild{x}{y}$, for a thing $x$ and
$Fx$-tuple $y$, in which case the set of children is $\setbr{x} \cup
\setbr{y_k \mid k \smin K}$.  So we apply 
\cref{prop:gensub}(\ref{item:gensub}).

Although $F$ is defined over $\totall$, only its restriction to
$\redbroad{F}$ matters.   More precisely, for broad arities $F$ and
$F'$ with the same restriction to $\redbroad{F} \cap \redbroad{F'}$,
we have $\redbroad{F} = \redbroad{F'}$.  Proof: the class $\redbroad{F} \cap
\redbroad{F'}$ is both $\incf$-prefixed and
$\incfprime$-prefixed.

Let us see some examples of simple broad numbers.
\begin{example} \label{example:simplebroad}
   Define $F$ to be the broad arity that sends $\rBuild{\rStart}{[\,]}$ to $\setbr{0,1}$, and everything else to $\emptyset$.  
   The following are simple $F$-broad numbers:
   \begin{itemize}
   \item  
     \begin{math}
       \rStart
     \end{math}
   \item
     \begin{math}
        \rBuild{\rStart}{[\,]}
     \end{math}
   \item
     \begin{math}
           \rBuild{\rBuild{\rStart}{[\,]}}{
      \begin{bmatrix}
\explind{0 & \mapsto &} \rStart \\ \explind{1 & \mapsto &} \rBuild{\rStart}{[\,]}
\end{bmatrix}
} 
     \end{math}
   \item
     \begin{math}
          \rBuild{\rBuild{\rBuild{\rStart}{[\,]}}{
      \begin{bmatrix}
        \explind{0 & \mapsto &} \rStart \\
        \explind{1 & \mapsto &} \rBuild{\rStart}{[\,]}
\end{bmatrix}
}}{[\,]} 
     \end{math}
   \end{itemize}
\end{example}
\noindent We can visualise a broad number as a \wellfound{}
three-dimensional tree, using the vertical dimension for
$\begin{bmatrix}{\vdots}\end{bmatrix}$, the horizontal dimension for
$\rBuild{-}{-}$ and the depth dimension for internal
structure.  The root appears at the front, and the
$\rStart$-marked leaves at the rear.

The axiom scheme of \emph{Simple Broad Infinity} says that, for every
broad arity $F$, a $\incf$-prefixed set exists; this 
is equivalent to $\redbroad{F}$ being a
set.

There is an alternative version, formulated as follows. 
Note that
$\sig$ is the class of all signatures.   Any function $G \ccolc \totall
\rightarrow \sig$, called a \emph{broad signature}, gives rise to a monotone endofunction $\incg$ on
$\class$, sending $X$ to $\summay_{x \in X} H_{Gx} X$.  Thus a
class $X$ is $\incg$-prefixed iff it contains $\Start$ and, for any $x
\smin X$ with $Gx = (K_i)_{i \in I}$ and any $i
  \smin I$ and $K_i$-tuple $y$ within $X$, contains  $\yBuild{x}{i}{y}$.

As before, there is a least such class, which we call the  \emph{class
  of all $G$-broad numbers}, denoted by 
$\broad{G}$.   For  functions $G, G' \ccolc \totall \rightarrow \sig$ with the same restriction to $\broad{G} \cap \broad{G'}$,
we have $\broad{G} = \broad{G'}$.

The axiom scheme of \emph{\full{}
  Broad Infinity} says that, for every broad signature $G$, a $\incg$-prefixed set exists; this is equivalent
to $\broad{G}$ being a set.  

\begin{theorem} \hfill \label{prop:broadinf}
  \begin{enumerate}
  \item Simple Broad Infinity and Full Broad Infinity are equivalent. \label{item:broadsimplefull}
  \item Full Broad Infinity implies Full Wide Infinity. \label{item:broadimplwide}
  \end{enumerate}
\end{theorem}
\begin{proof} \hfill
  \begin{enumerate}
  \item \label{item:broadsimplefullproof} To show ($\Leftarrow$), let $F$ be a broad arity.  We recursively
   define the injection $f$ on $\redbroad{F}$ that sends
   \begin{itemize}
   \item $\rStart$ to $\Start$.
   \item $\rBuild{w}{[a_k]_{k \in Fw}}$ to
     $\yBuild{f(w)}{0}{[f(a_k)]_{k \in Fw}}$.
   \end{itemize}
 Define the broad signature $G$ sending
   \begin{itemize}
   \item $f(x)$, for $x \smin \redbroad{F}$, to $(Fx)_{i \in
       \setbr{0}}$
   \item everything else to the empty signature.
   \end{itemize}
   Observe that $f$ sends each $w \smin \redbroad{F}$ to a $G$-broad
number, by induction on $w$.  Since
$\broad{G}$ is a set, $\redbroad{F}$ is a set.

To show ($\Rightarrow$), for a signature $S = (K_i)_{i \in I}$,
  we write $\ovs
    \eqdef I+\sum_{i \in I}K_i$.  Given a broad signature $G$, we
    recursively define an injection $g$ on $\broad{G}$ whose range
    does not contain $\rStart$, as follows.  It sends 
  \begin{itemize}
 \item $\Start$ to $\preBuild{\rStart}{[\,]}$
 \item $\yBuild{w}{i}{[a_k]_{k \in K_i}}$, where $Gw = (K_i)_{i \in
       I}$, to $\preBuild{\Succm{g(w)}}{[b_p]_{p \in
         \overline{Gw}}}$, using the definitions
      \begin{spaceout}{rcll}
          b_{\itinl i} & \eqdef & \rBuild{\rStart}{[\,]} \\
          b_{\itinl j} & \eqdef & \rStart  & \text{(for $j \smin I, j \not= i$)} \\
          b_{\itinr \tuple{i,k}} & \eqdef & \Succm{g(a_k)}  & \text{(for $k \smin
            K_i$)} \\
          b_{\itinr \tuple{j,k}} & \eqdef & \rStart  & \text{(for $j
              \smin I, j \not= i, k \smin K_j$).} \\
          \end{spaceout}%
        \end{itemize}
        Let $F$ be the broad arity
          that sends
          \begin{itemize}
       \item $\Succb{g(w)}$, for $w \smin \broad{G}$, to
         $\overline{Gw}$
  \item   everything else, including $\rStart$, to $\emptyset$.
  \end{itemize}
For any $w \smin \broad{G}$, we have $g(w) \in \redbroad{F}$, by
induction on $w$.  Since $\redbroad{F}$ is a set, 
 $\broad{G}$ is a set.
\item \label{item:broadimplwideproof} Given a signature $S$, let $G$ be the broad signature
    sending everything to $S$.  Recursively define the injection $g
    \ccolc \wide{S} \rightarrow \broad{G}$ sending $\tuple{i,[a_k]_{k
        \in K_i}}$ to $\yBuild{\Start}{i}{[g(a_k)]_{k \in Fa}}$.  Since $\broad{G}$ is a set, $\wide{S}$ is a set.   \qedhere
  \end{enumerate}
\end{proof}

\psection{Introducing Rubrics}  \label{sect:rubrics}
\psubsection{Generating a subset} \label{sect:setgen}

Having completed our presentation of the ``plausible'' principles, we
now move on to the ``useful'' ones.   First we consider how to 
generate a subset of a class using a suitable
collection of rules, called a \emph{rubric}. 
\begin{definition} \label{def:rubric}
 Let $C$ be a class.
  \begin{enumerate}
  \item \label{item:widerule} A \emph{wide rule  on $C$} consists of a set $K$ (the
    \emph{arity}) and a function  $R \,\colon\,C^K \rightarrow
    \famof{C}$.   It is written $\tuple{K,R}$, and the collection of all such is denoted $\widerule{C}$.
\item  \label{item:widerubric}  A \emph{wide rubric on $C$} is a family of wide rules---i.e.,
  a set $I$ and function $r \ccolc I \rightarrow \widerule{C}$.  It is written
  $(r_i)_{i \in I}$, and the collection
  of all such is denoted $\allrub{C}$.
\item \label{item:broadrule} A \emph{broad rule on $C$} consists of a set $L$ (the 
    \emph{arity}) and a function $S \ccolc C^L \rightarrow
    \allrub{C}$.  It is written $\tuple{L,S}$, and the collection of
    all such is denoted 
    $\broadrule{C}$.
\item \label{item:broadrubric}   A \emph{broad rubric on $C$} is a family of broad rules---i.e., a set $J$ and function $s \ccolc J \rightarrow \broadrule{C}$.  It is
  written $(s_j)_{j \in J}$ , and the
  collection of all such is denoted $\allbroadrub{C}$. 
  \end{enumerate}
\end{definition}


\begin{example} \label{example:rubric}
  Here is a wide rubric on $\nats$, consisting of two wide rules.
  \begin{itemize}
  \item {Rule 0} is binary and sends $
    \begin{bmatrix}
\explind{0 & \mapsto &}m_0 \\ \explind{1 & \mapsto &}m_1
\end{bmatrix}
\mapsto (m_0+m_1+p)_{p \geqslant 2m_0}$.
   \item {Rule 1} is nullary and sends $[\,] \mapsto (2p)_{p\geqslant 50}$.
   \end{itemize}
 Informally, these rules prescribe when an element of $\nats$ is acceptable.  Rule 0 says that, if $m_0$ and $m_1$ are acceptable, then $m_0+m_1+p$ is acceptable for all $p \geqslant 2m_0$.  Rule 1 says that $2p$ is acceptable for all $p \geqslant 50$.  So 100, 102 and 402 are acceptable, and by induction every acceptable number is $\geqslant 100$.
\end{example}

\begin{example} \label{example:broadrubric}
  Here is a broad rubric on $\nats$, consisting of two broad rules.  Broad
  rule 0 is nullary and sends $[\,]$ to the wide rubric described in
  Example~\ref{example:rubric}.  Broad rule 1 is unary.  It sends 
  $[7]$ to the the following
  wide rubric, consisting of one wide rule.
\begin{itemize}
\item {Rule 0} is binary and sends $
  \begin{bmatrix}
\explind{0 & \mapsto & }m_0 \\\explind{1 & \mapsto &} m_1
\end{bmatrix}
\mapsto (m_0+m_1+500p)_{p \geqslant 9}$.
\end{itemize}
It sends $[100]$ to the following wide rubric, consisting of three
wide rules.
\begin{itemize}
\item {Rule 0} is ternary and sends $
  \begin{bmatrix}
\explind{0 & \mapsto &}m_0 \\ \explind{1 & \mapsto &}m_1 \\ \explind{2 & \mapsto &}m_2
\end{bmatrix}
\mapsto (m_0+m_1m_2 +p)_{p \geqslant 17}$.
\item {Rule 1} is nullary and sends $[\,] \mapsto (p)_{p \geqslant 1000}$.
 \item {Rule 2} is binary and sends $
   \begin{bmatrix}
\explind{0 & \mapsto &}m_0 \\ \explind{1 & \mapsto &}m_1
\end{bmatrix}
\mapsto (m_1+p)_{p \geqslant 4}$.
\end{itemize}
And it sends $[n]$, for $n \smin \nats \setminus \setbr{7,100}$, to
the empty wide rubric.

 Informally, these rules prescribe when an element of $\nats$ is acceptable.   For example, if $100$ is acceptable and $m_0,m_1,m_2$ are too, then so is $m_0 + m_1m_2 + p$ for all $p \geqslant 17$.  So 100, 102, 402 and 107 are acceptable, and by induction every acceptable number is $\geqslant 100$.
\end{example}

To make the notion of ``acceptable element''
precise, we proceed as follows.
\begin{definition}  \label{def:clcomplclass} Let $C$ be a class.  A
  subclass $X$ is 
  \begin{itemize}
  \item  \emph{$\tuple{K,R}$-closed}, for a wide rule $\tuple{K,R}$ on
    $C$, when the family $R(x)$ is within $X$ for all $x \smin X^K$.  
 \item \emph{$\catr$-complete}, for a wide rubric $\catr = (r_i)_{i
     \in I}$ on $C$, when $X$ is $r_i$-closed for all $i \smin I$.
 \item \emph{$\tuple{L,S}$-closed}, for a broad rule $\tuple{L,S}$ on
   $C$, when $X$ is $S(y)$-complete for all $y \smin X^L$.
  \item \emph{$\cats$-complete} for a broad rubric $\cats = (s_j)_{j
      \in J}$ on $C$, when $X$ is $s_j$-closed for all $j \smin J$.
  \end{itemize}
\end{definition}
Our next task is to give an alternative formulation of completeness, using a
notion of ``plate''.  First we give preliminary notions that do not
depend on a rubric.
\begin{definition} \hfill \label{def:preplate}
   \begin{enumerate}
   \item \label{item:widepreplate} A \emph{wide preplate} is a triple $w =
\tuple{i,[x_k]_{k \in K},p}$, where $K$ can be any set.  The
\emph{component set} of $w$ is
\begin{eqnarray*}
\compsof{w} & \eqdef &
  \setbr{x_k \mid k \smin K}
\end{eqnarray*}
  \item \label{item:broadpreplate} A \emph{broad preplate} is a 5-tuple $w = \tuple{j,[y_l]_{l \in L},i, [x_k]_{k
    \in K},p}$, where $L$ and $K$ can be any sets.  The 
\emph{component set} of $w$ is
\begin{eqnarray*}
\compsof{w} & \eqdef & \setbr{y_l \mid l \smin L} \cup
\setbr{x_k \mid k \smin K}
\end{eqnarray*}
\item A preplate $w$ is \emph{within} a class $X$ when $\compsof{w} \subseteq X$.
   \end{enumerate}
 \end{definition}
Now we turn to rubrics.
\begin{definition} \label{def:plate} Let $C$ be a class.
  \begin{enumerate}
  \item \label{item:wideplate} Let $\catr$ be a wide rubric on
    $C$.    An \emph{$\catr$-plate} is a wide preplate 
    \begin{eqnarray*}
      w & = &  \tuple{i,
      [x_k]_{k \in K}, p}
    \end{eqnarray*}
within $C$ such that 
     \begin{itemize}
\item writing $\catr   = (\tuple{K_i,R_i})_{i \in I}$, we have $i
  \smin I$ and $K= K_i$
\item writing $R_j(x) =  (b_p)_{p \in
                                                       P}$, we have $p
                                                     \smin P$.
                                                   \end{itemize}
                                                   The \emph{result} of $w$ is $b_p$.
  \item \label{item:broadplate} Let $\cats$ be a broad rubric on
    $C$.  An \emph{$\cats$-plate} is a broad preplate
    \begin{eqnarray*}
w & = &  \tuple{j,
      [y_l]_{l \in L},i, [x_k]_{k \in K}, p}
    \end{eqnarray*}
within $C$ such that 
    \begin{itemize}
\item writing $\cats = (\tuple{L_j,S_j})_{j \in J}$, we have $j
      \smin J$ and $L = L_j$
  \item writing $S_j(y)   = (\tuple{K_i,R_i})_{i \in I}$, we have $i
    \smin I$ and $K= K_i$
     \item writing $R_i(x) =  (b_p)_{p \in
                                                       P}$, we have $p
                                                    \smin P$. 
      \end{itemize}
The \emph{result}  of $w$ is $b_p$.  
  \end{enumerate}
\end{definition}

\begin{definition} \label{def:rubfun}
  Let $\catr$ be a rubric on a class $C$.  
 We write  $\Gamma_{\catr}$ for the set-continuous endofunction on $\subclassof{C}$ sending $X$ to
  the class of all results  of $\catr$-plates within $X$.  
\end{definition}
Clearly a subclass of $C$ is $\catr$-complete iff $\Gamma_{\catr}$-prefixed.

The least $\catr$-complete
  subclass of $C$, if it exists, is said to be \emph{$\catr$-generated}. 
   We shall see below (\cref{prop:powallgen}) that this class exists if Powerset or
Collection holds. However, I do not know whether our base theory alone 
guarantees its existence.  
  In
  any case, a rubric's purpose  
  is to generate a
  \emph{set}, not merely a class.    So we formulate the following principles.
   \begin{itemize}
 \item The \emph{Wide \setgen{}} scheme says that any wide
   rubric on a class generates a subset.
\item The \emph{Broad \setgen{}} scheme says that any broad
   rubric on a class generates a subset.
 \end{itemize}
By Proposition~\ref{prop:stabord}, a rubric $\catr$ on a class $C$ generates a
subset iff $C$ has an $\catr$-complete subset.
 
 \psubsection{Application: Grothendieck Universes} \label{sect:groth}

 For this section, \textbf{assume Powerset $+$ Infinity}.

 As promised in \Cref{sect:plausibleuseful}, we see the utility of Broad \setgen{}: it gives
 Grothendieck universes without any detour via notions of cardinal or ordinal.
 \begin{definition}  \label{def:groth} 
 A \emph{Grothendieck universe} is a transitive set $\univv$ with the following properties:
   \begin{itemize}
         \item For every set $A \in \univv$, we have $\pset A \in \univv$.
    \item $\nats \in \univv$.
    \item For every set of sets $\cata \in \univv$, we have $\bigcup \cata \in \univv$.
    \item For every set $K \in \univv$ and $K$-tuple $[a_k]_{k \in K}$ within $\univv$, we have $\setbr{a_k \mid k \in K} \in \univv$.
    \end{itemize}
  \end{definition}

   \begin{proposition} \label{prop:groth}  
    Broad \setgen{} implies the ``Axiom of Universes'': For every set $X$, there is a least Grothendieck universe $\univv$ with $X \subseteq \univv$.  
  \end{proposition}
  \begin{proof}
  We define a broad rubric $\catb$ on $\totall$, consisting of two
  rules.   Broad rule 0 is nullary and sends $[\,]$ to the
  following wide rubric indexed by $X+4$.
    \begin{itemize}
    \item To achieve $X \subseteq \univv$, rule $\itinl x$ (for $x \in X$) is nullary, and sends $[\,]$ to $(x)\inone$.
    \item To achieve transitivity, rule $\itinr 0$ is unary, sending $[A]$ to $(b)_{b \in A}$
      if $A$ is a set, and the empty family otherwise.
    \item  Rule $\itinr 1$ is nullary, and sends $[\,]$ to $(\nats)$. 
    \item Rule $\itinr 2$ is unary, sending $[\cata]$ to $(\bigcup
      \cata)$ if $\cata$ is a set of sets, and the empty signature otherwise.
    \item Rule $\itinr 3$ has arity 1, sending $[A]$ to $(\pset A)$ if
      $A$ is a set, and the empty signature otherwise.
    \end{itemize}
Broad rule 1 is unary.  For any set $B$, it sends $[B]$ to
the rubric consisting of one $B$-ary rule that sends $[a_k]_{k \in B}$
to $(\setbr{a_k \mid k \in B})\inone$.  If $b$ is not a set, then
Broad Rule 1 sends $[b]$ to the empty rubric.  This completes the definition of $\catb$.  The set that it generates
is the desired Grothendieck universe.  \qqed
\end{proof}

 \psubsection{Derivations} \label{sect:genfam}

Intuitively, when we have a rubric on a class $C$, each acceptable element 
$x \smin C$ has one or more ``derivations'' that explain why it is acceptable.
 \begin{example} \label{example:widederiv}
For the wide rubric in Example~\ref{example:rubric}:
     \begin{itemize}
  \item $\tuple{1,[\,],50}$ derives 100.
  \item $\tuple{1,[\,],51}$ derives 102.
  \item $\tuple{0,
      \begin{bmatrix}
       \explind{ 0 & \mapsto &} \tuple{1,[\,],50} \\
       \explind{ 1  & \mapsto &} \tuple{1,[\,],50}
\end{bmatrix}
,202}$ and  $\tuple{0,
\begin{bmatrix}
  \explind{0 & \mapsto &} \tuple{1,[\,],50} \\
  \explind{1 & \mapsto &} \tuple{1,[\,],51}
\end{bmatrix}
,200}$ derive 402.
\end{itemize}
 Note that each derivation is a wide preplate whose components are derivations.
\end{example}

\begin{example} For the broad rubric in
  Example~\ref{example:broadrubric}: \label{example:broadderiv}
  \begin{itemize}
  \item $\tuple{0,[\,],1,[\,],50}$ derives 100.
  \item $\tuple{0,[\,],1,[\,],51}$ derives 102.
  \item $\tuple{0,[\tuple{0,[\,],1,[\,],50}],2,
      \begin{bmatrix}
\explind{0 & \mapsto &} \tuple{0,[\,],1,[\,],50} \\
\explind{1 & \mapsto &} \tuple{0,[\,],1,[\,],51}
\end{bmatrix},5}$ derives 107.
\end{itemize}
Note that each derivation is a broad preplate whose components are derivations.
\end{example}
Given a rubric $\catr$ on a class $C$, we would like to define the class
$\derivsof{\catr}$ of all \emph{$\catr$-derivations}.  Each
$\catr$-derivation $x$ will have an \emph{overall result} $\ovresof{x}
\in C$.   We shall call $(\derivsof{\catr}, \ovreso)$ the
 \emph{$\catr$-derivational} class-family within $C$---it is the
 algebraically least prefixpoint of an endofunctor $\Delta_{\catr}$ on
 $\classfamof{C}$ that we shall define.   To do this, we adapt the notion of
 $\catr$-plate (Definition~\ref{def:plate}).  First we give
 preliminary notation that does not depend on a rubric.
 \begin{definition}  \label{def:cwise} \hfill
   \begin{enumerate}
   \item \label{item:cwisewide} For a wide preplate $w =
\tuple{i,[x_k]_{k \in K},p}$, and a function $h$ on
  $\compsof{w}$, we write
  \begin{eqnarray*}
  \cwise{h}(w) & \eqdef & \tuple{i,[h(x_k)]_{k \in K} ,p}
  \end{eqnarray*}
  \item \label{item:cwisebroad} For a {broad preplate} $w = \tuple{j,[y_l]_{l \in L},i, [x_k]_{k
    \in K},p}$, and function $h$ on
  $\compsof{w}$, we write
    \begin{eqnarray*}
  \cwise{h}(w) & \eqdef & \tuple{j,[h(y_l)]_{l \in L},i,[h(x_k)]_{k \in K},p}
\end{eqnarray*}
   \end{enumerate}
 \end{definition}
 Now we turn to rubrics.
 \begin{definition} \label{def:platefam} Let $(M,F)$ be a class-family
   within a class $C$.
   \begin{enumerate}
   \item\label{item:wideplatefam}  Let $\catr$ be a wide rubric on $C$.  An
     \emph{$(\catr,M,F)$-plate} $w$ is a wide preplate within $M$
     whose $\cwise{F}$-image is an $\catr$-plate.  The \emph{result} of
     $w$ is 
     the result of $\cwise{F}(w)$. 
   \item  \label{item:broadplatefam} Let $\cats$ be a broad rubric on $C$.  An 
     \emph{$(\cats,M,F)$-plate} $w$ is a broad preplate within $M$
     whose $\cwise{F}$-image is an $\cats$-plate.  The  \emph{result} of
     $w$ is the result of $\cwise{F}(w)$. 
   \end{enumerate}
 \end{definition}

 \begin{definition} \label{def:rubfunfam}
    Let $\catr$ be a rubric on a class $C$.
    We define the set-continuous endofunctor
     $\gamfamcatr$ on $\classfamcat{C}$ sending
     \begin{itemize}
     \item a class-family $(M,F)$ to $(N,G)$, where $N$ is the class of all
       $(\catr,M,F)$-plates and $G$ sends each such plate to its
       result
     \item a map $h \ccolc (M,F) \rightarrow (M',F')$ to
       the map $w \mapsto \cwise{h}(w)$.
     \end{itemize}
   \end{definition}

\begin{proposition} \label{prop:gamfamscaff}
  Let $\catr$ be a rubric on a class $C$.  Then
  $\gamfamcatr$ has an inductive chain and an algebraically least prefixpoint.  
\end{proposition}
\begin{proof}
We show $\gamfamcatr$ has an inductive chain and a least prefixpoint
by \cref{prop:scaffpfnchain}(\ref{item:pfnchain}).  Noting that
$\classfamof{C} = \classfamoo{C}{x \in \totall}$, we express $\gamfamcatr$ as $\delwl$, for a set-based scaffold
  $(D,<)$ on $\totall$ with functionalization $L$ on $x \mapsto C$.  

 We give the wide case, as the broad case is similar.
 Firstly, we have  a set-based scaffold $(D,<)$ on
 $\totall$, where $D$ is the class of all wide preplates and $<$ is
 componenthood.    
 
Now let $\catr $ be a wide rubric on
   $C$.  Then $(D,<)$ has the following
 functionalization $L$ to $(C)_{x \in \totall}$.  For a wide preplate
 $u$, 
  we define $\domof{L_u}$ to be the class of all functions $f
 \ccolc J(u) \rightarrow C$ such that $\cwise{f}(u)$ is an
 $\catr$-plate, and $\overline{L_u}$ sends such a function $f $ to the
 result of $\cwise{f}(u)$.  

 It is then evident that $\gamfamcatr = \delwl$, so we obtain a least
 $\gamfamcatr$-prefixpoint,  and
 \cref{prop:wfrecur}(\ref{item:wfrecurtot}) gives the
 initial algebra property.
\end{proof}
Thus we define the $\catr$-derivational class-family
$(\derivsof{\catr},\ovreso) \eqdef \mu \gamfamcatr$.  
  We want to know whether this is a family---i.e., whether $\derivsof{\catr}$ is a
set.  So we formulate the following
principles.  
\begin{itemize}
  \item \emph{Wide Derivation Set}: Any wide rubric on a class
    has a derivation set.
   \item \emph{Broad Derivation Set}: Any broad
     rubric on a class has a derivation set.
   \end{itemize}

 \psubsection{Application: Tarski-style Universes} \label{sect:tarski}

{}  For this section, \textbf{assume Powerset $+$ Infinity}.
 
In the type theory literature~\cite{MartinLoef:inttypetheory}, a 
``Tarski-style universe'' is a family of types that is closed under various
constructions, such as $\sum$.
Furthermore, the existence of such universes follows from various ``induction-recursion'' principles~\cite{DybjerSetzer:indindrec,GhaniHancock:contmonir}.

In a similar way, we show that Broad
Derivation Set yields the existence of arbitrarily large Tarski-style
universes.

\begin{definition} \label{def:tarski}
  A \emph{Tarski-style universe} consists of the following data.
  \begin{itemize}
  \item A family of sets $(D_m)_{m \in M}$.  Elements of $M$ are
    called \emph{codes}.
   \item For each code $m$, a code $\tpow{m}$ and bijection $D_{\tpow{m}}
    \cong \pset D_{m}$.
  \item A code $\tzero$ and bijection $D_{\tzero} \cong \emptyset$.
  \item A code $\tnat$ and bijection $D_{\tnat} \cong \nats$.
  \item For each code $m$ and tuple of codes $[g_k]_{k \in D_m}$, a
    code $\tsigma{m}{g}$ and bijection $D_{\tsigma{m}{g}} \cong
    \sum_{k \in D_m} D_{g(m)}$.
  \end{itemize}
\end{definition}

\begin{definition} \label{def:tarskiext}
For a family of sets $(B_a)_{a \in A}$, a \emph{Tarski-style
  universe extension} consists of the following data:
\begin{itemize}
\item A Tarski-style universe  $(D_m)_{m \in M}$.
\item For each $a \smin A$, a code $j(a)$ and bijection
    $D_{j(a)} \cong B_{a}$.  
\end{itemize}
\end{definition}

\begin{proposition} \label{prop:tarskiext}
 Broad Derivation Set implies that every family of sets $(B_a)_{a \in
   A}$ has a Tarski-style universe extension.
\end{proposition}
\begin{proof}
We shall construct the extension so that all the required
bijections are identities.  To begin, define a broad
    rubric $\catb$ on $\allsets$, consisting of two broad rules.
    Broad rule 0 is nullary and sends $[\,]$ to the following rubric indexed by $A +\setbr{0,1,2}$.
    \begin{itemize}
    \item Rule $\itinl a$ (for $a \in A$) has arity 0 and sends $[\,]$
      to $(B_a)\inone$.
    \item Rule $\itinr 0$ has arity 1 and sends $X$ to $(\pset X)\inone$.
    \item Rule $\itinr 1$ has arity 0 and sends $[\,]$ to $(\emptyset)\inone$.
    \item Rule $\itinr 2$ has arity 0 and sends $[\,]$ to $(\nats)\inone$.
    \end{itemize}
   Broad rule 1 is unary, and sends $[X]$ to the following rubric 
   indexed by $\setbr{0}$.
   \begin{itemize}
    \item Rule $0$ has arity $X$ and sends $[Y_x]_{x \in X}$ to
      $(\sum_{x \in X}Y_x)\inone$.
      \end{itemize}
 The $\catb$-derivational family is the desired universe
 extension, where we define
 \begin{eqnarray*}
   j(a) & \eqdef & \tuple{0,[\,],\itinl a,[\,], *} \\
   \tpow{m} & \eqdef & \tuple{0,[\,],\itinr 0,[m], *}  \\
   \tzero & \eqdef & \tuple{0,[\,],\itinr 1,[\,], *} \\
   \tnat & \eqdef & \tuple{0,[\,],\itinr 2,[\,], *} \\
   \tsigma{m}{g} & \eqdef & \tuple{1,[m],0,g, *} 
 \end{eqnarray*}
\end{proof}

\psubsection{Proving the Derivation Set Principles} \label{sect:provederfam}

The following is the
central result of the paper (at least for
 people who do not accept AC), since it says that \emph{plausible principles
 entail useful ones}.   
\begin{proposition} \label{prop:implderiv} \hfill  
  \begin{enumerate}
  \item \label{item:wideimplderiv} Full Wide Infinity implies Wide Derivation Set.
  \item \label{item:broadimplderiv} Full Broad Infinity implies Broad Derivation Set.
  \end{enumerate}
\end{proposition}
\begin{proof} \hfill
  \begin{enumerate}
 \item  \label{item:wideimplderivproof} Let $\catr = (\tuple{K_i,R_i})_{i \in I}$ be a wide rubric on a
  class $C$.   Define the signature $S$ to be $(K_i)_{i \in I}$.  We
  recursively associate to each $t \in \mus$ a family $(M_t,F_t)$ 
    within $C$ as follows.  For $t = \tuple{i,[t_k]_{k \in
      K_i}}$, an element of $M_t$ is a triple $\tuple{i,[m_k]_{k \in
      K_i},p}$ where $i \in I$ and $m \in \prod_{k \in K_i} M_{t_{k}}$
  with $R_{i}[F_{t_k} (m_k)]_{k \in K_i} = (b_p)_{p \in P}$ and $p \in
  P$, and $F_t$ sends this element to $b_p$.   For any $t,t' \in
  \mus$, if $M_t\cap M_t'$ is inhabited, then $t=t'$, by induction on
  $t$.

 We define the family $(M,F)$ within $C$ to be the union of all
 these.  Thus we define $M \eqdef \bigcup_{t \in \mus} M_t$, and $F$ sends
  $m \smin M_t$ to $F_t (m)$.  It is then evident that $(M,F)$ is the
  derivational family of $\catr$.

\item  \label{item:broadimplderivproof} Let 
$\cats=(\tuple{L_j,S_j})_{j \in J}$ be a broad rubric on a class $C$.
We recursively define the
function $\theta$  on $\derivsof{\cats}$ that sends $\tuple{j,[y_l]_{l \in L_j},i,[x_k]_{k \in
    K_i},p}$ to
\begin{displaymath}
  \yBuild{\yBuild{\yBuild{\Start}{j}{[\theta y_i]_{i \in
          L_j}}}{i}{[\theta x_k]_{k \in K_i}}}{p}{[\,]}
\end{displaymath}
By induction, $\theta$ is injective.  To show that $\derivsof{\cats}$
is a set, it suffices to give a broad
signature $G$ such that the range of $\theta$ is included in
$\broad{G}$, which by Full Broad Infinity is a set.  Define $G \ccolc \totall
\rightarrow \allsets$ to send
\begin{itemize}
\item $\Start$ to $ (L_j)_{j \in J}$
\item $\yBuild{\Start}{j}{[\theta y_l]_{l \in L_j}}$ obtained from 
  \begin{itemize}
  \item an index $j \smin J$ and $y \smin \derivsof{\cats}^{L_j}$, giving
    \begin{eqnarray*}
      S_j[\ovresof{y_l}]_{l \in L_j} & = & (\tuple{K_i,R_i})_{i \in I}
    \end{eqnarray*}
  \end{itemize}
 to $ (K_i)_{i \in
   I}$
 \item $ \yBuild{\yBuild{\Start}{j}{[\theta y_i]_{i \in
         L_j}}}{i}{[\theta x_k]_{k \in K_i}}$ obtained from 
   \begin{itemize}
   \item an index $j \smin J$ and $y \smin \derivsof{\cats}^{L_j}$, giving
    \begin{eqnarray*}
      S_j[\ovresof{y_l}]_{l \in L_j} & = & (\tuple{K_i,R_i})_{i \in I}
    \end{eqnarray*}
   \item and an index $j \smin J$ and $x \smin \derivsof{\cats}^{L_j}$, giving
     \begin{eqnarray*}
       R_i[\ovresof{x_k}]_{k
                                                       \in K_i} & = & (b_p)_{p \in
                                                         P}
     \end{eqnarray*}
   \end{itemize}
   to $(\emptyset)_{p \in P}$
 \item everything else to the empty signature.
\end{itemize}
 By induction, for
                                                     every
                                                     $\cats$-derivation
                                                     $x$, we see that
                                                     $\theta x$ is a
                                                     $G$-broad
                                                     number, as
                                                     required.  \qedhere
  \end{enumerate}
\end{proof}




\psection{Special Kinds of Rubric} \label{sect:specialrub}
\psubsection{Injective Rubrics} \label{sect:injrub}

\begin{definition} \label{def:injrub}
Let $C$ be a class. A rubric $\catr$ on $C$ is \emph{injective} when any two
    $\catr$-plates with the same result are equal.
\end{definition}

  
 We now give the following principles.
 \begin{itemize}
   \item \emph{Injective Wide \setgen{}}: Any injective wide
   rubric on a class generates a subset.
\item \emph{Injective Broad \setgen{}}: Any injective broad
  rubric on a class generates a subset.
\end{itemize}

 
 \begin{proposition} \label{prop:setgeninf} \hfill
   \begin{enumerate}
   \item \label{item:widesginf} Injective Wide \setgen{} implies Full Wide Infinity.
   \item \label{item:broadsginf} Injective Broad \setgen{} implies
     Full Broad Infinity.
   \end{enumerate}
 \end{proposition}
 \begin{proof} \hfill
   \begin{enumerate}
     \item \label{item:widesginfproof} Let $S = (K_i)_{i
    \in I}$ be a signature.  A class is $H_S$-prefixed iff it is
  $\hat{S}$-complete, writing $\hat{S}$ for the following injective 
     wide rubric on $\totall$: it consists of $I$ rules, and rule $i
     \smin I$ has arity $K_i$ and sends  a $K_i$-tuple
       $x$ to the singleton $(\tuple{i,x})$.  So $\wide{S}$ is generated by $\hat{S}$,
       and is therefore a set by Injective Wide \setgen{}.
     \item \label{item:broadsginfproof} Let  $G$ be a broad signature.  A class is
  $\incg$-prefixed iff it is $\hat{G}$-complete, writing $\hat{G}$ for
  the injective 
   broad
     rubric on $\totall$ consisting of two rules:
     \begin{itemize}
     \item Rule 0 is nullary and sends $[\,]$ to the wide rubric
       consisting of a nullary broad rule that returns the singleton
       $(\Zero)$.
     \item Rule 1 is unary and sends $[w]$, where $Gw = (K_i)_{i
         \in I}$ to the wide rubric
       consisting of $I$ rules, where rule $I$ has arity $K_i$ and
       sends a $K_i$-tuple $x$
       to the singleton $(\yBuild{w}{i}{x})$.
     \end{itemize}
   So  $\broad{G}$ is generated by $\hat{G}$, and is therefore a set
   by Injective Broad \setgen{}.  \qedhere
   \end{enumerate}
 \end{proof}

 \psubsection{Comparing Rubrics} \label{sect:comparerub}


Suppose we have two rubrics $\catr$ and $\cats$ on the same class, and want to show they are ``essentially the same''.  Giving a natural isomorphism $\Delta_{\catr} \cong \Delta_{\cats}$
suffices for our purposes:
\begin{proposition} \label{prop:deliso}
  Let $C$ be a class.  Let $\catr$ and $\cats$ be rubrics on $C$, and
$\alpha \ccolc \Delta_{\catr} \cong \Delta_{\cats}$ a natural isomorphism.
\begin{enumerate}
  \item \label{item:deriviff} $\catr$ has a derivation set iff $\cats$ does.
  \item \label{item:delisorange} The square
    \begin{math}
      \xymatrix{
        \classfamof{C} \ar[r]^-{\Delta_{\cats}}
        \ar[d]_{\Delta_{\catr}} & \classfamof{C} \ar[d]^{\mathsf{Range}} \\
        \classfamof{C} \ar[r]_-{\mathsf{Range}} & \subclassof{C}
        }
      \end{math}
      commutes.
  \item \label{item:delisogam} The functions $\Gamma_{\catr}$ and $\Gamma_{\cats}$ are equal.
  \item \label{item:delisoinj} $\catr$ is injective iff $\cats$ is.
  \end{enumerate}
\end{proposition}
\begin{proof} \hfill
  \begin{enumerate}
 \item \label{item:deriviffproof}  By \cref{prop:initalgiso}.
  \item \label{item:delisorangeproof} Obvious.
  \item  \label{item:delisogamproof} The range of $\Delta_{\catr}(X,
    \iof{X}{C})$ is $\Gamma_{\catr}X$, and the range of  $\Delta_{\cats}(X,
    \iof{X}{C})$ is $\Gamma_{\cats}X$.  These ranges are the same, by part~(\ref{item:delisorange}).  
\item \label{item:delisoinjproof}  Since $\catr$-plates are the same thing as $(\catr,C,\id_C)$-plates.
  \qedhere
  \end{enumerate}
\end{proof}



 \psubsection{\Qwide{} Rubrics} \label{sect:qwide}

For any rubric $\catr$, we shall now consider the class $\arsof{\catr}$ of all arities that
appear inside it, defined explicitly as follows.
\begin{definition} \label{def:arsof}
  Let $C$ be a class.
  \begin{enumerate}
  \item \label{item:arwiderub} For a wide rubric $\catr = (\tuple{K_i,R_i})_{i \in I}$ on
    $C$, we obtain the set 
    \begin{eqnarray*}
      \arsof{\catr} & \eqdef & \setbr{K_i \mid i \smin I} 
    \end{eqnarray*}
  \item \label{item:arbroadrule} For a broad rule $\tuple{L,S}$ on $C$, we obtain the class
    \begin{eqnarray*}
      \arsof{S} & \eqdef & \bigcup_{x \in C^L} \arsof{S(x)}
    \end{eqnarray*}
\item\label{item:arbroadrub} For a broad rubric $\cats = (\tuple{L_j,S_j})_{j \in J}$ on $C$,
  we obtain the class
    \begin{eqnarray*}
      \arsof{\cats} & \eqdef & \bigcup_{j \in J} (\setbr{L_j} \cup \arsof{S_j}) 
    \end{eqnarray*}
  \end{enumerate}
\end{definition}

The fact that $\arsof{\catr}$ is a set for a wide rubric $\catr$
motivates the following.
\begin{definition} \label{def:qwide}
Let $C$ be a class.  A \emph{\qwide{} rubric} on $C$ is a broad rubric
$\cats$ such that $\arsof{\cats}$ is a set.
\end{definition}
We see that wide and \qwide{} are  
essentially the same for our purposes:
\begin{proposition} Let $C$ be a class. \label{prop:wideqwide}
  \begin{enumerate}
  \item \label{item:wideqwide}
    For any wide rubric $\catr$ on $C$, there is a \qwide{} rubric
    $\cats$ on $C$ and natural isomorphism $\Delta_{\catr} \cong
    \Delta_{\cats}$.
   \item \label{item:qwidewide} Conversely, for any \qwide{} rubric $\cats$ on $C$, there is
     a wide rubric $\catr$ on $C$ and natural isomorphism $\Delta_{\cats} \cong
     \Delta_{\catr}$.
   \end{enumerate}
\end{proposition}
\begin{proof} \hfill
  \begin{enumerate}
  \item \label{item:wideqwideproof}
    Let $\cats$ be given by a single nullary broad rule that sends $[\,]$ to
  $\catr$.  For a class-family $(M,F)$ within $C$, the
  $(\catr,M,F)$-plate $\tuple{i,[x_k]_{k \in K},p}$
  corresponds to the $(\cats,M,F)$-plate $\tuple{*,[\,],i,[x_k]_{k \in
      K},p}$.
\item  \label{item:qwidewideproof}  We introduce notation, for sets $A$ and $B$.  Given an $A$-tuple
    $u$ and $B$-tuple $v$, we write $\copair{u}{v}$ for the $A+B$ tuple whose $r$th component is $u_a$ or $v_b$
    according as $r$ is $\itinl a$ or $\itinr b$.   

    Given a \qwide{} rubric $\cats$ on $C$, we form a wide rubric
    indexed by the set $\arsof{\cats}^2$ as follows:
 \begin{eqnarray*}
      \catr & \eqdef & (\tuple{L+K,R_{L,K}})_{{L,K} \in \arsof{\cats}}
 \end{eqnarray*}
where $R_{L,K}$ sends $\copair{y}{x}$, for $y \smin C^L$ and $x \smin
C^K$, to the family  $(c_{n})_{n \in N}$ defined as follows.   An
      element of $N$ is a triple $\tuple{j,i,p}$ such that
      \begin{itemize}
      \item writing $\cats =  (\tuple{L_j,S_j})_{j \in J}$, we have $j
        \smin J$ and $L_j = L $ 
   \item writing  $S_j(y) = (\tuple{K_i,R_i})_{i \in I}$, we have  $i \smin I$ and
        $K_i = K$
      \item writing $R_i(x) = (z_p)_{p \in P}$, we have $p \smin P$
      \end{itemize}
      with $c_{\tuple{j,i,p}} \eqdef z_p$.  For a
      class-family $(M,F)$ within $C$, an $(\cats,M,F)$-plate
      \begin{displaymath}
        \tuple{j,[y_l]_{l \in L},i,[x_k]_{k \in K},p}
      \end{displaymath}
      corresponds to the  $({\catr},M,F)$-plate 
      \begin{displaymath}
         \tuple{\tuple{L,K},
                                                                \copair{y}{x},
                                                                \tuple{j,i,p}}   \qedhere
                                                            \end{displaymath}
  \end{enumerate}
\end{proof}


\begin{corollary} \label{cor:qwide} \hfill
  \begin{enumerate}
    \item \label{item:qwidesetgen} Wide \setgen{} is equivalent to
    \emph{\Qwide{} \setgen{}}: Any 
    \qwide{} rubric on a class generates a subset.
  \item \label{item:quasideriv} Wide Derivation Set is equivalent to \emph{\Qwide{} Derivation
    Set}: Any \qwide{} rubric on a
  class has a derivation set.
   \item \label{item:quasiinj} Injective Wide \setgen{} is equivalent to
    \emph{Injective \Qwide{} \setgen{}}: Any injective 
    \qwide{} rubric on a class generates a subset.
  \end{enumerate}
\end{corollary}



\psubsection{Application: Rubrics on a Set} \label{sect:rubonset}

We now consider the
special case of a rubric on a {set}.    Clearly such a rubric generates a
subset, but does it have a derivation set?  The following answer is adapted from~\cite{HankMcBrideGhaniMalatestaAltenkirch:smallir}.
\begin{proposition} \label{prop:exprubs} \hfill
  \begin{enumerate}
  \item \label{item:allrubs} Exponentiation is equivalent to the assertion ``Any broad
    rubric on a set is \qwide{}.''
  \item \label{item:rubonset} Full Wide Infinity is equivalent to the assertion ``Any rubric on a
set has a derivation set.''
  \end{enumerate}
\end{proposition}
\begin{proof} \hfill
  \begin{enumerate}
  \item Since ($\Rightarrow$) is evident, we just prove
    ($\Leftarrow$).  For a function $f$ on a set, we write
    $\graphof{f}$ for $f$ regarded as a set of ordered pairs.

    Given sets $A$ and $B$, define the broad rubric $\cats$ on $B$
    consisting of one $A$-ary rule, sending a tuple $x$ to the wide
    rubric consisting of one $\graphof{x}$-ary rule, sending each
    tuple to the empty family.  If $\cats$ is \qwide{}, then $B^A$ is
    a set.
  \item  For ($\Rightarrow$),  Full Wide Infinity gives Wide
    Derivation Set by
    Proposition~\ref{prop:implderiv}(\ref{item:wideimplderiv}) and
    hence \Qwide{} Derivation Set by \cref{cor:qwide}(\ref{item:quasideriv}).   It also gives
    Exponentiation, so any broad rubric on a set is \qwide{} by part (\ref{item:allrubs}), and
    therefore has a derivation set.

For ($\Leftarrow$), given a signature $S =
  (K_i)_{i \in I}$, form a wide rubric on 1 via
   \begin{eqnarray*}
  \catr & \eqdef &    (\tuple{K_i, [*]_{k \in K} \mapsto (*)})_{i \in I}
    \end{eqnarray*}
Recursively define the bijection $\theta
  \ccolc \wide{S} \cong \derivsof{\catr}$ sending $\tuple{i,[x_k]_{k
      \in K_i}}$ to $\tuple{i,[\theta(x_k)]_{k \in K}, *}$.   Thus
  $\wide{S}$ is a set iff 
  $\derivsof{\catr}$ is.
  \qedhere
  \end{enumerate}
   \end{proof}

\psection{Properties of Rubric Functions} \label{sect:provesetgen}

\psubsection{Preservation of Smallness and Injectivity}  \label{sect:pressmallinj}

Given a rubric $\catr$ on a class $C$, does the endofunction
$\Gamma_{\catr}$ on $\subclassof{C}$ restrict to one on $\psetset C$?
Likewise, does the endofunction $\Delta_{\catr}$ on $\classfamof{C}$ restrict to one
on $\famof{C}$ or $\injcfof{C}$ or $\injfof{C}$?
The following results answer these questions.

\begin{proposition} \label{prop:injpresinj}
  Let $\catr$ be a rubric on a class $C$.  Then $\catr$ is injective
  iff the endofunction $\Delta_{\catr}$ restricts to one on $\injcfof{C}$.
\end{proposition}
\begin{proof}
For ($\Rightarrow$), let $(M,F)$ be an injective class-family, and
  let $x$ and $x'$ be $(\catr,M,F)$-plates with the same result $c$.
  Then the $\catr$-plates $\cwise{F}(x)$ and $\cwise{F}(x')$ have 
  result $c$, so, by injectivity of $\catr$, they are equal.
  Injectivity of $F$ implies that $\cwise{F}$ is injective, so $x=x'$.

  For ($\Leftarrow$), observe that an $\catr$-plate is the same thing 
  as an $(\catr,M,\id_C)$-plate, with the same result.  Thus the 
  injectivity of the 
  class-family $\Delta_{\catr}(C,\id_C)$ means that $\catr$ is injective.
\end{proof}

\begin{proposition} \label{prop:rubfun}
  Each of the following is equivalent to Exponentiation.
  \begin{enumerate}
  \item \label{item:wideclass} For any rubric $\catr$ on a class $C$, the endofunction $\Gamma_{\catr}$
    restricts to one on $\psetset C$.
  \item  \label{item:wideclassfam} For any rubric $\catr$ on a class $C$, the endofunction $\Delta_{\catr}$
    restricts to one on $\famof{C}$.
     \item \label{item:injwideclass} For any injective rubric $\catr$ on a class $C$, the endofunction $\Gamma_{\catr}$
    restricts to one on $\psetset C$.
  \item  \label{item:injwideclassfam} For any injective rubric $\catr$ on a class $C$, the endofunction $\Delta_{\catr}$
    restricts to one on $\injfof{C}$.
  \end{enumerate}
\end{proposition}
\begin{proof}
  Assume Exponentiation.  The $\catr$-plates within a set form a set,
  giving~(\ref{item:wideclass}).  For any
  family $(M,F)$ within $C$, the $(\catr,M,F)$-plates form a set,
  giving~(\ref{item:wideclassfam}).

  Clearly~(\ref{item:wideclass}) implies~(\ref{item:injwideclass}),
  and~(\ref{item:wideclassfam}) implies~(\ref{item:injwideclassfam}) by \cref{prop:injpresinj}($\Rightarrow$).

  For any set $K$, define the  injective family
  $\makelitfam{K} \eqdef (K,\id_K)$, and the injective wide rubric $\makerub{K}$ on $\totall$ consisting
  of a single $K$-ary rule that sends $x \smin \totall^K$ to $(x)$.
  To deduce Exponentiation from~(\ref{item:injwideclass})
 or~(\ref{item:injwideclassfam}), let $A$ and $B$ be sets.  Then we
 have 
  \begin{eqnarray*}
    \Gamma_{\makerub{A}}(B) &  = & B^A \\
   \Delta_{\makerub{A}}(\makelitfam{B}) & = &
                                     (x)_{\tuple{*,x,*} \in 1 \times
                                     B^A\times 1} \\
    & \cong & \makefam{B^A} \quad \text{ via $\tuple{*,x,*} \mapsto x$.}
  \end{eqnarray*}
  So if either $\Gamma_{\makerub{A}}(B)$ is a set or $\Delta_{\makerub{A}}(\makelitfam{B})$
  is small, then $B^A$ is a set.
\end{proof}

Here is an application:
\begin{proposition} \label{prop:powallgen}
  Let $\catr$ be a rubric on a class $C$.  Suppose that either $C$ has an 
   $\catr$-complete subset, or Powerset or Collection holds.  Then
  $\Gamma_{\catr}$ has an inductive chain and a least prefixpoint.
\end{proposition}
\begin{proof}
  If $C$ has an $\catr$-complete subset, apply \cref{prop:stabord}.  If Powerset holds,
   apply
  \cref{prop:pressmall}(\ref{item:pressmallchain}) using
  \cref{prop:rubfun}(\ref{item:wideclass}).  If Collection holds, apply \ref{prop:collmu}.
\end{proof}

\psubsection{Proving the Injective Set Generation Principles} \label{sect:noac}

In order to prove the set generation principles for a rubric $\catr$, we establish a
relationship between the functions
$\Gamma_{\catr}$ and 
$\Delta_{\catr}$.
\begin{proposition} \label{prop:injrubcomm}
  Let $\catr$ be a rubric on a class $C$.
  The square
    \begin{displaymath}
      \xymatrix{
        \injcfof{C} \ar[rr]^-{\gamfamcatr} \ar[d]_{\mathsf{Range}} & & 
        \classfamof{C} \ar[d]^{\mathsf{Range}} \\
        \subclassof{C} \ar[rr]_-{\Gamma_{\catr}} & & \subclassof{C}
      }
    \end{displaymath} commutes.
\end{proposition}
\begin{proof}
Let  $(M,F)$ be an injective class-family within $C$.  Then
    $\cwise{F}$ is a result-preserving bijection from the class of all
    $(\catr,M,F)$-plates to that of all $\catr$-plates within
    $\rangeof{M,F}$.  So an element $c \smin C$ is the result of an
    $(\catr,M,F)$-plate iff it is the result of an $\catr$-plate within
    $\rangeof{M,F}$.  
\end{proof}

\begin{proposition}  \label{prop:injrub}
  Let $\catr$ be an injective rubric on a class $C$.
  \begin{enumerate}
 \item\label{item:injind} $\Gamma_{\catr}$ has an inductive chain and least prefixpoint.
  \item \label{item:injpresind} For every extended ordinal $\alpha$, the class-family
    $\mu^{\alpha}\Delta_{\catr}$ is injective, and its range is
    $\mu^{\alpha}\Gamma_{\catr}$.
\item \label{item:injrubrange} The class-family $\mu \Delta_{\catr}$ is injective, and its
  range is $\mu \Gamma_{\catr}$. 
  \end{enumerate}
\end{proposition}
\begin{proof}
    Firstly, for every extended ordinal $\alpha$, injectivity of
    $\mu^{\alpha}\Delta_{\catr}$ is proved by induction on $\alpha$,
    using \cref{prop:injpresinj}($\Rightarrow$) for the successor case.
    
   All that remains is to show that $(\rangeof{\mu^{\alpha}\Delta_{\catr}})_{\alpha \in
      \Ord}$ is an inductive chain for $\Gamma_{\catr}$, with supremum
    $\rangeof{\mu^{\subbiginfty} \Delta_{\catr}}$, as part
    (\ref{item:injrubrange}) follows by \cref{prop:stabbig}.

        The zero
    and extended limit requirements are
    obvious.  For the successor requirement, 
    $\rangeof{\mu^{\mathsf{S}\alpha} \Delta_{\catr}}$ is
      $\rangeof{\Gamma_{\catr} \mu^{\alpha} \Delta_{\catr}}$, which is
      $\Gamma_{\catr} \rangeof{\mu^{\alpha} \Delta_{\catr}}$ by \cref{prop:injrubcomm}.
\end{proof}

\begin{corollary} \label{cor:injrubset}
  Let $\catr$ be an injective rubric on a class $C$.  If
  $\derivsof{\catr}$ is a set, then $\catr$ generates a subset of $C$.
\end{corollary}

We arrive at our main result:
\begin{theorem} \label{prop:injequiv} \hfill
  \begin{enumerate} 
   \item \label{item:allwideequiv} Wide Infinity, Wide Derivation Set and Injective Wide
    \setgen{} are equivalent.
  \item \label{item:allbroadequiv} \full{} Broad Infinity, Broad Derivation Set and
    Injective Broad \setgen{} are equivalent.
   \end{enumerate} 
\end{theorem}
\begin{proof} 
  Follows from~\cref{prop:implderiv,prop:setgeninf,cor:injrubset}.
\end{proof}

\psubsection{Proving the Set Generation Principles, Assuming AC} \label{sect:rubac}

We again present the relationship between $\Gamma_{\catr}$ and
$\Delta_{\catr}$, this time assuming AC and ignoring injectivity.
\begin{proposition} \label{prop:useaccomm} \assumi{AC}
  Let $\catr$ be a rubric on a class $C$.
  \begin{enumerate}
  \item \label{item:aconlycomm} The square
       \begin{displaymath}
         \xymatrix{
           \famof{C} \ar[rr]^-{\Delta_{\catr}} \ar[d]_{\mathsf{Range}} & & 
           \classfamof{C} \ar[d]^{\mathsf{Range}} \\
           \psetset {C} \ar[rr]_-{\Gamma_{\catr}} & & \subclassof{C}
         }
       \end{displaymath}
       commutes.
   \item \label{item:accollcomm} If Collection holds, then the square
       \begin{displaymath}
         \xymatrix{
           \classfamof{C} \ar[rr]^-{\Delta_{\catr}} \ar[d]_{\mathsf{Range}} & & 
           \classfamof{C} \ar[d]^{\mathsf{Range}} \\
           \subclassof{C} \ar[rr]_-{\Gamma_{\catr}} & & \subclassof{C}
         }
       \end{displaymath}
       commutes.
  \end{enumerate}
\end{proposition}
\begin{proof} \hfill
  \begin{enumerate}
  \item \label{item:aconlycommproof} Let $(M,F)$ be a family within $C$.  Then $\cwise{F}$ is a
    result-preserving surjection from the class of all
    $(\catr,M,F)$-plates to that of all $\catr$-plates within
    $\rangeof{M,F}$.  To see surjectivity in the wide case (the broad
    case is similar), let $w= \tuple{i,[a_k]_{k \in K_i},p}$ be an
    $\catr$-plate within $\rangeof{M,F}$.  By AC and since $M$ is a
    set, we can choose for each $a_k$ an $F$-preimage $x_k \smin M$,
    and then $\tuple{i,[x_k]_{k \in K_i}, p}$ is a
    $\cwise{F}$-preimage of $w$.
  \item \label{item:accollcommproof} Since Collection $+$ AC  
    gives Collective
    Choice---Proposition~\ref{prop:collchoice}(\ref{item:classchoice}$\Leftarrow$)---we
    use the same argument as in part
    (\ref{item:aconlycomm}), except that $(M,F)$ is now a class-family. 
    \qedhere
  \end{enumerate}
\end{proof}

\begin{proposition} \assumi{AC}  \label{prop:chainac}
  Let $\catr$ be a rubric on a class $C$.  Suppose that  either 
  $\derivsof{\catr}$ is a set, or Powerset or Collection holds.
  \begin{enumerate}
  \item \label{item:acind} $\Gamma_{\catr}$ has an inductive chain and a least
    prefixpoint.
  \item \label{item:acpresind} For each extended ordinal $\alpha$, the range of $\mu^{\alpha}
    \Delta_{\catr}$ is $\mu^{\alpha} \Gamma_{\catr}$.
\item \label{item:acrange} The range of $\mu 
    \Delta_{\catr}$ is $\mu \Gamma_{\catr}$. 
  \end{enumerate}
\end{proposition}
\begin{proof}
  If $\derivsof{\catr}$ is a set, then so is $\mu^{\alpha}
  \Delta_{\catr}$, for every extended ordinal $\alpha$.   If Powerset
  holds, then $\mu^{\alpha}\Delta_{\catr}$ is a set for every ordinal
  $\alpha$, by induction on $\alpha$, using \cref{prop:rubfun}(\ref{item:wideclassfam}) for the successor case. 

  All that
  remains is to show
  that $(\rangeof{\mu^{\alpha}\Delta_{\catr}})_{\alpha \in
      \Ord}$ is an inductive chain for $\Gamma_{\catr}$ with supremum
    $\rangeof{\mu^{\subbiginfty} \Delta_{\catr}}$, as part
    (\ref{item:acrange}) follows by \cref{prop:stabbig}.

    The zero
    and extended limit requirements are
    obvious.  For the successor requirement, 
    $\rangeof{\mu^{\mathsf{S}\alpha} \Delta_{\catr}}$ is
      $\rangeof{\Gamma_{\catr} \mu^{\alpha} \Delta_{\catr}}$, which is
      $\Gamma_{\catr} \rangeof{\mu^{\alpha} \Delta_{\catr}}$ by
      \cref{prop:useaccomm}(\ref{item:aconlycomm}) in the case that 
      $\derivsof{\catr}$ is a set or Powerset holds, and by
      \cref{prop:useaccomm}(\ref{item:accollcomm}) in the case that Collection
      holds.
    \end{proof}

    \begin{corollary} \label{cor:acrubset} \assumi{AC}  Let $\catr$ be a rubric on a class
      $C$.  If  $\derivsof{\catr}$ is a set, then $\catr$ generates a subset of $C$.
    \end{corollary}

We arrive at our main result for those who accept AC.
\begin{theorem} \assumi{AC} \label{cor:acequiv} \hfill
  \begin{enumerate}
  \item  \label{item:acequivwide} Wide Infinity, Wide Derivation Set and
    Wide \setgen{} are equivalent.
  \item \label{item:acequivbroad} Broad Infinity, Broad Derivation Set
    and Broad \setgen{} are equivalent.
  \end{enumerate}
\end{theorem}
\begin{proof}
  Follows from~\cref{prop:implderiv,prop:setgeninf,cor:acrubset}.
\end{proof}


\psubsection{WISC Principles} \label{sect:wisc}

Our goal is to improve \cref{cor:acequiv}, by replacing AC
with a weak form of choice called WISC, originally studied
in~\cite{Streicher:wisc,vandenbergmoerdijk:amc}.  
We shall formulate three versions of WISC, using the following notions.
\begin{definition} \label{def:awisc}
  Let $K$ be a set.
  \begin{enumerate}
  \item \label{item:kcover} A \emph{$K$-cover} is a $K$-tuple of inhabited sets.
    More generally, a \emph{$K$-class-cover} is a $K$-tuple of
    inhabited classes.
  \item\label{item:kcoverpi}  For any $K$-class-cover $A$, the surjection $\pi \ccolc \sum A \rightarrow
    K$ sends $\tuple{k,a}$ to $k$.
  \item \label{item:kcovermap} Given $K$-class-covers $A$ and $B$, a \emph{map} $f \ccolc A
    \rightarrow B$  is a $K$-tuple of functions $[f_k \,\colon\, A_k \rightarrow B_k]_{k \in K}$.
  \item \label{item:awisc} A \emph{WISC} for $K$ is a set $\cata$ of $K$-covers that is weakly initial---i.e., for
    any $K$-cover $B$, there is $A \smin
    \cata$ and a map $f\ccolc A \rightarrow B$.
  \end{enumerate}
\end{definition}

Although the definition of WISC does not mention class-covers, only
covers, we note the following result.
\begin{proposition} \label{prop:classcov} \assumi{Collection} Let $K$
  be a set, and 
  $\cata$ a WISC for $K$.  Then, for any
  $K$-class-cover $B$, there is $A \smin \cata$ and a map $f \ccolc A
  \rightarrow B$.
\end{proposition}
\begin{proof}
  Collection yields an element $X \smin \prod_{k \in
    K} \psetinh B_k$.  Since $X$ is a $K$-cover, there is $A \smin
  \cata$ and a map $f \ccolc A
  \rightarrow X$, so we have  $f \ccolc A
  \rightarrow B$.
\end{proof}

We continue without assuming Collection.
\begin{definition} \hfill \label{def:wiscfun}
  Let $\catk$ be a class of sets.   A \emph{WISC function} on
    $\catk$ sends each $K \smin \catk$ to a WISC for $K$.  
\end{definition}


 Now consider the following principles.
\begin{itemize}
\item \emph{Simple WISC}: Every set has a WISC.
\item \emph{Local WISC}: Every set of sets has a WISC function.
\item \emph{Global WISC}: The class $\allsets$ has a WISC function.
\end{itemize}
These principles are related as follows.
\begin{proposition} \hfill \label{prop:wiscprops}
  \begin{enumerate}
  \item \label{item:relwisc} AC implies Global WISC, which implies Local WISC, which is
    equivalent to Simple WISC.
   \item\label{item:zfwisc} {\bf In ZF} the three WISC principles are equivalent.
  \end{enumerate}
\end{proposition}
\begin{proof} \hfill
  \begin{enumerate}
  \item \label{item:relwiscproof} To show that AC implies Global WISC: note that AC is equivalent to $K
    \mapsto \setbr{[1]_{k \in K}}$ being a global WISC function.

    To show that Global WISC implies Local WISC: given a set of sets $\catk$, restrict the global WISC
    function to it.

    To show that  Local WISC implies Simple WISC: for any set $K$, 
    obtain a WISC function for the singleton $\setbr{K}$ and apply
    it to $K$.

    To show that Simple WISC implies Local WISC: let $\catk$ be a set of sets, and
    write $L \eqdef \sum_{K \in \catk} K$.  For each $K \smin \catk$,
    define $\funcok$ to be the functor from the
    category of $L$-covers to that of $K$-covers, sending $[A_l]_{l
      \in L}$ to $[A_{\tuple{K,k}}]_{k \in K}$ and likewise for
    maps.  
  Given an WISC $\cata$ for $L$, define $f$ to be the function sending $K \smin \catk$ to the set of  $K$-covers
  \begin{math}
    \setbr{ \funcok(A) \mid A \smin \catk }
  \end{math}, 
 which is weakly initial by the following argument.  Any $K$-cover $B$
 is equal to $\funcok(C)$, where
 $C$ is the $L$-cover whose   $\tuple{M,k}$-component is $B_k$ if $M
 = K$ and $1$ otherwise.  Weak initiality of $\cata$ gives an $L$-cover $A \smin \cata$ and map
  $g \ccolc A \rightarrow C$, so we obtain $\funcok(A) \in f(K)$ and
  $\funcok(g) \ccolc \funcok(A) \rightarrow \funcok(C) = B$.
  Thus $f$ is a WISC function on $\catk$.
\item \label{item:zfwiscproof} We write $(V_{\alpha})_{\alpha \in \Ord}$ for the cumulative
  hierarchy in the usual way.  Suppose Simple WISC holds. For each set $K$, 
  define $t(K)$ to be the least ordinal $\alpha$ such that the set 
   $(\psetinh V_{\alpha})^K$ of all ``$\alpha$-bounded'' $K$-covers is weakly initial.  Then $K
    \mapsto (\psetinh V_{t(K)})^K$ is a global WISC function.    \qedhere
  \end{enumerate}
\end{proof}

\noindent It has been shown that the theory ZF $+$ WISC
is strictly between ZF and ZFC, provided ZF is
consistent~\cite{Karagila:embedorder,Roberts:wiscfail}.  For 
applications of WISC,
see~\cite{vandenbergmoerdijk:amc,FiorePittsSteenkamp:quotinduct,PittsSteenkamp:constructinit}.

\psubsection{Proving the Set Generation Principles, Assuming WISC} \label{sect:suffwisc}

Our task is to weaken the AC assumption of \cref{cor:acequiv}.
Specifically, we shall prove that Local WISC suffices for
part~(\ref{item:acequivwide}), and Global WISC for part~(\ref{item:acequivbroad}).

First we shall provide some constructions.  For a rubric $\catr$ on a class $C$, and any WISC function $f$ on
$\arsof{\catr}$, we shall construct an ``extended'' rubric $\catr_f$ on
$C$.  This is done as follows, using Proposition~\ref{prop:basicsurj}.
\begin{definition} Let $C$ be a class. \label{def:rubwisc}
  \begin{enumerate}
  \item \label{item:widerulear} Let $\tuple{K,R}$ be a wide rule on $C$.  For any $K$-cover
    $A$, 
    define the wide rule $\tuple{K,R}^A$ consisting of the arity $\sum
    A$ and function $C^{\sum A}
    \rightarrow\famof{C}$ sending $C^{\pi}(x)$ to
    $R (x)$ and
    everything else to the empty family.
  \item \label{item:widerubar} Let $\catr = (\tuple{K_i,R_i})_{i \in I}$ be a wide rubric on
    $C$. For any WISC function $f$ on $\arsof{\catr}$, define
   the wide rubric
   \begin{eqnarray*}
 \catr_f & \eqdef & (\tuple{K_i,R_i}^{A})_{i \in I, A
      \in f(K_i)}
   \end{eqnarray*}
  \item \label{item:broadrulear} Let $\tuple{L,S}$ be a broad rule on $C$.  For any $L$-cover $B$
    and WISC function $f$ on $\arsof{S}$,  
    define the broad rule $\tuple{L,S}_f^B$ consisting of the arity
    $\sum B$ and function $C^{\sum B} \rightarrow
    \allrub{C}$ sending $C^{\pi}(x)$ to
   $S(x)_{\wiscres{f}{S(x)}}$ 
    and
    everything else to the empty rubric.
  \item \label{item:broadrubar} Let $\cats = (\tuple{L_j,S_j})_{j \in J}$ be a broad rubric on
    $C$.  For any WISC function
    $f$ on $\arsof{\cats}$, define the broad rubric
     \begin{eqnarray*}
 \cats_ f & \eqdef & (\tuple{L_j,S_j}^B_{\wiscres{f}{{\arsof{S_j}}}})_{j \in J, B
      \in f(L_j)}
   \end{eqnarray*}
 \end{enumerate}
\end{definition}


 Now we adapt~\cref{prop:useaccomm} as
 follows.
 \begin{proposition} \label{prop:commfromwisc}
   Let $\catr$ be a rubric on a
   class $C$, and $f$ a WISC function on $\arsof{\catr}$.
   \begin{enumerate}
   \item \label{item:commwiscsmall} The square \begin{displaymath}
      \xymatrix{
        \famof{C} \ar[rr]^-{\Delta_{\catr_f}} \ar[d]_{\mathsf{Range}} & & 
        \classfamof{C} \ar[d]^{\mathsf{Range}} \\
        \psetset {C} \ar[rr]_-{\Gamma_{\catr}} & & \subclassof{C}
      }
    \end{displaymath}
    commutes.
   \item \label{item:commwisccoll} If Collection holds, then the square
    \begin{displaymath}
      \xymatrix{
        \classfamof{C} \ar[rr]^-{\Delta_{\catr_f}} \ar[d]_{\mathsf{Range}} & & 
        \classfamof{C} \ar[d]^{\mathsf{Range}} \\
        \subclassof{C} \ar[rr]_-{\Gamma_{\catr}} & & \subclassof{C}
      }
    \end{displaymath}
    commutes.
   \end{enumerate}
 \end{proposition}
 \begin{proof} \hfill
   \begin{enumerate}
   \item \label{item:commwiscsmallproof} We just prove the wide case,
      as the broad case is similar.  Let $(M,F)$ be a family within $C$.     For any $(\catr_f,M,F)$-plate
          \begin{eqnarray*}
m & =  & \tuple{\tuple{i,A},
          [x_{k,a}]_{k \in K_i, a \in A_k}, p}
          \end{eqnarray*}
the tuple $[F(x_{k,a})]_{k \in K_i,  a \in A_k}$ is (uniquely)
expressible  as $C^{\pi}(u)$, and we obtain an $\catr$-plate
        $\cowise{F}(m) \eqdef \tuple{i,u,p} $.

        We see next 
        that $\cowise{F}$ is a result-preserving surjection from the
        class of all  $(\catr_f,M,F)$-plates to that of all $\catr$-plates within
        $\rangeof{M,F}$.  To prove surjectivity, let $n
        = \tuple{i,y,p} $ be an $\catr$-plate within
        $\rangeof{M,F}$.  We obtain a $K_i$-cover $[F^{-1}(y_k)]_{k \in K_i}$, so
        there is $A \smin f(K_i)$ and a $K$-cover map $g \ccolc A
        \rightarrow [F^{-1}(y_k)]_{k \in K_i}$.  We then obtain an
        $(\catr_f,M,F)$-plate
        \begin{eqnarray*}
m & = & \tuple{\tuple{i,A},
          [g_k(a)]_{k \in K_i, a \in A_k}, p}
        \end{eqnarray*}
Since $F(g_k(a)) =
        y_k$ for all $k \smin K_i$ and $a \smin A_k$, we have
        $\cowise{F}(m) = n$.   

So an element $c \smin C$ is the result of an $(\catr_f,M,F)$-plate
iff it is the result of an $\catr$-plates within $\rangeof{M,F}$.
   \item \label{item:commwisccollproof} ISimilar to part (\ref{item:commwiscsmall}), except that we
     use \cref{prop:classcov} and speak of classes rather than sets.
   \qedhere
   \end{enumerate}
 \end{proof}
 
 \begin{proposition} \label{prop:wiscchain}
   Let $\catr$ be a rubric on a
   class $C$, and $f$ a WISC function on $\arsof{\catr}$.  Suppose
   that either $\derivsof{\catr_f}$ is a set, or Powerset or Collection holds.
   \begin{enumerate}
   \item \label{item:wiscchain} $\Gamma_{\catr}$ has an inductive chain and a least
    prefixpoint.
  \item \label{item:wiscpreschain} For each extended ordinal $\alpha$, the range of $\mu^{\alpha}
    \Delta_{\catr_f}$ is $\mu^{\alpha} \Gamma_{\catr}$.
\item\label{item:wiscrange} The range of $\mu 
    \Delta_{\catr_f}$ is $\mu \Gamma_{\catr}$. 
   \end{enumerate}
 \end{proposition}
 \begin{proof}
   Similar to the proof of Proposition~\ref{prop:chainac}, using \cref{prop:commfromwisc} rather
   than \cref{prop:useaccomm}.
 \end{proof}

 \begin{corollary} \label{cor:wiscrubset} 
   Let $\catr$ be a rubric on a
   class $C$, and $f$ a WISC function on $\arsof{\catr}$.   If  $\derivsof{\catr_f}$ is a set, then $\catr$ generates a subset of $C$.
 \end{corollary}
The key question is whether $\arsof{\catr}$ has a WISC function,
which we answer as follows.
\begin{proposition} \label{prop:aritywisc} \hfill
  \begin{enumerate} 
  \item \label{item:localwiscrub} Local WISC is equivalent to the assertion ``For any wide
    rubric $\catr$ on a class $C$, the set $\arsof{\catr}$ has a WISC function.''
  \item \label{item:globalwiscrub} Global WISC is equivalent to the assertion: ``For any broad
    rubric $\catr$ on a class $C$, the class $\arsof{\catr}$ has a WISC function.''
  \end{enumerate}
\end{proposition}
\begin{proof}
  Since ($\Rightarrow$) is obvious, we just prove ($\Leftarrow$).
  \begin{enumerate}
  \item \label{item:localwiscrubproof} Given a set of sets $\catk$, define $\catr$ to be
    the following
    wide rubric on $\emptyset$: it is indexed by $\catk$, and rule $K$ has arity $K$. Then
    $\arsof{\catr} =  \catk$.
  \item \label{item:globalwiscrubproof} Define $\cats$ to be the broad rubric on $\allsets$ consisting
    of a single unary rule, sending $[X]$ to the wide rubric
    consisting of a single $X$-ary rule, sending every tuple to empty
    family.  Then $\arsof{\cats} = \allsets$.
  \qedhere
  \end{enumerate}
\end{proof}

We obtain our main result:
\begin{theorem}  \label{cor:wiscequiv} \hfill
  \begin{enumerate}
  \item \label{item:widewiscequiv} \assumi{Local WISC} Wide Infinity, Wide Derivation Set and
    Wide \setgen{} are equivalent.
  \item \label{item:broadwiscequiv} \assumi{Global WISC} Broad Infinity, Broad Derivation Set and Broad \setgen{} are equivalent.
  \end{enumerate}
\end{theorem}
\begin{proof}
  From Propositions~\ref{prop:implderiv} and \ref{prop:setgeninf}, 
  and~\cref{cor:wiscrubset} using \cref{prop:aritywisc}.
\end{proof}

\ppart{Wide and Broad principles for ordinals} \label{part:widebroadord}


 \psection{From Rubrics to \Supgeneration{}} \label{sect:supgen}


Our next task will be to adapt Wide and Broad \setgen{} into 
similar principles for ordinals.  An extended ordinal  
that is neither 0 nor an successor is called an \emph{extended limit}.  Recall that 
Definition~\ref{def:clcomplclass} gave us the notion of  a class being
closed  or complete.  Here are analogous properties for extended ordinals:
\begin{definition}  \label{def:supclosed} An extended limit $\lambda$ is
  \begin{itemize}
  \item \emph{$K$-\supclosed{}}, for a set $K$, when 
    $\supo_K$ (or equivalently $\ssup_K$)
    restricts to a
    function $[0 \twodots \lambda)^{K} \rightarrow [0 \twodots \lambda)$.  
  \item \emph{$\catk$-\supcomplete{}}, for a class of sets $\catk$, when it is
    $K$-\supclosed{} for all $K \smin \catk$.
  \item \emph{$F$-\supclosed{}}, for a function  $F \ccolc \Ord \rightarrow \allsets$, when it is $F\alpha$-\supclosed{} for all $\alpha < \lambda$.
  \item \emph{$H$-\supcomplete{}}, for a function $H \ccolc \Ord \rightarrow
    \subclassof{\allsets}$, when it is $H\alpha$-\supcomplete{} for all $\alpha
    < \lambda$. 
  \end{itemize}
\end{definition}
\noindent Here are some examples.  
\begin{enumerate}
\item \label{item:constksup} For a class of sets $\catk$, let $\constk$ be the constant
  function $\gamma \mapsto \catk$.   An extended limit is $\constk$-\supcomplete{}
   iff it is $\catk$-\supcomplete{}.
\item \label{item:veesup} For functions $H, H' \ccolc \Ord \rightarrow \subclassof{\allsets}$,
  let $H \vee H'$ be the pointwise union $\gamma \mapsto H(\gamma)
  \cup H'(\gamma)$.  An extended limit is $(H \vee H')$-\supcomplete{} iff it is
  both $H$-\supcomplete{} and $H'$-\supcomplete{}.
\end{enumerate}
Below (\cref{prop:kclosed}) we shall characterize \supclosed{}ness and
\supcomplete{}ness in an explicit way.



\begin{definition}  \label{def:supgen} \hfill
  \begin{enumerate}
  \item \label{item:simplysupgen} For a set $K$, the \emph{simply
      $K$-\supgenerat{}ed} extended limit is the  least $K$-\supclosed{} one.
  \item \label{item:supgen} For a class of sets $\catk$, the
    \emph{$\catk$-\supgenerat{}ed} extended limit is the 
    least $\catk$-\supcomplete{} one.
 \item \label{item:simplysupgenbroad} For a function $F \ccolc \Ord \rightarrow \allsets$, the \emph{simply
     $F$-\supgenerat{}ed}
   extended limit is the least
   $F$-\supclosed{} one.
 \item  \label{item:supgenbroad} For a function  $H \ccolc \Ord \rightarrow
   \subclassof{\allsets}$, the \emph{$H$-\supgenerat{}ed} extended limit 
  is the 
   least $H$-\supcomplete{} one.
  \end{enumerate}
\end{definition}


\noindent This leads to the following principles.
\begin{itemize}
\item \emph{Simple Wide \Supgeneration{}}: Any set simply
  \supgenerat{}es a limit ordinal.
\item \emph{\full{} Wide \Supgeneration{}}: Any set of sets 
  \supgenerat{}es a limit ordinal.
\item \emph{Simple Broad \Supgeneration{}}: Any function $\Ord
  \rightarrow \allsets$ simply \supgenerat{}es a limit ordinal.
\item \emph{\full{} Broad \Supgeneration{}}: Any function  $\Ord \rightarrow
    \psetset \allsets$ \supgenerat{}es a limit ordinal.
 \end{itemize}

\begin{theorem} \hfill \label{thm:supgenequiv}
  \begin{enumerate}
  \item \label{item:widesupequiv} The two forms of Wide \Supgeneration{} are equivalent.
  \item \label{item:broadsupequiv} The two forms of Broad \Supgeneration{} are equivalent.
  \item \label{item:broadwidesup} Full Broad \Supgeneration{} implies Full Wide \Supgeneration{}.
  \end{enumerate}
\end{theorem}
\begin{proof} \hfill
  \begin{enumerate}
  \item \label{item:widesupequivproof} \full{} $\Rightarrow$ Simple is obvious.  For the converse, we
    first note that, for a set of sets $\catk$, any extended limit 
    $\lambda$ that is 
    $\sum_{K \in \catk} K$-\supclosed{}  is
    $\catk$-\supcomplete{}.  That is because, for $K \smin \catk$ and
    $p \smin [0\twodots \lambda)^{K}$, we have
    \begin{eqnarray*}
      \bigvee_{k \in K} p_k & = & \bigvee_{\tuple{L,K} \in \sum_{K \in \catk}
                                  K} \left\{
                                  \begin{array}{ll}
                                    p_k  & \text{ $(L=K)$} \\
                                    0 & \text{ otherwise}
                                  \end{array} \right. \\
                            & < & \lambda
    \end{eqnarray*}
 \item  \label{item:broadsupequivproof} \full{} $\Rightarrow$ Simple is obvious.  For the converse: for a function
    $H \ccolc \Ord \rightarrow \psetset \allsets$, any extended limit 
    that is  $(\beta \mapsto \sum_{K \in H\beta}K)$-\supclosed{} is
    $H$-\supcomplete{}, as before.
  \item \label{item:broadwidesupproof} Given a set of sets $\catk$, the $\constk$-\supgenerat{}ed
    extended limit is $\catk$-\supgenerat{}ed.
    \qedhere
  \end{enumerate}
\end{proof}

Now we give the relationship between \supgenerat{}ion and set generation.
\begin{theorem}  \hfill \label{prop:setsup}
  \begin{enumerate}
  \item \label{item:widesetsup} Wide \setgen{} is equivalent to Powerset $+$ Full Wide \Supgeneration{}.
  \item \label{item:broadsetsup} Broad \setgen{} is equivalent to
    Powerset $+$ Full Broad \Supgeneration{}.
  \end{enumerate}
\end{theorem}
\begin{proof} \hfill
  \begin{enumerate}
  \item \label{item:widesetsupproof} For ($\Rightarrow$), we have Powerset by \cref{prop:powinfty} and
    \cref{prop:setgeninf}(\ref{item:widesginf}).  To show Full Wide \Supgeneration{}, let $\catk$ be a set of sets.
    A $\catk$-\supcomplete{} ordinal is an $\catr$-inductive subset of
    $\Ord$, where the wide rubric $\catr$ on $\Ord$ consists of the
    following.
    \begin{itemize}
    \item The unary rule sending $[\alpha]$ to
      $(\beta)_{\beta < \alpha}$.  (Note that a set is closed under
      this rule iff it is transitive, i.e., an ordinal.)
    \item For each $K \smin \catk$, a $K$-ary rule sending
      $[\alpha_k]_{k \in K}$ to $(\bigvee_{k \in K} \alpha_k)$.
    \end{itemize}
For ($\Leftarrow$),  let $\catr = (\tuple{K_i,R_i})_{i \in I}$ be a wide
    rubric on $C$.  We show that the inductive chain of
    $\Gamma_{\catr}$ (which preserves smallness by Exponentiation and
    \cref{prop:rubfun}(\ref{item:wideclass})) stabilizes
    at an ordinal $\alpha$.  Define the set of sets
    \begin{eqnarray*}
      \catk & \eqdef & \setbr{K_i \mid i \smin I}
    \end{eqnarray*}
    Let $\alpha$ be a $\catk$-\supcomplete{} limit ordinal.  For any
    $x \smin \mu^{\alpha} \Gamma_{\catr}$, put $\overline{x}$ for the
    unique $\beta < \alpha$ such that
    $x \in \mu^{\mathsf{S}\beta} \Gamma_{\catr} \setminus \mu^{\beta}
    \Gamma_{\catr}$.  Given an $\catr$-plate
    \begin{eqnarray*}
      w & = & \tuple{i,x, p}
    \end{eqnarray*}
    within $\mu^{\alpha} \Gamma_{\catr}$, put
    \begin{eqnarray*}
      \beta & \eqdef & \ssup_{l \in L_j}\overline{x_l}  
    \end{eqnarray*}
    Since $\alpha$ is $K_i$-\supclosed{}, and
    $\rnkof{y_k}{\Gamma_{\catr}} < \alpha$ for all $k \smin K_i$, we
    obtain $\beta < \alpha$.  Since $w$ is an $\catr$-plate within
    $\mu^{\beta} \Gamma_{\catr}$, its result is in
    $\mu^{\succof{\beta}} \Gamma_{\catr}$, which is included in
    $\mu^{\alpha} \Gamma_{\catr}$ since $\succof{\beta} < \alpha$.

\item \label{item:broadsetsupproof}  For ($\Rightarrow$), to show Full Broad \Supgeneration{}, let
    $H \ccolc \Ord \rightarrow \psetset \allsets$ be a function.  An
    $H$-\supcomplete{} ordinal is an $\catr$-inductive subset of
    $\Ord$, where the broad rubric $\catr$ on $\Ord$ consists of the
    following.
    \begin{itemize}
    \item The nullary rule returning the wide rubric consisting of
      just the unary rule sending $[\alpha]$ to
      $(\beta)_{\beta < \alpha}$.  (Note that a set is closed under
      this rule iff it is transitive, i.e., an ordinal.)
    \item The unary rule sending $[\alpha]$ to the wide rubric
      consisting of, for each $K \smin H\alpha$, the $K$-ary rule
      sending $[\alpha_k]_{k \in K}$ to
      $(\bigvee_{k \in K} \alpha_k)$.
    \end{itemize}
For ($\Leftarrow$),  let $\cats = (\tuple{L_j,S_j})_{j \in J}$ be a broad
rubric on $C$. We show that the inductive chain of
    $\Gamma_{\cats}$ (which preserves smallness by Exponentiation and
    \cref{prop:rubfun}(\ref{item:wideclass})) stabilizes
    at an ordinal $\alpha$. Define the set 
    \begin{eqnarray*}
      \catl & \eqdef & \setbr{L_j \mid j \smin J}
    \end{eqnarray*}
    and the function $H \ccolc \Ord \rightarrow \psetset \allsets$
    sending $\beta$ to
    \begin{displaymath}
      \bigcup_{j \in J}  \bigcup_{\substack{y \in (\mu^{\beta}\Gamma_{\cats})^{L_j}
          \\ S_j(y) = (\tuple{K_i,R_i})_{i \in I}}} \setbr{K_i \mid i \smin I}  
    \end{displaymath}
    Let $\alpha$ be a limit ordinal that is
    $(\constl \vee H)$-\supcomplete{}---i.e., both
    $\catl$-\supcomplete{} and $H$-\supcomplete{}.  For any
    $x \smin \mu^{\alpha} \Gamma_{\cats}$, put $\overline{x}$ for the
    unique $\beta < \alpha$ such that
    $x \in \mu^{\mathsf{S}\beta} \Gamma_{\cats} \setminus \mu^{\beta}
    \Gamma_{\cats}$.  Given an $\cats$-plate
    \begin{eqnarray*}
      w & = & \tuple{j,    y,i, x, p}
    \end{eqnarray*}
    within $\mu^{\alpha} \Gamma_{\cats}$, put
    \begin{eqnarray*}
      \gamma & \eqdef & \ssup_{l \in L_j} {\overline{y_l}}  \\
      \beta & \eqdef & \ssup_{k \in K_i} \overline{x_k}
    \end{eqnarray*}
    Since $\alpha$ is $L_j$-\supclosed{}, and
    $\rnkof{y_l}{\Gamma_{\cats}} < \alpha$ for all $l \smin L_j$, we
    obtain $\gamma < \alpha$.  We have
    $y \in (\mu^{\gamma} \Gamma_{\cats})^{L_j}$ and
    $S_j(y) = (\tuple{K_i,R_i})_{i \in I}$, so $K_i \in H\gamma$, so
    $\alpha$ is $K_i$-\supclosed{}.  Since
    $\rnkof{x_k}{\Gamma_{\cats}} < \alpha$ for all $k \smin K$, we
    obtain $\beta < \alpha$.  Since $w$ is an $\cats$-plate within
    $\mu^{\gamma \vee \beta} \Gamma_{\cats}$, its result is in
    $\mu^{\succof{(\gamma \vee \beta)}} \Gamma_{\cats}$, which is
    included in $\mu^{\alpha} \Gamma_{\cats}$ since
        $\succof{(\gamma \vee \beta)} < \alpha$. \qedhere
  \end{enumerate}
\end{proof}

\psection{Lindenbaum Numbers}  \label{sect:lind}

We interrupt our journey towards Mahlo's principle to give some useful 
constructions that relate sets to ordinals.  First we give some notation:
\begin{definition} \label{def:injsurjcomp}
  Let $A$ and $B$ be sets.
  \begin{enumerate}
  \item \label{item:injcomp} We write $A \injcomp B$ when there is an injection $A \rightarrow B$.
  \item \label{item:surjcomp} We write $A \surjcomp B$ when there is a partial surjection
    $B \rightarrow A$.  Equivalently: when either $A = \emptyset$ or
    there is a surjection $B \rightarrow A$.
  \end{enumerate}
\end{definition}
\noindent Thus $A \injcomp B$ implies $A \surjcomp B$, and conversely
if  AC holds or $B$ is
well-orderable.

\begin{definition} \label{def:hartlind}
  Let $K$ be a set.
  \begin{enumerate}
  \item \label{item:hart} Define $R$ to be the class of all pairs $(X,<)$, consisting of
    a subset $X$ of $K$, and a well-order $<$ on $X$.   Then we obtain a
    lower class
  \begin{displaymath}
      \setbr{\ordtype{(X,<)} \mid (X,<) \smin R}
    \end{displaymath}
    whose strict supremum is denoted $\aleph(K)$.
   \item \label{item:lind} Let $S$ be the class of all triples $(X,\sim,<)$, consisting
     of a subset $X$ of $K$, and an equivalence relation $\sim$ on $X$,
    and a well-order $<$ on the set $X/\sim$ of all equivalence
    classes.  Then we obtain a lower class
 \begin{displaymath}
      \setbr{\ordtype{(X/\sim,<)} \mid (X,\sim,<) \smin S}
    \end{displaymath}
       whose strict supremum is denoted $\alestar(K)$.
  \end{enumerate}
\end{definition}
The extended ordinals $\aleph(K)$ and $\alestar(K)$ are called the \emph{Hartogs number} and the
\emph{Lindenbaum number} of $K$,
respectively.\footnote{See~\cite{KaragilaRyanSmith:hartlind} for
  analysis of the range of possibilities.}  Note that $\aleph(K)
\leqslant \alestar(K)$, with equality if  AC holds or $K$ is well-orderable.  For any ordinal $\alpha$, we have $\alpha \injcomp K$ iff
     $\alpha < \aleph(K)$, and $\alpha \surjcomp K$ iff $\alpha
     < \alestar(K)$.  Thus we have  $\aleph(K) \not \injcomp K$ and
     $\alestar(K) \not \surjcomp K$.

 These constructions are often applied to an ordinal $\beta$, giving $\beta < \aleph(\beta) \leqslant \alestar(\beta)$.

For the sake of \cref{sect:blassmahlo} below, we introduce some axioms about Lindenbaum
numbers.  The first is 
\emph{Full Lindenbaum}: For any set $K$, the
 extended ordinal $\alestar(K)$ is an ordinal.
\begin{proposition} \label{prop:implales} \hfill
  \begin{enumerate}
  \item \label{item:psetales} Powerset implies Full Lindenbaum.
  \item \label{item:widelimales} Wide \Supgeneration{} implies Full Lindenbaum.
  \end{enumerate}
\end{proposition}
\begin{proof} \hfill
  \begin{enumerate}
  \item  \label{item:psetalesproof} By Powerset, the classes $R$ and $S$ in Definition~\ref{def:hartlind} are
    sets.
  \item \label{item:widelimalesproof} For a set $K$, let $\lambda$ be a $K$-closed limit ordinal.  If
    $\lambda < \alestar(K)$, then there is a surjection $f \ccolc K
    \rightarrow \lambda$, so $\lambda =
    \sup_{k \in K} f(k) < \lambda$, contradiction.  Thus $\alestar(K)
    \leqslant \lambda$, so $\alestar(K)$ is an ordinal.  \qedhere
  \end{enumerate}
\end{proof}
We divide Full Lindenbaum into two parts:
  \begin{itemize}
    \item \emph{Ordinal Lindenbaum}: For any ordinal
      $\alpha$, the extended ordinal $\alestar(\alpha)$ is an ordinal.
    \item \emph{Relative Lindenbaum}: For any set $K$, there
      is an ordinal $\alpha$ such that $\alestar(K) \subseteq \alestar(\alpha)$.
    \end{itemize}

    \begin{proposition} \label{prop:decompfullales}
      Full Lindenbaum is equivalent to Ordinal Lindenbaum $+$
      Relative Lindenbaum.
    \end{proposition}
    \begin{proof}
       ($\Leftarrow$) is obvious, and clearly Full Lindenbaum implies
    Ordinal Lindenbaum.   To show that it implies Relative 
    Lindenbaum, put $\alpha \eqdef \alestar(K)$ so that $\alestar(K) <
    \alestar(\alpha)$.
    \end{proof}

    Lastly we consider the \emph{Well-orderability} axiom:  Every set is
    well-orderable.  This principle has the following properties:
    \begin{proposition} \label{prop:woprop} \hfill
      \begin{enumerate}
  \item AC $+$ Powerset implies 
    Well-orderability.
\item \label{item:worl} Well-orderability implies AC $+$ Relative Lindenbaum.
  \end{enumerate}
\end{proposition}
\begin{proof}  We prove only Well-orderability
  $\Rightarrow$ Relative Lindenbaum, as the rest is standard.  Given a set $K$, define $\alpha$ to be the
  least 
  order-type of a well-ordering of $K$.  Since $\alpha \cong K$, we
  have  
    $\alestar(K)  = \alestar(\alpha)$. 
\end{proof}

\psection{From \Supgeneration{} to Mahlo's Principle}  \label{sect:mahlo}
\psubsection{Unbounded and Stationary Classes} \label{sect:unbstat}

Now at last, it is time to treat Mahlo's principle; but we approach
it more slowly than in \Cref{sect:broadvsmahlo}.   To begin, we revisit the notions of unbounded and stationary
class from \Cref{sect:regstat}.
\begin{definition} \label{def:cofinal}
  Let $A$ be a set-based well-ordered class.
  \begin{enumerate}
  \item \label{item:cofinal} A subclass $B$ is  \emph{cofinal} when, for all $x \smin
    A$, there is $y \smin B$ such that $y \geqslant x$.  Equivalently: 
    when it has no strict upper bound.
  \item \label{item:strictlycof} A subclass $B$ is \emph{strictly cofinal} when, for all $x \smin
    A$, there is $y \smin B$ such that $y > x$.   Equivalently: when it
    has no upper bound.
  \end{enumerate}
\end{definition}
If $A$ has no greatest element (e.g., when $A = \Ord$), then ``cofinal'' and
``strictly cofinal'' are equivalent, and the word ``unbounded'' is
also used.

We turn next to ordinal functions.
\begin{definition} \label{def:basedclosed}
  An extended limit $\lambda$ is
  \begin{itemize}
   \item \emph{$G$-based}, for a function  $G \ccolc \Ord \rightarrow
    \extord$, when, for all   $\alpha < \lambda$, we have $G(\alpha) \leqslant \lambda$.
  \item \emph{$F$-closed}, for a function $F  \ccolc \Ord \rightarrow
    \Ord$, when, for all $\alpha < \lambda$, we have
    $F(\alpha) < \lambda$.  In short: when $F$ restricts to an endofunction
    on $\lambda$.
  \end{itemize}
\end{definition}

\begin{proposition} \label{prop:basedclosed}
  For a class of limit ordinals $D$, the following are equivalent.
  \begin{enumerate}
  \item \label{item:fbased} For every function $G \ccolc \Ord \rightarrow \Ord$, there is
    a $G$-based limit ordinal in $D$.
  \item \label{itemfclosed} For every function $F \ccolc \Ord \rightarrow \Ord$, there is
    a $F$-closed limit ordinal in $D$.
  \end{enumerate}
\end{proposition}
\begin{proof}
  Since $F$-closed is the same as $\mathsf{S}F$-based and implies $F$-based.
\end{proof}

A class of limit ordinals with these properties is said to be
\emph{stationary}.  It is then unbounded, and, for any function
$F \ccolc \Ord \rightarrow \Ord$, contains stationarily many
$F$-closed elements.  Here is an application:
\begin{proposition} \label{prop:infstat}
  Each of the following is equivalent to Infinity.
  \begin{enumerate}
 \item \label{item:infunb} $\limclass$ is unbounded.
  \item \label{item:infstat} $\limclass$ is stationary.
  \end{enumerate}
\end{proposition}
\begin{proof}
  If Infinity does not hold, then $\limclass$ is empty.

  Assume Infinity.  To show $\limclass$ is stationary, let $F \ccolc
  \Ord \Rightarrow \Ord$.  Define $G \ccolc \Ord \Rightarrow
  \Ord$ sending $\alpha$ to $\ssup_{\beta < \alpha} (\succof{\beta}
  \vee 
  F\beta)$.  We show that $\lambda \eqdef \bigvee_{n \in \nats}
  G^{n}(1)$ is an $F$-closed limit ordinal.  Firstly, $0 < 1 =
  G^{0}(1) \leqslant 
  \lambda$.  If $\beta <
  \lambda$ then there is $n \smin \nats$ such that $\beta <
  G^{n}(1)$, so $\succof{\beta} < G^{n+1}(1) \leqslant \lambda$, and
  likewise $F\beta < \lambda$. 
\end{proof}

\psubsection{Cofinality} \label{sect:cofin}

The treatment of cofinality relies on the following result:
    \begin{proposition} \label{prop:subsetcof}
  For any extended ordinal $\alpha$ and function $f \ccolc [0 \twodots
  \alpha) \rightarrow
  \Ord$, the range of $f$ has a cofinal subclass of order-type
  $\leqslant \alpha$.
\end{proposition}
\begin{proof}
Let $K$ be the class of all $i < \alpha$ such that $f(i)$ is a strict
upper bound of $\setbr{f(j) \mid j < i}$.  We prove by induction on $i
< \alpha$ that there
is $k \leqslant i$ such that $k \smin K$ and $f(k) \geqslant f(i)$, as
follows.
  If $i \in K$, put $k \eqdef i$, and if not, then there is $j < i$ such that
 $f(i) \leqslant f(j)$, and we apply the inductive hypothesis to it.

Thus the range of $f \restriction_{K}$ is cofinal within that of $f$,
and (since $f \restriction_K$ is
strictly monotone) has the same order-type as $K$, which is $\leqslant \alpha$ by \cref{prop:relatewellorder}(\ref{item:deflat}).
\end{proof}

Let $\lambda$ be an extended limit.  The cofinal and strictly cofinal
subclasses of $[0\twodots \lambda)$ are
the same (as stated above), and the order-type of each is an extended limit.   The least
such order-type is called the \emph{cofinality} of $\lambda$, and written
$\cfof{\lambda}$.  Clearly it satisfies $\cfof{\lambda} \leqslant
\lambda$ 
and $\cfof{\cfof{\lambda}} =
\cfof{\lambda}$.

Now we use cofinality to characterize \supclosed{}ness and \supcomplete{}ness{}.
\begin{proposition} \label{prop:kclosed} Let $\lambda$ be an extended limit.  It is
  \begin{enumerate}
  \item \label{item:kclosed} $K$-\supclosed{}, for a set $K$, 
    iff $\alestar(K) \leqslant \cfof{\lambda}$.
  \item     $\catk$-\supcomplete{}, for a class of sets $\catk$, iff $\supo_{K \in \catk} \alestar(K)
    \leqslant \cfof{\lambda}$.
   \item \label{item:fclosed}  $F$-\supclosed{}, for a function $F \ccolc \Ord \rightarrow
     \allsets$, 
     iff $\cfof{\lambda}$
     is $(\alpha \mapsto \alestar (F(\alpha)))$-based.
   \item \label{item:hcomplete} $H$-\supcomplete{}, for
     a function  $H \ccolc \Ord \rightarrow
     \subclassof{\allsets}$,  iff $\cfof{\lambda}$
     is $(\alpha \mapsto \supo_{K \in H(\alpha)} \alestar (K))$-based.
  \end{enumerate}
\end{proposition}
\begin{proof} 
  We prove part~(\ref{item:kclosed}), from which the other parts follow.

  For ($\Rightarrow$), take a  cofinal subclass $B$ of
  $[0 \twodots \lambda)$ with order-type $\cfof{\lambda}$.  If  $\cfof{\lambda} <
  \alestar(K)$, then we have a surjection $K \rightarrow
 [0 \twodots \cfof{\lambda})$, and hence $K \rightarrow B$, so $\lambda$ is
  $B$-\supclosed{}, so $\bigvee_{\beta < B} \beta < \lambda$, a contradiction.  

  For ($\Leftarrow$), the range of any $f \ccolc K \rightarrow
    [0 \twodots \lambda)$ has order-type  in $\alestar(K)$ and
    hence in $\cfof{\lambda}$, so its supremum is $< \lambda$.   
\end{proof}

Where the sets in questions are ordinals, we give a simpler characterization:
\begin{proposition} \label{prop:alphaclosed}
  Let $\lambda$ be an extended limit.  It is
  \begin{enumerate}
  \item \label{item:alphaclosed} 
    $\alpha$-\supclosed{}, for an ordinal $\alpha$, iff $\alpha < \cfof{\lambda}$.
  \item \label{item:alphacompl} $\rho$-\supcomplete{}, for an extended
    ordinal $\rho$, iff $\rho
    \leqslant \cfof{\lambda}$.
   \item \label{item:ordfunclosed} $F$-\supclosed{}, for a function $F \ccolc \Ord \rightarrow \Ord$,  iff $\cfof{\lambda}$ is $F$-closed.
    \item \label{item:ordfuncompl} $G$-\supcomplete{}, for a function $G \ccolc \Ord \rightarrow
      \extord$, iff $\cfof{\lambda}$ is $G$-based.
  \end{enumerate}
\end{proposition}
\begin{proof} Part~(\ref{item:alphaclosed}) is by \cref{prop:subsetcof}, and
  the rest follows.
\end{proof}

\psubsection{Regularity} \label{sect:regularity}

We revisit the notion of regularity from
\Cref{sect:regstat}.  First we note that 
\cref{prop:alphaclosed}(\ref{item:alphacompl})  gives the following:
\begin{corollary} \label{cor:reg}
  Let $\lambda$ be an extended limit.  It is $\lambda$-\supcomplete{} iff $\cfof{\lambda} = \lambda$.
\end{corollary}
We say that $\lambda$ is \emph{regular} when it satisfies these
conditions.  Thus the cofinality of any extended limit is regular.
Here are more ways of obtaining examples:
\begin{proposition} \label{prop:genisreg}
  All of the following are regular.
  \begin{itemize}
  \item The extended limit that is simply \supgenerat{}ed by a set $K$.
  \item The extended limit that is \supgenerat{}ed by a class of sets $\catk$.
  \item The extended limit that is simply \supgenerat{}ed by a function  $F \ccolc
  \Ord \rightarrow \allsets$.
  \item The extended limit that is \supgenerat{}ed a function $H
   \ccolc \Ord \rightarrow \subclassof{\allsets}$.
  \end{itemize}
\end{proposition}
\begin{proof}
We first prove that, for a class of sets $\catk$, a minimal
$\catk$-\supcomplete{} extended limit $\lambda$ is regular.  Fix $\alpha <
  \lambda$ and a tuple $[x_{i}]_{i < \alpha}$ within $\lambda$.  We 
  shall show that $\bigvee_{i < \alpha} x_i < \lambda$.
  For $\beta < \lambda$, put $\upto{\beta} \eqdef \bigvee_{i <
  \alpha \wedge \beta} x_i$.  The class $P \eqdef \setbr{\beta <
\lambda \mid \upto{\beta} < \lambda}$ is a $\catk$-\supcomplete{}
 limit 
$\leqslant \lambda$ by the following reasoning.
\begin{itemize}
\item For $\gamma \leqslant \beta \smin P$ we have $\gamma \in P$,
  since $\upto{\gamma} \leqslant \upto{\beta}$.  So $P$ is an 
  ordinal
  $\leqslant \lambda$.
\item We have $0 \in P$ since $\upto{0} = 0$, and
  for any $\beta \smin P$ we have $\succof{\beta} \in P$, since
  $\upto{\succof{\beta}}$ is $\upto{\beta} \wedge x_{\beta}$ if
  $\beta < \alpha$ and $\upto{\beta}$ otherwise So $P$ is a limit.
\item  For any $K \smin \catk$ and tuple $[\beta_k]_{k \in K}$ within $P$, we
  have $\bigvee_{k \in K} \beta_k \in P$ since
  $\upto{\bigvee_{k \in K} \beta_k} = \bigvee_{k \in K}
  \upto{\beta_k}$.  So $P$ is $\catk$-\supcomplete{}.
\end{itemize}
Minimality of $\lambda$ gives $P =
\lambda$, so $\alpha \in P$,
meaning that $\bigvee_{i < \alpha} x_i = \upto{\alpha} < \lambda$ as required.  

Lastly, for $H \ccolc \Ord \rightarrow \subclassof{\allsets}$, any minimal
$H$-\supcomplete{} limit $\lambda$ is also a minimal $(\bigcup_{\beta < \lambda}
H\beta)$-\supcomplete{} limit, and therefore regular.
\end{proof}

\psubsection{Blass's Axiom and Mahlo's Principle} \label{sect:blassmahlo}

Now we revisit the principles from \Cref{sect:twoprinc}: \emph{Blass's axiom} says that $\regord$ is unbounded,
and \emph{Mahlo's principle}  that $\regord$ is stationary.  We begin
with 
basic consequences.
\begin{proposition} \label{prop:blassordlind} \hfill
  \begin{enumerate}
  \item\label{item:blassordlind} Blass's axiom implies Ordinal
    Lindenbaum $+$ Infinity.
  \item \label{item:mahloblass} Mahlo's principle implies Blass's axiom.
  \end{enumerate}
\end{proposition}
\begin{proof}  \hfill
  \begin{enumerate}
  \item  \label{item:blassordlindproof} Given an ordinal $\alpha$, take a regular limit ordinal 
    $\lambda > \alpha$.  It is is $\alpha$-\supclosed{} by
    \cref{prop:alphaclosed}(\ref{item:alphaclosed}).  We
    show $\alestar(\alpha) \leqslant \lambda$.  For any
    $\beta < \alestar(\alpha)$, we have a partial surjection
    $\alpha \rightarrow \beta$, so $\lambda$ is $\beta$-\supclosed{},
    so $\beta < \lambda$ by
    \cref{prop:alphaclosed}(\ref{item:alphaclosed}).   Infinity holds since $\regord \subseteq \limclass$.  
  \item \label{item:mahloblassproof} Since stationary implies unbounded. \qedhere
 \end{enumerate}
\end{proof}
We arrive at the main result of the section:
\begin{theorem} \label{prop:equivbound} \hfill
  \begin{enumerate}
  \item \label{item:equivboundwide} Wide \Supgeneration{} is equivalent to Blass's axiom $+$
    Relative Lindenbaum.
  \item \label{item:equivboundbroad} Broad \Supgeneration{} is equivalent to Mahlo's principle $+$
    Relative Lindenbaum.
  \end{enumerate}
\end{theorem}
\begin{proof}
  We prove part~(\ref{item:equivboundbroad}), as part~(\ref{item:equivboundwide}) is similar.

  For ($\Rightarrow$), Broad \Supgeneration{} implies Relative Lindenbaum by
  \cref{prop:implales}(\ref{item:widelimales}).  To show that it implies Mahlo's principle,
  let $F \ccolc \Ord \Rightarrow \Ord$.  The simply $F$ \supgenerat{}ed limit ordinal is $F$-\supclosed{}, so by
  \cref{prop:alphaclosed}(\ref{item:ordfunclosed}) it is
  $F$-closed.  It is regular
    by \cref{prop:genisreg}.

    For ($\Leftarrow$), let $F \ccolc \Ord \rightarrow \allsets$.  Define $G \ccolc
     \Ord \rightarrow \Ord$ sending $\alpha$ to the least ordinal
     $\beta$ such that $\alestar(F(\alpha)) \leqslant
     \alestar(\beta)$. Then there is a regular limit ordinal that is
     $G$-closed.  By
     \cref{prop:alphaclosed}(\ref{item:ordfunclosed})  it
     is 
     $G$-\supclosed{}, so by
     \cref{prop:kclosed}(\ref{item:fclosed})  it is
     $F$-\supclosed{}. 
\end{proof}

\begin{corollary} \assumi{Powerset or Well-orderability} \label{cor:psetequiv}
   \begin{enumerate}
  \item \label{item:supblass} Wide \Supgeneration{} is equivalent to Blass's axiom.
  \item  \label{item:supmahlo} Broad \Supgeneration{} is equivalent to Mahlo's principle.
  \end{enumerate}
\end{corollary}
\begin{proof}
  Immediate from  \cref{prop:equivbound}, using
  \cref{prop:implales}(\ref{item:psetales}) and~\ref{prop:woprop}(\ref{item:worl}).
\end{proof}

\psection{The Power of Stationarity} \label{sect:powstat}
\psubsection{Club Classes and Continuous Functions}  \label{sect:clubcont}


Our final task is to develop the traditional theory of stationarity,
in which ``iterated inaccessibles'' of various kinds are obtained from Mahlo's
principle.  Throughout \cref{sect:powstat}, \emph{class} will always mean a class of ordinals, and 
\emph{function} an endofunction on $\Ord$.  We use the following constructions:
\begin{definition} \label{def:prefof} \hfill
  \begin{enumerate}
  \item \label{item:prefof} For any monotone function $H$, we write
    $\prefof{H}$ for the class of all its prefixpoints.
  \item \label{item:intersect} For any family of classes $(C_i)_{i \in I}$, the \emph{intersection} is given by
    \begin{eqnarray*}
      \bigcap_{i \in I} C_i & \eqdef & \setbr{\alpha \smin \Ord \mid
        \forall i \smin I.\,  \alpha \in C_i}
    \end{eqnarray*}
   \item \label{item:supfunc} For any family of monotone functions $(H_i)_{i \in I}$,  the \emph{supremum} is given by
     \begin{eqnarray*}
       \bigvee_{i \in I}H_i \quad \ccolc \quad \alpha & \mapsto & \bigvee_{i
                                                          \in I}
                                                          H_i(\alpha)
       \\
       \text{so that }  \quad \prefof{ \bigvee_{i \in I}H_i}  & = & \bigcap_{i
                                                            \in I} \prefof{H_i}
     \end{eqnarray*}
   \item \label{item:diagint} For any sequence of classes $(C_{\alpha})_{\alpha \in \Ord}$,
     the \emph{diagonal intersection} is given by
    \begin{eqnarray*}
      \Delta_{\alpha \in \Ord} C_{\alpha} & \eqdef & \setbr{\alpha
                                                      \smin \Ord \mid
                                                      \forall \beta <
                                                      \alpha.\, \alpha
                                                      \in C_{\beta}}
    \end{eqnarray*}
    \item \label{item:diagsup} For any sequence of monotone functions $(H_{\alpha})_{\alpha \in
        \Ord}$, the  \emph{diagonal supremum}  is given by 
      \begin{eqnarray*}
        \nabla_{\alpha \in \Ord} H_{\alpha} \quad \ccolc \quad \alpha &
                                                                \mapsto
        & \bigvee_{\beta < \alpha} H(\beta) \\
        \text{so that } \quad \prefof{ \nabla_{\alpha \in \Ord}
                          H_{\alpha}} & = & \Delta_{\alpha \in \Ord} \prefof{H_{\alpha}}
      \end{eqnarray*}
  \end{enumerate}
\end{definition}


\begin{definition} \label{def:cont}
   A function $H$ is
    \emph{continuous} when it is monotone and sends every limit
    ordinal 
    $\lambda$ to  $\bigvee_{\alpha < \lambda} H(\alpha)$.
\end{definition}
  Here are some ways to obtain continuous functions:
\begin{proposition} \label{prop:contfun} \hfill
    \begin{enumerate}
    \item \label{item:constcont} For every ordinal $\alpha$, the function $\constalpha$ is continuous.
   \item \label{item:supcont} For any family of continuous functions $(H_i)_{i \in I}$, the
    supremum 
     $\bigvee_{i \in I} H_i$ is continuous.
   \item \label{item:diagsupcont} For any sequence of continuous functions $(H_{\alpha})_{\alpha \in
        \Ord}$, the diagonal supremum $\nabla_{\alpha \in \Ord}
      H_{\alpha} $ is continuous.
    \end{enumerate}
  \end{proposition}
  \begin{proof} \hfill
    \begin{enumerate}
    \item \label{item:constcontproof} Obvious.
    \item \label{item:supcontproof} Straightforward.
    \item \label{item:diagsupcontproof} Monotonicity is obvious.  For continuity, let $\lambda$ be  a
   limit ordinal. Then
   \begin{eqnarray*}
     (\nabla_{\alpha \in \Ord} H_{\alpha})(\lambda) & = & \bigvee_{\alpha <
                                              \lambda}
                                              H_{\alpha}(\lambda) \\
     & = & \bigvee_{\alpha < \lambda} \bigvee_{\beta < \lambda}
           H_{\alpha}(\beta) \\
     & \leqslant & \bigvee_{\alpha < \lambda} \bigvee_{\beta <
                   \lambda}  H_{\alpha}(\beta \vee \mathsf{S}\alpha) \\
     & \leqslant & \bigvee_{\gamma  <\lambda} \bigvee_{\alpha < \gamma}
                   H_{\alpha}(\gamma) \\
     & = & \bigvee_{\gamma < \lambda} (\nabla_{\alpha \in \Ord}H_{\alpha})(\gamma)
   \end{eqnarray*}  \qedhere
    \end{enumerate} 
  \end{proof}

  \begin{definition} \label{def:closedord}
    Let $C$ be a class.
    \begin{enumerate}
    \item \label{item:limpt} A \emph{limit point} of $C$ is a limit ordinal $\lambda$
      such that $\lambda \cap C$ is unbounded in $\lambda$.  The class
      of all such is written $\limsof{C}$.
    \item \label{item:closedclass} A class $C$ is \emph{closed} when it is  $\limso$-prefixed,
      i.e., contains every limit ordinal $\lambda$ such
  that $\lambda \cap C$ is unbounded in $\lambda$.
    \end{enumerate}
  \end{definition}
     Here are some ways to obtain closed classes:
  \begin{proposition} \label{prop:closedeg} \hfill
    \begin{enumerate}
    \item \label{item:prefcl} For any continuous function $H$,
      the class $\prefof{H}$ is closed.
   \item \label{item:limcl} For any class $C$, the class $\limsof{C}$ is 
     closed.
   \item \label{item:intcl} For a family of closed classes  $(C_i)_{i \in I}$, the
     intersection $ \bigcap_{i \in I} C_i $ is closed.
   \item \label{item:diagcl} For a sequence of closed classes  $(C_{\alpha})_{\alpha \in
        \Ord}$, the diagonal intersection $\Delta_{\alpha \in \Ord} C_{\alpha}$ is closed.
    \end{enumerate}
  \end{proposition}
  \begin{proof} \hfill
    \begin{enumerate}
    \item \label{item:prefclproof} We must show that a limit point $\lambda$ of $\prefof{H}$ is
      in $\prefof{H}$. For any
    $\gamma < \lambda$, there is $\beta \in [\gamma \twodots \lambda) \cap \prefof{H}$, giving $H(\gamma) \leqslant H(\beta) \leqslant
    \beta < \lambda$.  So we have
    \begin{math}
      H(\lambda)  =  \bigvee_{\gamma < \lambda} H(\gamma) 
       \leqslant   \lambda 
    \end{math}
    as required.
    \item \label{item:limclproof} We show that a limit point $\lambda$ of $\limsof{C}$ is
      a limit point of $C$.   For any
    $\alpha < \lambda$, there is $\beta \smin (\alpha \twodots \lambda) \cap \limsof{C}$.  Since $\beta \in \limsof{C}$, there
    is $\gamma  \in [\alpha \twodots \beta) \cap C$.  Thus we have
    $\gamma \in [\alpha \twodots \lambda) \cap C$ as required.

  \item \label{item:intclproof} Since an infimum of prefixpoints is a prefixpoint.
    \item\label{item:diagclproof}  It suffices to show that $\limsof{\Delta_{\alpha \in \Ord}
        C_{\alpha}} \subseteq \Delta_{\alpha \in \Ord}
      \limsof{C_{\alpha}}$.  This means that any limit point $\lambda$ of $\Delta_{\alpha \in \Ord}
        C_{\alpha}$  is, for all $\alpha < \lambda$,  a limit point
      of $C_{\alpha}$.    For any $\beta \smin (\alpha \twodots
      \lambda)$, there is $\gamma \in [\beta \twodots \lambda) \cap
      \Delta_{\alpha \in \Ord} C_{\alpha}$.  Since $\alpha < 
      \beta \leqslant \gamma$, we have $\gamma \in C_{\alpha}$ as
      required.  \qedhere
    \end{enumerate}
  \end{proof}

  Next we consider unbounded classes.  
  \begin{definition} \hfill \label{def:closure}
    \begin{enumerate}
    \item \label{item:closure} A \emph{closure operator} (on $\Ord$) is a function that is
      inflationary and idempotent.
    \item \label{item:unbclose} For an unbounded class $C$, we write $H_C$ for the unique
      closure operator whose range is $C$.  Explicitly, it sends
      $\alpha$ to the least element of $C$ that is $\geqslant \alpha$.
    \end{enumerate}
  \end{definition}
  Thus we have a bijection between the collection of all unbounded
  classes and that of all closure operators. Moreover it is  \emph{dual}---i.e., for unbounded classes $C$ and $D$, we have
  $H_C \leqslant H_D$ iff $D \subseteq C$.
\begin{proposition} \label{prop:closedcont}
  For any unbounded class $C$, the following are equivalent:
  \begin{itemize}
  \item $C$ is closed.
  \item $H_C$ is continuous.
  \end{itemize}
  \end{proposition}
  \begin{proof}
   Firstly, $H_C$ is continuous at every  $\lambda \not\smin \limsof{C}$, since there is $\alpha < \lambda$
    such that $[\alpha \twodots \lambda)$ has no element in $C$, so $H_C$
sends every ordinal in this interval to $H_{C}(\lambda)$.  As for $\lambda \smin \limsof{C}$, we have $\bigvee_{\gamma
      < \lambda} H_C (\gamma) = \lambda$, so 
    $H_C$ is continuous at $\lambda$ iff $\lambda \in C$.
\end{proof}
Thus we have a dual bijection between the collection of all closed unbounded
classes, known as \emph{club classes}, and that of all continuous closure
operators.  
Here are some ways to obtain club classes:
  \begin{proposition} \label{prop:obclub}
    Each of the following is equivalent to Infinity.
    \begin{enumerate}
    \item \label{item:prefclub} For any continuous function $H$,
      the class $\prefof{H}$ is club.
    \item \label{item:limclub} For any unbounded class $C$, the class $\limsof{C}$ is
      club.
   \item \label{item:intclub} For a family of club classes  $(C_i)_{i \in I}$, the
     intersection $ \bigcap_{i \in I} C_i $ is club.
   \item \label{item:diagclub} For a sequence of club classes  $(C_{\alpha})_{\alpha \in
        \Ord}$, the diagonal intersection $\Delta_{\alpha \in \Ord} C_{\alpha}$ is club.
    \end{enumerate}
  \end{proposition}
  \begin{proof}
    Firstly, if Infinity does not hold, then a club class is just
    an unbounded class of natural numbers, and the statements are all false:
    \begin{enumerate}
    \item \label{item:prefclubrefute} $\prefof{\mathsf{S}}$ is empty.
    \item  \label{item:limclubrefute} $\limsof{\Ord}$ is empty.
    \item  \label{item:intclubrefute} Let $C$ be the class of all even numbers and $D$ that of all
      odd numbers.  Each is club, but  $C \cap D$ is empty.
    \item  \label{item:diagclubrefute} For each ordinal $n$, let $C_n$ be the class of ordinals
      $\geqslant n+2$.  Each of these is club, but $\Delta_{n \in \Ord} C_n = \setbr{0}$.
    \end{enumerate}
    Now assume Infinity.  Because of
    \cref{prop:closedeg}, we need only prove unboundedness.
    \begin{enumerate}
    \item \label{item:prefclubproof} Let $\alpha$ be an ordinal.  Form a strictly increasing sequence of ordinals $(x_n)_{n \in
      \nats}$ via $x_0 \eqdef \alpha$ and $x_{n+1} \eqdef \mathsf{S}(x_n) \vee
    H(x_{n})$.  Its supremum $\lambda$ is a limit ordinal and satisfies
    \begin{eqnarray*}
      H (\lambda) & = & \bigvee_{\gamma < \lambda} H(\beta) \\
                  & = & \bigvee_{n \in \nats} H(x_n) \\
      & \leqslant & \bigvee_{n \in \nats} x_{n+1} \text{ (since
                    $H(x_n) \leqslant x_{n+1}$)} \\
      & = & \lambda
    \end{eqnarray*}
    so $\lambda$ is an $H$-prefixpoint $\geqslant \alpha$.
    \item\label{item:limclubproof}  Let $\alpha$ be an ordinal.  Form a strictly increasing sequence of ordinals $(x_n)_{n \in
      \nats}$ via $x_0 \eqdef \alpha$ and $x_{n+1} \eqdef $ the least
    ordinal in $C$ that is greater than $x_{n}$.  Then 
    $\bigvee_{n \in \nats}x_n$ is 
    in $\limsof{C}$ and is $> \alpha$.
  \item \label{item:intclubproof} Since
    \begin{eqnarray*}
       \bigcap_{i \in I} C_i & = & \bigcap_{i \in I} \prefof{H_{C_i}}
      \\
      & = & \prefof{\bigvee_{i \in I} H_{C_i}}
    \end{eqnarray*}
    which is club by 
    part~(\ref{item:prefclub}), using
    \cref{prop:contfun}(\ref{item:supcont}) .
  \item\label{item:diagclubproof}  Similar.  \qedhere
  \end{enumerate}
  \end{proof}
  

  \begin{proposition} \label{prop:equivstat}
    For a class of limit ordinals $D$, the following are equivalent.
    \begin{enumerate}
    \item \label{item:stat} $D$ is stationary.
      \item \label{item:statcont} Every continuous function has a prefixpoint in
      $D$.
    \item  \label{item:statclub} Every club class has an element in $D$. 
   \end{enumerate}
  \end{proposition}
  \begin{proof}
To show that~(\ref{item:stat}) implies~(\ref{item:statcont}), let  $H$ be a continuous function.  Then any
$H$-based limit ordinal is $H$-prefixed.  To show the converse, let $G$
be a function.  Then the function $H \eqdef \nabla_{\alpha \in \Ord}
\mathsf{const}_{G(\alpha)}$ is continuous by
\cref{prop:contfun}, and
a limit ordinal is $G$-based iff it is $H$-prefixed.

To show that~(\ref{item:statcont}) implies~(\ref{item:statclub}): for any
club class $C$, we have $C = \prefof{H_C}$ and $H_C$ is continuous.  To show the converse, we note
that the club class $\Ord$ has an element in $D \subseteq \limclass$.  
So Infinity holds and we can apply \cref{prop:obclub}(\ref{item:prefclub}).
  \end{proof}

  \begin{corollary} \label{cor:statclub}
    Let $D$ be a stationary class of limit ordinals.  The intersection of $D$
    with any club class is stationary.
  \end{corollary}

  \psubsection{Application: Iterated Inaccessibility} \label{sect:iterinacc}
  
In order to formulate iterated inaccessibility, we use the following result.
\begin{proposition} \label{prop:iterstat}
  Let $D$ be a class of limit ordinals.
  \begin{enumerate}
  \item \label{item:iterdef} There is a sequence of classes $(X_{\alpha})_{\alpha \in
      \Ord}$ and class $X_{\subbigginfty}$ uniquely specified by
    \begin{eqnarray*}
      X_{\alpha} & = & D \cap \bigcap_{\beta < \alpha} \limsof{X_{\beta}} \\
      X_{\subbigginfty} & = & D \cap \Delta_{\beta \in \Ord} \limsof{X_{\beta}}
    \end{eqnarray*}
 \item \label{item:iterstat} If $D$ is stationary, then so are all these classes.
    \end{enumerate}
  \end{proposition}
  \begin{proof} \hfill
    \begin{enumerate}
    \item  \label{item:iterdefproof} We cannot define a sequence of
      classes recursively, so we proceed as follows.  For any ordinal $\rho$, we recursively define a sequence 
      $(X^{\rho}_{\alpha})_{\beta \in \Ord}$ of subsets of $\rho$ via
      \begin{eqnarray*}
         X^{\rho}_{\alpha} & = & \rho \cap D \cap \bigcap_{\beta < \alpha} \limsof{X^{\rho}_{\beta}} 
      \end{eqnarray*}
      These sequences are compatible in the sense that, for $\rho \leqslant \sigma$, we have $X^{\rho}_{\alpha} =
      \rho \cap X^{\sigma}_{\alpha}$.  We define
      \begin{eqnarray*}
        X_{\alpha} & \eqdef & \setbr{\rho \in \Ord \mid \rho \smin X^{\mathsf{S}\rho}_{\alpha}}
      \end{eqnarray*}
 and obtain the
required properties.
\item \label{item:iterstatproof} Firstly, since $D$ is stationary, Infinity holds and we can use \cref{prop:obclub}.

  We cannot simply prove $X_{\alpha}$ stationary by induction
  on $\alpha$, as stationarity involves second-order
  quantification.  
  Instead, we prove {unboundedness} of $X_{\alpha}$ by induction
  on $\alpha$, as follows.  For all $\beta < \alpha$, the class
  $X_{\beta}$ is unbounded, so $\limsof{X_{\beta}}$ is club.  So
  $\bigcap_{\beta < \alpha} X_{\beta}$ is club, making $X_{\alpha}$ 
  stationary and hence unbounded.  This completes the induction.

  Next, for any ordinal $\alpha$, we  see (again) that $\bigcap_{\beta <
    \alpha} \limsof{X_{\beta}}$ is club, making $X_{\alpha}$ stationary.  Likewise
  $\nabla_{\beta \in \Ord} \limsof{X_{\beta}}$ is club, making
  $X_{\subbigginfty}$ stationary.   \qedhere
    \end{enumerate}
  \end{proof}


For an application {\bf assuming Powerset $+$ Infinity $+$ AC}, let
$D$ be the class of all inaccessible cardinals.
Then $X_{\alpha}$ is the class of all \emph{$\alpha$-inaccessible} cardinals,
and $X_{\subbigginfty}$ the class of all \emph{hyper-inaccessible} cardinals.
  \cref{prop:iterstat}(\ref{item:iterstat}) tells us that
  these classes are stationary if Broad Infinity holds.  

This construction can be further iterated, giving
hyper-hyper-inaccessible cardinals
and more.  In Carmody's work~\cite{Carmody:killsoftly}, this is
achieved by generalizing the 
 subscripts used in \cref{prop:iterstat} to a system of
``meta-ordinals''.

\paragraph*{Related work} \quad A standard treatment of stationarity
 is
given in~\cite{Jech:settheory}, not for classes
but for 
{subsets of a given limit ordinal}.  So the predicativity issue does not
arise, 
and additional results are obtained, such as Fodor's
pressing-down lemma and Solovay's partitioning theorem.  See~\cite{GitmanHamkinsKaragila:classfodor}
for an analysis of whether Fodor's lemma applies to classes. 

\ppart{Wrapping up} \label{part:wrap}

\psection{Conclusions}  \label{sect:conc}

\psubsection{Summary of Achievements} \label{sect:sumachieve}

We have now established all the relationships in
Figure~\ref{fig:diagsub}.  The main technical achievement was proving the
equivalence (assuming Powerset $+$ AC)  of Simple Broad Infinity and Mahlo's
principle.  The centrepiece is the implication Full
Broad Infinity $\Rightarrow$ Broad Derivation Set, which relies on \cref{prop:genpf} to generate the
$\catr$-derivational class-family.

On the philosophical side, I claim that the notion of $F$-broad
number (for a broad arity $F$) is easily grasped, making Simple Broad
Infinity a plausible axiom scheme.   This is 
for the reader to judge.

On the practical side, we have seen several equivalent principles  
that are convenient for applications.  Specifically:
\begin{itemize}
\item Broad Derivation Set yields the existence of Tarski-style
  universes.
\item Broad \setgen{}  yields the existence of Grothendieck
  universes.
\item Mahlo's principle in the form ``Every club class contains a
  regular limit ordinal'' yields the existence of
  $\alpha$-inacessibles and hyper-inaccessibles.
\end{itemize}
As promised in \Cref{sect:summgoals}, we have developed our results in a setting that allows urelements and
  non-\wellfound{} membership,  proved the sufficiency of (a version of) WISC
  for our main AC-reliant results,  and seen the pattern of resemblance
  between Wide and Broad principles throughout the paper.

  \psubsection{Further Work}  \label{sect:future}
  
Beyond the above contributions, more work remains to be done.  Firstly, there are
unanswered questions, particularly 
about the power of Broad ZF.
\begin{enumerate}
\item \label{item:gitikbroad} By analogy with Gitik's work~\cite{gitik:uncsing}, can it be shown, under some
  consistency hypothesis for large cardinals, that Broad ZF does not prove
  the existence of an uncountable regular limit ordinal?
\item  \label{item:broadblass} Does Broad
   ZF $+$ Blass's axiom prove Mahlo's principle?
\item \label{item:heredbroad} Jech~\cite{Jech:heredcount} showed in ZF that the class of all hereditarily
  countable sets is a set, and his result has been extended to other
  cardinalities~\cite{Diener:transhull,Holmes:heredsmall}.  
   Can a
  stronger version be proved in Broad ZF?  For example, given a
  broad arity $F \ccolc \totall \rightarrow
  \allsets$, let $H(F)$ denote the least class $X$ that contains $\rStart$ and, for any  $x \smin X$ and
   $y \smin X^{F(x)}$, contains
 $\rBuild{x}{\rangeof{y}}$.  This exists by
 \cref{prop:gensub}(\ref{item:gensub}).  Does Broad ZF prove, for
 every broad arity $F$, that $H(F)$ is a
 set?
\end{enumerate}
Everything in this paper has been done in a base theory that---like
ZF---uses classical first-order logic and ignores logical complexity.   But some other versions of set theory 
 use intuitionistic logic and/or restrict the use of logically
complex sentences~\cite{Crosilla:czfizf,Mathias:strengthmac}.  The task of adapting our results to such theories (as far as
possible)  is left to future work.

Lastly, the link between type-theoretic work on
induction-recursion~\cite{DybjerSetzer:indindrec,GhaniHancock:contmonir,Ghanietal:varyir} and the principles in this paper remains to be developed.

\newcommand{\etalchar}[1]{$^{#1}$}
\providecommand{\noopsort}[1]{}



\begin{acks} I thank Asaf Karagila for teaching me
  about Lindenbaum numbers.  I also thank Tom de Jong, for persuading me to
 devote this paper to classical set theory, and to postpone  
 intuitionistic and constructive set theory to future work.
\end{acks}

\end{document}